\documentclass[twoside, leqno, 11pt]{article}

\usepackage[top=30truemm,bottom=30truemm,left=25truemm,right=25truemm]{geometry}

\usepackage{amsmath, amssymb}%To use complexity symbols.
\usepackage{amsthm}%To construct theorem environment.
\usepackage{amscd}%To write figures commutative diagram like that.
\usepackage[all]{xy}
\usepackage[dvipdfmx]{graphicx}
\usepackage{tikz}
\usepackage{comment}

\theoremstyle{definition}%Numbering of theorems and others.

\newtheorem{lemma}{Lemma}[section]
\newtheorem{definition}[lemma]{Definition}
\newtheorem{theorem}[lemma]{Theorem}
\newtheorem{example}[lemma]{Example}
\newtheorem{proposition}[lemma]{Proposition}
\newtheorem{corollary}[lemma]{Corollary}
\newtheorem{notation}[lemma]{Notation}
\newtheorem{remark}[lemma]{Remark}
\newtheorem{conjecture}[lemma]{Conjecture}

\def\address#1#2{\begingroup
\noindent\parbox[t]{7.8cm}{%
\small{\scshape\ignorespaces#1}\par\vskip1ex
\noindent\small{\itshape E-mail address}%
\/: #2\par\vskip4ex}\hfill%
\endgroup}%

\title{\uppercase{\bf On the boundary components of central streams
in the two slopes case}} %title of the paper
\author{
\\
\textsc{Nobuhiro Higuchi} %names of authors
}
\date{} %leave empty

\begin{document}

\maketitle

\begin{abstract}
In 2004 Oort studied the foliation
on the space of $p$-divisible groups.
In his theory, special leaves called central streams play an important role.
It is still meaningful to investigate central streams, for example,
there remain a lot of unknown things  
on the boundaries of central streams.
In this paper, we classify
the boundary components of the central stream
for a Newton polygon consisting of two segments,
where one slope is less than $1/2$
and the other slope is greater than $1/2$.
Moreover we determine the generic Newton polygon
of each boundary component using this classification.
\end{abstract}

\footnote[0]{2010 Mathematics Subject Classification : Primary:14L15 Group schemes; 
Secondary:14L05 formal groups,
$p$-divisible group; 14K10 algebraic moduli, classification.}
\footnote[0]{ %key words and phrases
\textit{Key words and phrases}.
$p$-divisible group; deformation space; Newton polygons.
}

\section{Introduction} 
\label{Intro}

In \cite[p.\,1023]{oortMinimal},
Oort introduced the notion of 
minimal $p$-divisible groups,
where $p$-divisible groups are often called Barsotti-Tate groups.
Let $k$ be an algebraically closed field of characteristic $p$.
%% 上の文章の位置を変更
Oort showed in \cite[1.2]{oortMinimal} that they have the following special property:
Let $X$ be a minimal $p$-divisible group over $k$.
Let $Y$ be a $p$-divisible group over $k$.
If $X[p]\simeq Y[p]$, then $X\simeq Y$,
where $X[p]$ is the kernel of $p$-multiplication $p : X \rightarrow X$.
For a Newton polygon $\xi$, we have a minimal $p$-divisible group $H(\xi)$.
See Section~\ref{p-divAndDieudonne} \eqref{DefOfMinimalp-div}
for the definition of $H(\xi)$.
A minimal $p$-divisible group is a $p$-divisible group which
is isomorphic to $H(\xi)$
for some $\xi$ over an algebraically closed field.

Let $X_0$ be a $p$-divisible group over $k$.
Let ${\rm Def}(X_0) = {\rm Spf} (\Gamma)$ be the deformation space 
of $X_0$.
Here,
the deformation space is 
the formal scheme pro-representing the functor ${\rm Art}_k \to {\rm Set}$ 
sending $R$ to the set of isomorphism classes of $p$-divisible groups $X$ over $R$
satisfying $X_k \cong X_0$,
where ${\rm Art}_k$ is the category of local Artinian rings with residue field $k$.
%% Artinian 大文字に
%The given $p$-divisible group over ${\rm Def}(X_0)$
%comes from a $p$-divisible group $\mathfrak X$ over 
%$\Delta = {\rm Spec}(\Gamma)$.
It was proved by de Jong in \cite[2.4.4]{deJongCrystalline} that
the category of $p$-divisible groups over ${\rm Spf}(\Gamma)$ is 
equivalent to the category of $p$-divisible groups over $\Delta := {\rm Spec}(\Gamma)$.
Let $\mathfrak X' \to {\rm Spf}(\Gamma)$ be the universal $p$-divisible groups,
and let $\mathfrak X$ be the $p$-divisible groups over $\Delta$
obtained from $\mathfrak X'$ by the equivalence above.

In \cite[2.1]{oortFoliations}, for a $p$-divisible group $Y$ over $k$
and a $k$-scheme ${\rm S}$,
%% and a $k$-scheme ${\rm S}$を追加
Oort introduced a locally closed subset 
%% a leaf から a closed subset に変えた
$\mathcal C_Y({\rm S})$ 
for a $p$-divisible group $\mathcal Y$ over ${\rm S}$
characterized by $s \in \mathcal C_Y({\rm S})$ if and only if
$\mathcal Y_s$ is isomorphic to $Y$ over an algebraically closed field containing 
$k(s)$ and $k$.
He called $\mathcal C_Y({\rm S})$ 
the leaf associated with $Y$ in ${\rm S}$; 
see \eqref{DefOfLeaf} in Section~\ref{p-divAndDieudonne} for the details.
We are interested in the case that $(\mathcal Y, {\rm S}) = (\mathfrak X, \Delta)$.
In particular, if $Y$ is minimal, he called the leaf the {\it central stream}.
The notion of central streams is a ``central" tool in the theory of
foliations.

Let $X$ and $Y$ be $p$-divisible groups over $k$.
We say that $X$ appears as a {\it specialization} of $Y$
if there exists a family of $p$-divisible group $\mathfrak X \rightarrow {\rm Spec}(R)$
with discrete valuation ring $(R, \mathfrak m)$ in characteristic $p$ such that
$\mathfrak X_L$ is isomorphic to $Y$ 
over an algebraically closed field containing $L$ and $k$, and
$\mathfrak X_\kappa$ is isomorphic to $X$
over an algebraically closed field containing $\kappa$ and $k$,
where $L = {\rm frac}\, R$ is the field of fractions of $R$,
and $\kappa = R/\mathfrak m$ is the residue field of $R$.
We say that a specialization $X$ of $Y$ is {\it generic} if 
$\ell(X[p]) = \ell(Y[p]) - 1$ holds, where 
%for a $p$-divisible group $X$,
$\ell(X[p])$ is
the length of the element of the Weyl group corresponding to $X[p]$;
see the paragraph below Corollary~\ref{CoroOfXzeta}
for this correspondence between
elements of the Weyl group and $p$-kernels of $p$-divisible groups.

%We are interested in the boundaries of leaves and especially these of central streams.
%In this paper, we classify the boundary components of
%central streams associated with some special Newton polygons 
%and determine the Newton polygon of
%the generic point of every boundary component,
%i.e., it is dense in BC
%where ``special" means that the Newton polygon
%consists of two segments with slopes $\lambda$ and $\lambda'$ 
%satisfying $\lambda < 1/2 < \lambda'$.
%In Section~\ref{p-divAndDieudonne}, we recall the definition of Newton polygons.

Let $\xi$ and $\zeta$ be Newton polygons.
See the paragraph containing \eqref{DefOfMinimalp-div} for the definition of
Newton polygons.
We say $\zeta \prec \xi$ if each point of $\zeta$ is above or on $\xi$.
Moreover, we say that $\zeta \prec \xi$ is {\it saturated} if
%% \itを追加 
there exists no Newton polygon $\eta$ such that 
$\zeta \precneqq \eta \precneqq \xi$.

In this paper, we will give two main results.
The first result (Theorem~\ref{ThmOfCentralClassification})
classifies the boundary components of certain central streams. 
See Section~\ref{ClassBoundaryComp} 
for the statement as
we need some notation given in Section~\ref{ConSpeAbs}
to state the first result.
From this result, we expect that Conjecture~\ref{ConjOfArbitBC}
stated in Section~\ref{ClassBoundaryComp} is true.
This conjecture says that it suffices to deal with central streams
associated to Newton polygons consisting of two segments in order to classify 
boundary components of arbitrary central streams.
The second result is

\begin{theorem}\label{ThmOfp-divSpe}
Let $\xi$ be a Newton polygon consisting of two segments 
with slopes $\lambda$ and $\lambda'$ satisfying
$\lambda < 1/2 < \lambda'$.
Let $X$ be an arbitrary generic specialization of $H(\xi)$.
Then there exists a Newton polygon $\zeta$ satisfying that
\begin{itemize}
\item[(i)] $\zeta \prec \xi$ is saturated, and 
\item[(ii)] $H(\zeta)$ appears as a specialization of $X$.
\end{itemize}
\end{theorem}
%% 主張の書き方を変更

With the same notation as Theorem~\ref{ThmOfp-divSpe},
by Grothendieck-Katz \cite[2.3.1]{katzcrystal},
we have
$$\zeta \prec {\rm NP}(X) \precneqq \xi,$$
where ${\rm NP}(X)$ is the Newton polygon of $X$.
Here we note that ${\rm NP}(X) \neq \xi$ 
will be proved in Lemma~\ref{LemOfStrictlyMin} below.
%In Lemma~\ref{LemOfStrictlyMin} we shall show that ${\rm NP}(X)$
%is different from $\xi$.
Then the saturatedness of $\zeta \prec \xi$ implies ${\rm NP}(X) = \zeta$
and therefore
\begin{corollary} \label{CoroOfXzeta}
Let $\xi$ and $X$ be as in Theorem~\ref{ThmOfp-divSpe}.
Then ${\rm NP}(X) \prec \xi$ is saturated.
\end{corollary}
%% Coro.とかその上を追加

For a finite flat commutative group scheme $G$ over some base scheme,
we say that 
$G$ is a {\it truncated Barsotti-Tate group of level one} (abbreviated as ${\rm BT_1}$) 
%% イタリックに変更
if Frobenius $\mathcal F$ and Verschiebung $\mathcal V$ on $G$ satisfy
that ${\rm Ker}\, \mathcal F = {\rm Im}\, \mathcal V$ 
and ${\rm Ker}\, \mathcal V = {\rm Im}\, \mathcal F$. 
The dimension of ${\rm BT_1}$ $G$ is defined by
$\log_p {\rm rk}({\rm Ker}\, \mathcal F)$.
%% dim of Gを追加
Let $W$ be the Weyl group of 
the general linear group $GL_h$.
%% Let $W$ be a Weyl group. を Let $W$ be the Weyl group of $GL_h$. に変更
By Kraft \cite{kraftKom}, Oort \cite{oortstr} 
and Moonen-Wedhorn \cite{moonen-wedhornDis},
%% 上3名を追加
for each algebraically closed field $k$,
there exists a canonical one-to-one correspondence
$$\{{\rm BT_1}\text{'s over } k \text{ of height } h 
\text{ and dimension }d\}
\leftrightarrow {}^J W$$
for a subset ${}^J W$ of $W$ depending on $h$ and $d$;
see %Section~\ref{p-divAndDieudonne}, 
the paragraph below 
Example~\ref{ExOfDM1}
for the definition of ${}^J W$
and this correspondence.
%% and this correspondence., below Theorem~\ref{ThmOfDMAndSeq} を追加
Let $w$ and $w'$ be elements of ${}^J W$. 
We say $w' \subset w$
%% \precを\subsetに変更 
if there exists a discrete valuation ring $R$ of characteristic $p$
such that 
there exists a finite flat commutative group scheme $G \rightarrow {\rm Spec}(R)$
satisfying that $G_{\overline{\kappa}}$ is a ${\rm BT_1}$ of type $w'$,
and $G_{\overline{L}}$ is a ${\rm BT_1}$ of type $w$,
where $L$ is the fractional field of $R$, 
and $\kappa$ is the residue field $R/\mathfrak m$
with the maximal ideal $\mathfrak m$ of $R$.
We say that $w'$ is a {\it generic specialization} of $w$
%% \it追加
if $w' \subset w$ and $\ell(w') = \ell(w) - 1$ holds.

The reason to use the word ``generic'' comes from the following fact.
For a $p$-divisible group $X_0$ of type $w$,
i.e., for a $p$-divisible group such that its $p$-kernel 
corresponds to $w \in {}^J W$,
%% i.e., を追加
let $\mathcal S_w(\Delta)$ be the reduced subscheme of 
%% the subscheme of を the reduced subscheme of に変更
$\Delta = {\rm Spec}(\Gamma)$ consisting of $s$
satisfying that $\mathfrak X_s[p]$ is of type $w$.
Then it is known that $\dim \mathcal S_w(\Delta) = \ell(w)$
and $\mathcal S_w(\Delta)$ is non-empty;
see \cite[6.10]{wedhorn} and \cite[3.1.6]{moonen}.
This justifies the terminology ``generic'' for elements $w$ of ${}^J W$.
%% and $\mathcal S_w(\Delta)$ is non-empty を追加，referenceを追加
%% This justifies...追加

For a Newton polygon $\xi$,
let $w_\xi$ be the element of ${}^J W$ corresponding to $H(\xi)[p]$.
%% Section 2からSection2.1に変更
Using notation of the Weyl group, 
by \cite[4.1]{HarashitaSupremum},
Theorem~\ref{ThmOfp-divSpe} is paraphrased as

\begin{theorem} \label{ThmOfzetaxi}
Let $\xi$ be a Newton polygon 
consisting of two segments,
where one slope is less than $1/2$
and the other is larger than $1/2$.
For an arbitrary generic specialization $w \in {}^J W$
of $w_\xi$, 
there exists a Newton polygon $\zeta$ such that
\begin{itemize}
\item[(i)] $\zeta \prec \xi$ is saturated, and
\item[(ii)] $w_\zeta \subset w$.
\end{itemize}
\end{theorem}
%% with $\ell(w) = \ell(w_\xi) - 1$ を追加
%% CoroからThmに変更
The results of this paper lead us to expect that
Theorem~\ref{ThmOfzetaxi} can be generalized to
the case that $\xi$ is an arbitrary Newton polygon:
%% ここの文章修正

\begin{conjecture} \label{ConjOfArbitzeta}
For an arbitrary Newton polygon $\xi$, let $w \in {}^J W$
be a generic specialization of $w_\xi$.
Then there exists a Newton polygon $\zeta$ such that
\begin{itemize}
\item[(i)] $\zeta \prec \xi$ is saturated, and
\item[(ii)] $w_\zeta \subset w$.
\end{itemize}
\end{conjecture}
%% satisfying that $\ell(w) = \ell(w_\xi) - 1$ を追加
If Conjecture~\ref{ConjOfArbitBC} stated in Section~\ref{ClassBoundaryComp} holds, 
then for Newton polygons $\xi$,
%% the Newton polygonからNewton polygonsに変更
the two segments case is essential to show Conjecture~\ref{ConjOfArbitzeta}.

This paper is organized as follows:
In Section~\ref{Prelim}, we recall definitions of 
$p$-divisible groups, truncated Dieudonn\'e modules
of level one and related matters.
Moreover, we introduce a notion of ``arrowed binary sequences". 
%%``"内を追加
The proofs of our main results are described using this notion.
%% 説明を追加
In Section~\ref{ConSpeAbs}, 
to show the first result, 
we give explanations about tools
which are used to construct specializations combinatorially,
and show some properties of these tools. 
In Section~\ref{ClassBoundaryComp}, we give a proof of 
Theorem~\ref{ThmOfCentralClassification}.
This theorem classifies boundary components of central streams.
%%classify を classifies に変更
In Section~\ref{ConstGoodSpe}, 
we see the key proposition (Proposition~\ref{PropOfSpe})
which is used to prove the second result,
and we show Theorem~\ref{ThmOfzetaxi}.
%%the theorem を Theorem~\ref{ThmOfp-divSpe} に変更

\section{Preliminary} 
\label{Prelim}

In this section, we recall definitions of $p$-divisible groups,
central leaves, minimal $p$-divisible groups,
central streams and ${\rm DM_1}$'s.
Finally, we introduce the notion of arrowed binary sequences;
see Definition~\ref{DefOfSequences},
used in the proofs of main theorems.

\subsection{$p$-divisible groups and Dieudonn\'e modules}
\label{p-divAndDieudonne}

First, let us recall the definition of $p$-divisible groups.
Let $p$ be a prime number.
Let $h$ be a non-negative integer.
Let ${\rm S}$ be a scheme in characteristic $p$.
We say that $X$ is a {\it $p$-divisible group} (Barsotti-Tate group) 
of height $h$ over ${\rm S}$
if $X$ is an inductive system $X = (G_v, i_v)_{v \geq 1}$ for natural numbers $v$,
where $G_v$ is a finite locally free commutative group scheme 
over ${\rm S}$ of order $p^{vh}$,
and for each $v$, there exists the exact sequence 
of commutative group schemes
$$0 \rightarrow G_v \xrightarrow[]{i_v} 
G_{v+1} \xrightarrow[]{p^v}G_{v+1},$$
where $i_v$ is a canonical inclusion.
Let $X = (G_v, i_v)_{v \geq 1}$ be a $p$-divisible group over ${\rm S}$.
For an arbitrary scheme ${\rm T}$ over ${\rm S}$,
we have the $p$-divisible group $X_{\rm T}$ over $\rm T$
which is defined by $(G_v \times_{\rm S} {\rm T},\ i_v \times {\rm id})_{v \geq 1}$.
In particular, if $\rm T$ is a closed point $s$ over ${\rm S}$,
%% \rm Spec(k) 削除
then the $p$-divisible group $X_s$ is called fiber of $X$ over $s$. 

Let $K$ be a perfect field of characteristic $p$.
We denote by $W(K)$ the ring of Witt-vectors with coefficients in $K$.
Let $\sigma$ be the Frobenius over $K$.
We denote by the same symbol $\sigma$ the Frobenius over $W(K)$ 
if no confusion can occur.
We say that $M$ is a {\it Dieudonn\'e module over $K$} if 
$M$ is a finite $W(K)$-module equipped with 
$\sigma$-linear homomorphism ${\rm F}:M \rightarrow M$
and $\sigma^{-1}$-linear homomorphism ${\rm V}:M \rightarrow M$
satisfying that ${\rm F}\circ {\rm V}$ and ${\rm V} \circ {\rm F}$
is equal to the multiplication by $p$.
%% \sigma-homomorphism 等を修正
We use the covariant Dieudonn\'e theory,
which says that 
there exists a canonical categorical equivalence $\mathbb D$
from the category of $p$-divisible groups 
(resp. finite commutative group schemes) over $K$
to that of Dieudonn\'e modules over $K$
which are free as $W(K)$-modules
(resp. are of finite length).
%% It is known that...削除，上を追加

Here, let us recall the notion of minimal $p$-divisible groups.
We define a $p$-divisible group $H_{m, n}$ over $\mathbb F_p$ as follows:
$H_{m, n}$ is of dimension $n$, and its Serre-dual is of dimension $m$.
Moreover, the Dieudonn\'e module is obtained by
\begin{equation}
\mathbb D(H_{m,n}) = \bigoplus_{i = 1}^h \mathbb Z_p e_i,
\end{equation}
where $h = m+n$, and $\mathbb Z_p$ is the ring of $p$-adic integers.
For the basis $e_i$, operations ${\rm F}$ and ${\rm V}$ satisfy that 
${\rm F}e_i = e_{i-m}$, ${\rm V}e_i = e_{i-n}$ and $e_{i-h} = pe_i$.

Let $\{(m_i, n_i)\}_{i = 1, \dots, z}$ be a finite number of 
pairs of coprime non-negative integers
endowed with a non-increasing order of $\lambda_i := n_i/h_i$
with $h_i = m_i + n_i$,
i.e., we have $\lambda_1 \geq \lambda_2 \geq \cdots \geq \lambda_z$.
A {\it Newton polygon} $\xi = \sum_{i = 1}^z\, (m_i, n_i)$
is a lower convex polygon in $\mathbb{R}^2$,
breaking on integral coordinates,
consisting of slopes $\lambda_i$.
We call each coprime pair $(m_i, n_i)$ {\it segment}.
%% segmentの説明追加

For a Newton polygon $\xi = \sum_i\, (m_i, n_i)$, 
we set a $p$-divisible group 
\begin{equation}\label{DefOfMinimalp-div}
H(\xi) = \bigoplus_i H_{m_i, n_i}.
\end{equation}
We say that a $p$-divisible group $X$ is {\it minimal} if 
there exists a Newton polygon $\xi$ such that 
$X$ is isomorphic to $H(\xi)$ over an algebraically closed field.
For a $p$-divisible group $Y$,
there exists an isogeny from $Y$ to $H(\xi)$ over an algebraically closed field
for some Newton polygon $\xi$.
This $\xi$ is called the {\it Newton polygon of $Y$},
which is denoted by ${\rm NP}(Y)$.

In \cite[2.1]{oortFoliations},
for a $p$-divisible group $\mathcal Y$ of height $h$ 
%% of height $h$ を追加．
over a scheme ${\rm S}$ in characteristic $p$ 
and for a $p$-divisible group $Y$ over a field of characteristic $p$,
%%  in characteristic $p$を追加
Oort gave the definition of a {\it leaf} by
\begin{equation} \label{DefOfLeaf}
\mathcal C_Y({\rm S}) = \{s \in {\rm S}
\mid \mathcal Y_s \text{ is isomorphic to } Y 
\text { over an algebraically closed field}\},
\end{equation}
which is considered as a locally closed subscheme of ${\rm S}$
by giving $\mathcal C_Y({\rm S})$ the induced reduced structure.

Let ${\rm K}$ be a field of characteristic $p$,
and let $X$ be a $p$-divisible group over ${\rm K}$.
Let $\mathcal Y \to {\rm S}$ be a $p$-divisible group of 
height $h$ and dimension $d$
over a noetherian scheme ${\rm S}$ over ${\rm K}$.
Let $\xi$ be a Newton polygon starting at $(0, 0)$, ending at $(h, d)$.
We write
\begin{eqnarray}
\mathcal W^0_\xi({\rm S}) = \{s \in {\rm S} \mid {\rm NP}(\mathcal Y_s) = \xi\}.
\end{eqnarray}
%where ${\rm NP}(\mathcal Y_s)$ denotes the Newton polygon of $\mathcal Y_s$.
By Grothendieck-Katz \cite{katzcrystal},
we know that $\mathcal W^0_\xi({\rm S}) \subset {\rm S}$ is locally closed.
We recall Oort's result on leaves:
\begin{theorem}[\cite{oortFoliations}, Theorem 2.2] \label{FactOfW^0}
Let $\rm K$ be a field, and let $X_0 \to {\rm Spec}({\rm K})$
be a $p$-divisible group over ${\rm K}$.
Set $\xi = {\rm NP}(X_0)$.
Let ${\rm S} \to {\rm Spec}({\rm K})$ be an excellent scheme over ${\rm K}$.
%excellent scheme: it has a cover by open affine subschemes with the same property
For a $p$-divisible group $\mathcal Y \to {\rm S}$,
\begin{eqnarray}
\mathcal C_{X_0}({\rm S}) \subset \mathcal W^0_\xi ({\rm S})
\end{eqnarray}
is a closed subset.
\end{theorem}
Using Theorem~\ref{FactOfW^0}, we show
\begin{lemma} \label{LemOfStrictlyMin}
Let $X$ be a specialization of the minimal $p$-divisible group $H(\xi)$
for a Newton polygon $\xi$.
Assume that $X$ is not isomorphic to $H(\xi)$ over an algebraically closed field.
Then ${\rm NP}(X) \precneqq \xi$.
\end{lemma}
\begin{proof}
It suffices to show the case that $X$ is a $p$-divisible group over
an algebraically closed field $k$ of characteristic $p$.
%Let $k$ be an algebraically closed field of characteristic $p$.
Let $R$ be a discrete valuation ring over $k$.
Let $L$ denote the fractional field of $R$.
Let $\mathfrak X \to {\rm Spec}(R)$ be a $p$-divisible group satisfying that
$\mathfrak X_k \simeq X$ and $\mathfrak X_{\overline L} \simeq H(\xi)_{\overline L}$.
Suppose that ${\rm NP}(X) = \xi$ were true.
By applying Theorem~\ref{FactOfW^0} to 
$(\mathcal Y \to {\rm S}) = (\mathfrak X \to {\rm Spec}(R))$
and $X_0 = H(\xi)_k$,
we have $X \simeq H(\xi)_k$.
This is a contradiction.
\end{proof}
%% 上のパラグラフを追加

For a $p$-divisible group $X$, the kernel of $p$-multiplication $p : X \to X$
is called the {\it $p$-kernel} of $X$, denoted by $X[p]$.
The Dieudonn\'e module of $X[p]$ makes a truncated Dieudonn\'e module of level one,
defined below:
\begin{definition}
Let $N$ be a $K$-vector space of finite dimension.
Let ${\rm F}$ and ${\rm V}$ be a $\sigma$-linear map
and a $\sigma^{-1}$-linear map respectively
from $N$ to itself.
A triple $(N, {\rm F}, {\rm V})$ is a
{\it truncated Dieudonn\'e modules of level one} over $K$
(abbreviated as ${\rm DM_1}$) if 
the above ${\rm F}$ and ${\rm V}$ satisfy that 
${\rm Ker}\, {\rm F} = {\rm Im}\, {\rm V}$ and 
${\rm Im}\, {\rm F} = {\rm Ker}\, {\rm V}$. 
We say that a ${\rm DM_1}$ $(N, {\rm F}, {\rm V})$ is of height $h$ and
dimension $d$ if
$\dim_K N = h$ and $\dim_K N/{\rm V}N = d$. 
\end{definition}
%% 次元の定義を追加
%%It is known that a truncated Dieudonn\'e module of level one is obtained by the Dieudonn\'e module of $H(\xi)[p]$. の$H(\xi)[p]$を$X[p]$に変更
%% ほとんど差し替え
%%BT1, DM1のくだりは削除
Let us recall the notion of specializations.
Let $R$ be a commutative ring of characteristic $p > 0$.
Let $\sigma:R \rightarrow R$ be the Frobenius endomorphism 
defined by $\sigma(a)=a^p$.
Then the definition of ${\rm DM_1}$'s over the ring $R$ is given as follows:
\begin{definition}\label{DefFamilyDM1}

A ${\rm DM_1}$ over $R$ of height $h$ is a quintuple 
${\cal N}=({\cal N}, C, D, F, V^{-1})$ 
satisfying
\begin{enumerate}
\item[(i)] ${\cal N}$ is a locally free $R$-module of rank $h$,
%% free を locally free に変更 
\item[(ii)]$C$ and $D$ are submodules of ${\cal N}$
which are locally direct summands of ${\cal N}$,
\item[(iii)]
$F:({\cal N}/C) \otimes_{R,\sigma} R\rightarrow D$ and 
$V^{-1}: C\otimes_{R,\sigma} R\rightarrow {\cal N}/D$
are $R$-linear isomorphisms. 
\end{enumerate}
\end{definition}
Put $R = K[\![t]\!]$.
Let $\mathcal N$ be an arbitrary ${\rm DM_1}$ over $R$.
We set $\mathcal N_K = \mathcal N \otimes_R K$,
and we have a ${\rm DM_1}$ over $K$.
From this we obtain a map, called a {\it specialization},
$$\{{\rm DM_1} \text{ over }R\} \rightarrow \{{\rm DM_1} \text{ over }K\}$$
which maps $\mathcal N$ to $\mathcal N_K$.

Let $\xi = \sum_i (m_i, n_i)$ be a Newton polygon.
We denote by $N_{\xi}$ the ${\rm DM_1}$ 
associated to the $p$-kernel of $H(\xi)$.
Equivalently, $N_\xi$ is described as
\begin{equation}\label{DefOfMinDM1}
N_{\xi} = \bigoplus N_{m_i, n_i},
\end{equation}
where $N_{m,n}$ is the ${\rm DM_1}$ associated to
the $p$-kernel of $H_{m,n}$.
In this paper, we mainly treat ${\rm DM_1}$'s $N_{\xi}$
with $\xi = (m_1, n_1) + (m_2, n_2)$ satisfying $\lambda_2 < 1/2 < \lambda_1$.

Let $k$ be an algebraically closed field of characteristic $p$.
Then the following classification of ${\rm DM_1}$'s over $k$ is given by
Kraft \cite{kraftKom}, Oort \cite{oortstr} 
and Moonen-Wedhorn \cite{moonen-wedhornDis}.

\begin{theorem}\label{ThmOfDMAndSeq}
There exists a bijection:    
$$\{0,1\}^h\leftrightarrow \{{\rm DM_1}\ 
{\rm over}\ k\ {\rm of\ height}\ h\}/\cong.$$
\end{theorem}

Let us review the construction of the bijection above.
We often identify an element $S$ of $\{0, 1\}^h$ with
the pair of a totally ordered set $\tilde S = \{t_1 < \dots < t_h\}$
and a map $\delta : \tilde S \rightarrow \{0, 1\}$, 
so that the $i$-th coordinate of $S$ is $\delta(t_i)$. 
We write this identification as an equality $S = (\tilde S, \delta)$. 
The bijection of Theorem~\ref{ThmOfDMAndSeq} is obtained in the following way.
Let $S = (\tilde S, \delta)$ be as above.
%an element of $\{0, 1\}^h$ 
%with $\tilde S = \{t_1 < \dots < t_h\}$.
To $S$, we associate a ${\rm DM_1}$ $(N, {\rm F}, {\rm V})$ as follows.
Let $N =ke_1 \oplus \cdots \oplus ke_h$.
We define a map ${\rm F}$ by 
$${\rm F}e_i =
\begin{cases}
e_j,\ j = \#\{ l \mid \delta(t_l) = 0,\ l \leq i\} & {\rm for}\ \delta(t_i) = 0,\\
0 & {\rm otherwise}.
\end{cases}
$$
Let $t_{j_1}, \dots, t_{j_c}$,
with $j_1 < \dots < j_c$, be the
elements of $\tilde S$ satisfying $\delta(t_{j_l}) = 1$.
Put $d = h - c$.
Then a map ${\rm V}$ is defined by
$$
{\rm V}e_i =
\begin{cases}
e_{j_l},\ l = i - d & {\rm for}\ i > d,\\
0 & {\rm otherwise}.
\end{cases}
$$
We call $\{e_1, \dots, e_h\}$ a {\it standard basis} of $(N, {\rm F}, {\rm V})$.

\begin{example}\label{ExOfDM1}
Let us see an example of ${\rm DM_1}$'s.
Let $S = (1,1,0,1,0)$ be an element of $\{0, 1\}^5$.
Then the ${\rm DM_1}$ $(N, {\rm F}, {\rm V})$, with
$N = ke_1\oplus \dots \oplus ke_5$, corresponding to $S$
is given by the following diagram.
$$\xymatrix{1 &1
&0\ar@/_20pt/[ll]_{\rm F} \ar@/^20pt/[ll]_{\rm V} 
&1 \ar@/^20pt/[ll]_{\rm V} 
&0\ar@/_20pt/[lll]_{\rm F} \ar@/^20pt/[l]_{\rm V}\\
e_1 &e_2 &e_3 &e_4 &e_5}$$
For the above diagram, if there exists no vector of ${\rm F}$ (resp. ${\rm V}$)
from $e_i$, 
then we regard ${\rm F}$ (resp. ${\rm V}$) maps $e_i$ to zero.
One can check that the above satisfies the condition of ${\rm DM_1}$'s.
\end{example}

Let $h$ and $c$ be non-negative integers.
Put $d = h - c$.
Let $W$ be the Weyl group of the general linear group $GL_h$.
We identify $W$ with the symmetric group $\mathfrak S_h$.
Here, we associate ${\rm DM_1}$'s over $k$ 
with elements of the Weyl group. 
%% Here, we... の文章の位置を変更
Let $s_i \in W$ be the simple reflection $(i, i+1)$ for $i = 1, \dots, h-1$.
Let $\Omega$ be the standard generator of $\mathfrak S_h$,
namely $\Omega = \{s_1, \dots, s_{h-1}\}$.
%% standard generator の記号を \Delta から \Omega に変更
Set $J_c = \Omega \setminus \{s_c\}$.
Let $W_J$ be the subgroup of $W$ generated by $J = J_c$.
Let ${}^J W$ be the set consisting 
%% consists を consisting に変更
of elements of minimal length 
in $W_J \backslash W$, 
i.e., the shortest representatives of $W_J \backslash W$.
Let $w$ be an element of ${}^J W$.
We associate $S = (\tilde S, \delta)$ to $w$ by the pair of
a totally ordered set $\tilde S = \{t_1 < \cdots < t_h\}$ and
the map $\delta : \tilde S \rightarrow \{0, 1\}$ defined by
$\delta(t_i) = 0$ if and only if $w(i) > c$.
Here the property of the minimal length of $w$ is used.
We regard this pair $(\tilde S, \delta)$ as the element of $\{0, 1\}^h$.
One can check that this gives a one-to-one correspondence
between ${}^J W$ and the set of elements $S$ of $\{0, 1\}^h$
satisfying $\# \{t \in \tilde S \mid \delta(t) = 0\} = d$.
Thus there exists a bijection between ${}^J W$ and
the set of isomorphism classes of ${\rm DM_1}$'s over $k$ 
of height $h$ and dimension $d$.
%% This association $w \mapsto S$ is a bijection which uses the property of the minimal length of $w$. を追加　せず，それに相当する文章を追加 
We denote by $w_\xi$ the element of ${}^J W$
associated to $N_\xi$.

%% 次のパラグラフ追加
In the rest of this subsection,
we show a lemma used for the construction of generic specializations.
Let $W$ and $J = J_c$ be as above.
Set $d = h - c$.
We define $x \in W$ to be 
$x(i) = i+d$ if $i \leq c$ and $x(i) = i-c$ otherwise.
We define $\theta : W \to W$ by $u \mapsto xux^{-1}$.
It follows from \cite[4.10]{VWEOShimura} by Viehmann-Wedhorn
that $w' \subset w$ if and only if 
there exists an element $u$ of $W_J$ such that
$u^{-1}w'\theta(u) \leq w$, where 
$\leq$ denotes the Bruhat order. 
%if there exists a ${\rm BT_1}$ $G \to {\rm Spec}(R)$
%over a discrete valuation ring $(R, \mathfrak m)$ such that
%$G_{\overline L}$ (resp. $G_{\overline \kappa}$) is of type $w$ (resp. $w'$)
%with $L = {\rm frac}(R)$ and $\kappa = R/\mathfrak m$,
%then $w' \subset w$.

\begin{lemma} \label{LemOfGeneSpew'w}
Let $w'$ and $w$ be elements of ${}^J W$ with $w' \subset w$.
If $\ell(w') = \ell(w) - 1$, then
there exist $v \in W$ and $u \in W_J$ such that
\begin{itemize}
\item[(i)] $v = ws$ for a transposition $s$,
\item[(ii)] $\ell(v) = \ell(w) - 1$,
\item[(iii)] $w' = u v \theta (u^{-1})$.
\end{itemize} 
\end{lemma}

\begin{proof}
Let $w \in {}^J W$.
Let $w' \in {}^J W$ satisfying that $w' \subset w$ and 
$\ell(w') = \ell(w) - 1$.
Choose an element $u$ of $W_J$ satisfying that
$u^{-1} w' \theta(u) < w$.
Set $v = u^{-1} w' \theta(u)$.
We show (ii). 
Since $w'$ is an element of ${}^J W$,
we have
$\ell(v) 
\geq
\ell(u^{-1}w') - \ell(\theta(u)^{-1})
= \ell(u) + \ell(w') - \ell(\theta(u))$.
Moreover, $\ell(u) + \ell(w') - \ell(\theta(u)) = \ell(w')$
since for all element $u'$ of $W_J$ we have $\ell(u') = \ell(\theta(u'))$
by the definition of $\theta$.
As $v < w$, we have $\ell(v) < \ell(w)$.
These prove (ii).
Let $w = s_{i_1} s_{i_2} \dots s_{i_l}$ be a reduced expression of $w$
with $v = s_{i_1} \dots s_{i_{q-1}} s_{i_{q+1}} \dots s_{i_l}$.
Set $s = (s_{i_l} \dots s_{i_{q+1}})s_{i_q}(s_{i_{q+1}} \dots s_{i_l})$.
Then $s$ is a transposition,
and this $s$ satisfies $v = ws$. 
%Write $s = (a, b)$ with $a < b$.
%Then $wsw^{-1} = (w(a), w(b))$,
%and we have $\ell(wsw^{-1}w) = \ell(v) = \ell(w) - 1$.
%Since $v = wsw^{-1}w$ with $w \in {}^J W$,
%we see that $wsw^{-1}$ does not belong to $W_J$.
%By the following lemma,
%we see $w(b) \leq d < w(a)$.
\end{proof}
%\begin{lemma} \label{LemOfwbwa}
%Let $w \in W$. 
%Let $s = (a, b)$ be a transposition with $a < b$.
%Set $v = ws$.
%If $\ell(v) < \ell(w)$, 
%then $w(b) < w(a)$.
%\end{lemma}
%\begin{proof}
%The length $\ell(w)$ of $w$ is defined by
%$$\ell(w) = \# \{(i, j) \mid i < j,\ \ w(i) > w(j)\}.$$
%We set
%$$\ell_0 = \# \{(i, j) \mid i < j,\ \ w(i) > w(j),\ \ (i < a \text{ or } b < j)\}.$$
%Then $\ell_0(w) - \ell_0(v) = 0$.
%Put
%$$\ell_1(w) = \# \{ (i, j) \mid i< j,\ \ w(i) > w(j),\ \ a \leq i < j \leq b\}.$$
%If $w(a) < w(b)$, then $\ell_1(w) - \ell_1(v)$ is not positive,
%which contradicts with $\ell(w) = \ell_0(w) + \ell_1(w)$ and the hypothesis
%$\ell(v) < \ell(w)$.
%\end{proof}

\subsection{Definition of arrowed binary sequences}
\label{ABS}

To show our main results, we introduce arrowed binary sequences.
This object can be regarded as a generalization of ${\rm DM_1}$'s.

\begin{definition}\label{DefOfSequences}
An {\it arrowed binary sequence} (we often abbreviate as ABS) 
is the triple $(\tilde S, \delta, \pi)$
consisting of a totally ordered set
$\tilde S = \{t_1 < \cdots < t_h\}$,
a map $\delta: \tilde S \rightarrow \{0, 1\}$
and a bijection $\pi : \tilde S \rightarrow \tilde S$.
We denote by $\mathcal H$ the set of all arrowed binary sequences.
\end{definition}

\begin{definition}\label{DefOfPi}
Let $N = (N, {\rm F}, {\rm V})$ be a ${\rm DM_1}$ of height $h$.
Let $e_1, \dots, e_h$ be a standard basis of $N$.
Let $S = (\tilde S, \delta)$ be the element of $\{0, 1\}^h$ corresponding to $N$
with $\tilde S = \{t_1, \dots, t_h\}$.
%We regard ${\rm F}$ and ${\rm V}$ as maps on $\tilde S$
%by the natural bijection between the basis of $N$ and $\tilde S$.
%%上を消した．
%%with $\tilde S = \{t_1, \dots, t_h\}$. を追加
We define a bijective map $\pi: \tilde S \rightarrow \tilde S$ by
$\pi(t_i) = t_j$, where $j$ is uniquely determined by
\begin{equation}\label{EqOfPi}
\begin{cases}
{\rm F}(e_i) = e_j & \text{if }\delta(t_i) = 0, \\
{\rm V}(e_j) = e_i & \text{otherwise}.
\end{cases}
\end{equation}
Note that $(\tilde S, \delta, \pi)$ is an ABS.
We say that an ABS $S = (\tilde S, \delta, \pi)$ is {\it admissible} if
there exists a ${\rm DM_1}$ $N = (N, {\rm F}, {\rm V})$ such that
$(\tilde S, \delta)$ corresponds to $N$ by Theorem~\ref{ThmOfDMAndSeq},
%% $(\tilde S, \delta)$ corresponds to $N$ by Theorem~\ref{ThmOfDMAndSeq}, 追加
and $\pi$ is constructed by \eqref{EqOfPi} from $N$.
The admissible ABS $S$ obtained from a ${\rm DM_1}$ $N$ is called
{\it the ABS associated to $N$}.
We denote by $\mathcal H'$ the set of all admissible ABS's.
\end{definition}
%% 上の定義を改良

%\begin{remark}
%By construction, there exists a one-to-one correspondence 
%between sets ${}^J W$ and $\mathcal H'$.
%We often identify these sets if no confusion can occur.
%\end{remark}

For an arrowed binary sequence $S = (\tilde S, \delta, \pi)$,
as seen in Example~\ref{ExampleOfABS},
we obtain a diagram of the ABS using elements of $\tilde S$ and arrows
$$\xymatrix{\bullet & \bullet \ar@/_15pt/[l]_\pi \ ,
& \bullet \ar@/_15pt/[r]^\pi &\bullet}.$$

\begin{example}\label{ExampleOfABS}
The diagram of the ABS
corresponding to $(1,1,0,1,0)$ is
$$\xymatrix{1\ar@/_20pt/[rr]^\pi &1\ar@/_20pt/[rr]^\pi
&0\ar@/_20pt/[ll]_\pi &1\ar@/_20pt/[r]^\pi &0\ar@/_20pt/[lll]_\pi}.$$
From Example~\ref{ExOfDM1} and this diagram,
one can check that the admissible ABS $(\tilde S, \delta, \pi)$
is obtained by a ${\rm DM_1}$.
\end{example}

%By construction, for each element $S$ of $\{0, 1\}^h$
%associated with a ${\rm DM_1}$, 
%we uniquely obtain the bijection $\pi$. 
%% the map $\pi$ satisfies that $\pi(t) < \pi(t')$ if $t < t'$ and $\delta(t) = \delta(t')$. をLemma~\ref{LemOftt'BinExp}に移動

Now we show a property of admissible ABS's.
Let $S = (\tilde S, \delta, \pi)$ be an ABS.
For each element $t$ of $\tilde S$,
we define the {\it binary expansion} $b(t)$ by
$b(t) = 0.b_1b_2\cdots$,
where $b_i = 0$ if $\delta(\pi^{-i}(t)) = 0$,
and $b_i = 1$ otherwise.
\begin{lemma} \label{LemOftt'BinExp}
Let $S = (\tilde S, \delta, \pi)$ be an admissible ABS.
Let $t$ and $t'$ be elements of $\tilde S$.
Then the following holds.
\begin{itemize}
\item[(i)] Suppose $\delta(t) = \delta(t')$. 
Then $t < t'$ if and only if $\pi(t) < \pi(t')$,
\item[(ii)] Suppose $b(t) \neq b(t')$. 
Then $t < t'$ if and only if $b(t) < b(t')$.
\end{itemize}
\end{lemma}
\begin{proof}
(i) follows from the construction of $\pi$ defined in Definition~\ref{DefOfPi}.

Let us see (ii).
To show ``only if'' part, we suppose $b(t) > b(t')$.
It implies that there exists a natural number $v$ such that
$\delta(\pi^{-u}(t)) = \delta(\pi^{-u}(t'))$ for all $u$ with $u < v$
and $\delta(\pi^{-v}(t)) > \delta(\pi^{-v}(t'))$.
Then we clearly have $\delta(\pi^{-v}(t)) = 1$ and $\delta(\pi^{-v}(t')) = 0$.
By the construction of $\pi$, we have $\pi^{-v+1}(t) > \pi^{-v+1}(t')$,
and this contradicts with (i).
Let us see ``if'' part.
If $b(t) < b(t')$,
then there exists a natural number $v$ such that
$\delta(\pi^{-u}(t)) = \delta(\pi^{-u}(t'))$ for $u < v$
and $\delta(\pi^{-v}(t)) < \delta(\pi^{-v}(t'))$.
Then we have $\delta(\pi^{-v}(t)) = 0$ and $\delta(\pi^{-v}(t')) = 1$,
and $\pi^{-v+1}(t) < \pi^{-v+1}(t')$ holds.
Using (i) repeatedly, we see $t < t'$.
\end{proof} 

%From now on, we mainly use arrowed binary sequences 
%to show the main theorem.
Let us define special ${\rm DM_1}$'s 
which correspond to special $p$-divisible groups $H_{m,n}$.

\begin{definition}
We say a ${\rm DM_1}$ $N$ is {\it $DM_1$-simple}
if there exist coprime natural numbers $h$ and $m$
such that $N$ is associated to the ABS 
$(\{t_1, \dots, t_h\}, \delta, \pi)$,
where $\delta(t_i) = 1$ if $i \leq m$
and $\delta(t_i) = 0$ otherwise,
with the map $\pi$ defined as Definition~\ref{DefOfPi}.
%% 上の言い方を変更
In other words, we have $\pi(t_i) = t_{i-m \bmod h}$.
That kind of ${\rm DM_1}$ corresponds to $N_{m,n}$
with $n = h - m$,
and we call this ${\rm DM_1}$ {\it simple $DM_1$} 
if no confusion can occur.
\end{definition}

Here, let us recall the direct sum of admissible ABS's,
which corresponds to the direct sum of associated Dieudonn\'e modules.
%introduced in \cite[p.\,221]{HarashitaCon}.
%For an admissible ABS $S$,
%if $t < t'$ in $\tilde S$,
%then $b(t) < b(t')$ holds.
%% 先のlemmaに書いたので消した．
%% binary expansionの説明を先に回した．
We define the direct sum $S = (\tilde S, \delta, \pi)$ of elements
$A = (\tilde A, \delta_A, \pi_A)$ and 
$B = (\tilde B, \delta_B, \pi_B)$ of $\mathcal H'$ as follows.
We define a set $\tilde S$ by $\tilde S =\tilde A \sqcup \tilde B$. 
We define a map $\delta : \tilde S \rightarrow \{0, 1\}$ to be 
$\delta|_{\tilde A} = \delta_A$ and $\delta|_{\tilde B} = \delta_B$,
and let $\pi: \tilde S \rightarrow \tilde S$ be a map satisfying that 
$\pi|_{\tilde A} = \pi_A$ and $\pi|_{\tilde B} = \pi_B$. 
We define an order $<$ in $\tilde S$ so that 
\begin{itemize}
\item[(i)] for elements $t, t' \in \tilde S$, 
if $b(t) \leq b(t')$, then $t < t'$, and
\item[(ii)] there is no elements $t, t'$ of $\tilde S$ such that
$t < t'$ and $\pi(t') < \pi(t)$,
\end{itemize}
where $b(t)$ is the binary expansion of $t \in \tilde S$ determined by
$\pi$.
For instance, if $A = B$ with $\tilde A = \{t_1, \dots, t_h\}$
and $\tilde B = \{t'_1, \dots, t'_h\}$,
then $\tilde S = \{t_1, t'_1, \dots, t_h, t'_h\}$.
Thus we get the ABS $A \oplus B = (\tilde S, \delta, \pi)$
which also belongs to $\mathcal H'$.
Let $M$ and $N$ be ${\rm DM_1}$'s,
and let $A$ (resp. $B$) be the ABS
corresponding to $M$ (resp. $N$).
Then the ABS $A \oplus B$ corresponds to 
the direct sum of ${\rm DM_1}$'s $M \oplus N$.
In this paper, we consider the case that 
$M = N_{m_1,n_1}$ and $N = N_{m_2,n_2}$,
with pairs of coprime non-negative integers $(m_1,n_1)$ and $(m_2,n_2)$.
%% 直和の定義を修正

\begin{example}\label{ExOfDirectSum}
Let $M = N_{2,7}$ and $N = N_{5,3}$.
Let $A$ and $B$ be ABS's corresponding to 
$M$ and $N$ respectively.
Then diagrams of these ABS's are given by the following:\\
\\
$$A = \xymatrix@=10pt{\underline{1}\ar@/_20pt/[rrrrrrr] &1\ar@/_20pt/[rrrrrrr]
&0\ar@/_20pt/[ll] &0\ar@/_20pt/[ll] &0\ar@/_20pt/[ll] 
&0\ar@/_20pt/[ll] &0\ar@/_20pt/[ll] &0\ar@/_20pt/[ll] 
&0\, ,\ar@/_20pt/[ll]}\ \ \ 
B = \xymatrix@=10pt{1\ar@/_20pt/[rrr] &1\ar@/_20pt/[rrr] &1\ar@/_20pt/[rrr]
&1\ar@/_20pt/[rrr] &1\ar@/_20pt/[rrr] 
&0\ar@/_20pt/[lllll] &0\ar@/_20pt/[lllll] &0\ar@/_20pt/[lllll]}.$$
\\
We write $t_1$ for the element of $\tilde A$ 
corresponding to the underlined element in the diagram.
Then the binary expansion $b(t_1)$ is given by 
$b(t_1) = 0.000010001\cdots$.
In the same way, we obtain the binary expansions of
all elements of $\tilde A$ and $\tilde B$.
We sort all elements by the binary expansions, 
and the direct sum $A \oplus B$ is given by the following:\\
\\
$$\xymatrix@=2pt{1^A_1\ar@/_20pt/[rrrrrrrrrr] & 1^A_2\ar@/_20pt/[rrrrrrrrrr] &
0^A_3\ar@/_20pt/[ll] & 0^A_4\ar@/_20pt/[ll] & 0^A_5\ar@/_20pt/[ll] &
0^A_6\ar@/_20pt/[ll] & 0^A_7\ar@/_20pt/[ll] &
1^B_1\ar@/_20pt/[rrrrr] & 1^B_2\ar@/_20pt/[rrrrr] &
1^B_3\ar@/_20pt/[rrrrr] & 0^A_8\ar@/_20pt/[lllll] &
0^A_9\ar@/_20pt/[lllll] &1^B_4\ar@/_20pt/[rrr] &
1^B_5\ar@/_20pt/[rrr] & 0^B_6\ar@/_20pt/[lllllll] &
0^B_7\ar@/_20pt/[lllllll] & 0^B_8\ar@/_20pt/[lllllll]},$$
\\
\\
where $\tau^A_i$ (resp. $\tau^B_i$), with $\tau = 0$ or $1$,
is the $i$-th element of $A$ (resp. $B$).
\end{example}

For certain ${\rm DM_1}$'s, 
we have the following:

\begin{lemma} \label{LemOf1/2SepABS}
Let $\xi = (m_1, n_1) + (m_2, n_2)$ be a Newton polygon satisfying 
$\lambda_2 < 1/2 < \lambda_1$.
Set $h_1 = m_1 + n_1$ and $h_2 = m_2 + n_2$.
Let $N_{\xi}$ be the minimal ${\rm DM_1}$ of $\xi$.
For the above notation, 
the ABS $S$ associated to $N_\xi$ is obtained by the following:
$$\underbrace{1^A_1\cdots 1^A_{m_1}}_{m_1}
\underbrace{0^A_{m_1+1}\cdots 0^A_{n_1}}_{n_1-m_1}
\underbrace{1^B_1\cdots 1^B_{n_2}}_{n_2}
\underbrace{0^A_{n_1+1}\cdots 0^A_{h_1}}_{m_1}
\underbrace{1^B_{n_2+1}\cdots 1^B_{m_2}}_{m_2-n_2}
\underbrace{0^B_{m_2+1}\cdots 0^B_{h_2}}_{n_2}.$$
\end{lemma}

\begin{proof}
See \cite{HarashitaCon}, Proposition~4.20.
\end{proof}

\section{Constructing a specialization using arrowed binary sequences} 
\label{ConSpeAbs}

In this section, 
we give a method to construct a specialization of a ${\rm DM_1}$
using arrowed binary sequences.
We introduce some sets which help us 
to investigate properties of the specialization 
obtained by this method.
These properties are useful for classification of
boundary components of central streams, given in Section~\ref{ClassBoundaryComp}.

\subsection{Some notation for specializations} 
\label{NotaSpe}

Here, we prepare some notation for describing a specialization of a ${\rm DM_1}$
in arrowed binary sequences.
First, we define the lengths of ABS's.
We use some notation of Section~\ref{ABS}.

\begin{definition}
Let $S = (\tilde S, \delta, \pi)$ be an ABS.
We define the {\it length of $S$} by 
$$\ell(S) = \# \{(t', t) \mid \delta(t') = 0 \text{ and } 
\delta(t) = 1 \text{ with } t' < t \}.$$
\end{definition}

\begin{remark}
If an ABS $S$ corresponds to a ${\rm DM_1}$ $N$,
then the value $\ell(S)$ is equal to the length $\ell(w)$
for the element $w$ of ${}^J W$, which corresponds to $N$.
In other words, the length of $w$ can be calculated 
by the above $\ell(S)$.
\end{remark}

\begin{example}
For the ABS $A \oplus B$
associated to $N_{2,7}\oplus N_{5,3}$,
which is constructed in Example~\ref{ExOfDirectSum},
we have $\ell(A \oplus B) = 29$.
In general, for a Newton polygon $\xi = (m_1, n_1) + (m_2, n_2)$ of two segments
with $\lambda_2 < 1/2 < \lambda_1$,
the length $\ell(S)$ of the ABS $S$ corresponding to $N_{\xi}$
is equal to $m_2n_1 - m_1n_2$. 
\end{example}

Let $\mathcal H'(h, d)$ denote the set of admissible ABS's
whose corresponding ${\rm DM_1}$'s are of height $h$ and dimension $d$.
By translating the ordering $\subset$ on ${}^J W$
via the bijection from ${}^J W$ to $\mathcal H'(h, d)$,
we have an ordering on $\mathcal H'(h, d)$ as well as
the notion of specialization of admissible ABS's.

Here we give a method to construct a type of the specializations of ABS's.
Those will turn out to correspond to specializations of the form
$w' \subset w$ with $v = ws < w$ and $w' = uv\theta(u^{-1})$,
where $s$ is a transposition and $u \in W_J$.
See the paragraph before Lemma~\ref{LemOfGeneSpew'w}
for this specialization.
Starting with $S = (\tilde S, \delta, \pi_0) \in \mathcal H'(h, d)$,
%and $t_i, t_j \in \tilde S$ with $\delta(t_i) = 0$, $\delta(t_j) = 1$ and $t_i < t_j$
we construct a new admissible ABS $S' = (\tilde S', \delta, \pi)$.
First we consider a ``small modification'' of $\pi_0$:

\begin{definition}\label{DefOfExchanging}
Let $S = (\tilde S, \delta, \pi_0)$ be an ABS
with $\tilde S = \{t_1 < \cdots < t_h\}$.
Choose elements $t_i$ and $t_j$ of $\tilde S$.
We define a map $\pi$ on $\tilde S$ 
by $\pi(\pi^{-1}_0(t_i)) = t_j$, $\pi(\pi^{-1}_0(t_j)) = t_i$ and $\pi(t) = \pi_0(t)$
for the other elements $t$ of $\tilde S$.
We call this map $\pi$
a {\it small modification by $t_i$ and $t_j$ of $\pi_0$}.
\end{definition}

We require $\delta(t_i) = 0$, $\delta(t_j) = 1$ and $t_i < t_j$,
when we consider specializations.
%In Lemma~\ref{LemOfGeneSpew'w},
%we saw a property of generic specializations $w'$ of $w$.
%By translating Lemma~\ref{LemOfGeneSpew'w}
%in terms of ABS's,
%we obtain a ``specialization'' of an admissible ABS as follows.
%\begin{definition} \label{DefOfSpeOfABS}
%Let $S = (\tilde S, \delta, \pi_0)$ be an admissible ABS
%with $\tilde S = \{t_1, \dots, t_h\}$.
Let $\pi : \tilde S \to \tilde S$ be the 
small modification by elements $t_i$ and $t_j$ of $\pi_0$.
%Let us consider the case that
%$\delta(t_i) = 0$, $\delta(t_j) = 1$ and $t_i < t_j$.
%For the ABS $(\tilde S, \delta, \pi)$,
Let $\tilde S' = \tilde S$ as sets.
There uniquely exists an ordering $<'$ on $\tilde S'$ so that
\begin{itemize}
\item[(i)] If $t <' t'$, then $b(t) \leq b(t')$.
\item[(ii)] Suppose $\delta(t) = \delta(t')$.  
Then $t <' t'$ if and only if $\pi(t) <' \pi(t')$,
\end{itemize}
where $b(t)$ denotes the binary expansion determined by $\delta$ and $\pi$.
%We denote by $\tilde S'$ the set $\tilde S$ equipped with this order.
Then we get an admissible ABS $S' = (\tilde S', \delta, \pi)$;
see Remark~\ref{RemOfww'}.
We call this $S'$ the {\it specialization of $S$
obtained by exchanging $t_i$ and $t_j$}.
%Especially, if the specialization $S'$ obtained by exchanging $t_i$ and $t_j$
%satisfies that $\ell(S') = \ell(S) - 1$,
%then we write $S^-$ for $S'$.
We say {\it an exchange of $t_i$ and $t_j$ is good}
if the specialization $S'$ obtained by exchanging $t_i$ and $t_j$
satisfies $\ell(S') = \ell(S) - 1$.
We often write $S^-$ for $S'$
when $S'$ is the specialization obtained by a good exchange.
We call $S^-$ a {\it generic specialization} of $S$.
Any good exchange produces a generic specialization, and conversely
every generic specialization is obtained by a good exchange.
This follows from Lemma~\ref{LemOfGeneSpew'w} and Remark~\ref{RemOfww'} below:
%\end{definition}
%% 特殊化の定義を修正
%% 上のパラグラフを追加
%% generic specialization など言葉づかい修正

\begin{remark} \label{RemOfww'}
We use the same notation as Section~\ref{p-divAndDieudonne}.
%Let $S = (\tilde S, \delta, \pi_0)$ be an admissible ABS 
%with $\tilde S = \{t_1 < \dots < t_h\}$.
%Pick $t_i$ and $t_j$ satisfying that
%$\delta(t_i) = 0$ and $\delta(t_j) = 1$ with $t_i < t_j$.
%Let $\pi$ be the small modification by $t_i$ and $t_j$.
%Let $S' = (\tilde S', \delta, \pi)$ be the specialization of $S$, obtained by 
%exchanging $t_i$ and $t_j$ with $t_i < t_j$.
Let $S = (\tilde S, \delta, \pi_0) \in \mathcal H'(h, d)$
with $\tilde S = \{t_1 < \dots < t_h\}$.
Let $S' = (\tilde S', \delta, \pi)$ be the specialization of $S$
obtained by exchanging $t_i$ and $t_j$.
%Set $c = \#\{ t \in \tilde S \mid \delta(t) = 1\}$ and $d = h - c$.
Let $w$ be the element of ${}^J W$ corresponding to $S$.
%with $J = J_c$.
Set $s = (i, j)$ the transposition.
We regard $\pi_0$ and $\pi$ as elements of $W$.
Then $\pi_0 = xw$. 
Put $\tilde S' = \{t'_1 <' \dots <' t'_h\}$.
Set $\tilde S^{(0)} = \{t_1^{(0)} <'' \dots <'' t_h^{(0)}\}$ to be
$t_{s(z)} = t_z^{(0)}$.
We define an element $\varepsilon$ of $W$ to be
$t_{\varepsilon (z)}^{(0)}= t'_z$.
%Let us see that $\varepsilon$ belongs to $\mathfrak S_d \times \mathfrak S_c$.
%Clearly $\{t_1, \dots, t_d\} = \{t'_1, \dots, t'_d\}$
%and $\{t_{d+1}, \dots, t_h\} = \{t'_{d+1}, \dots, t'_h\}$ as sets.
Note that $\varepsilon$ stabilizes $\{1, 2, \dots, d\}$
(and therefore $\{d + 1, \dots, h\}$),
since $b(t_z^{(0)}) < 0.1$ for $z \leq d$ and $b(t_z^{(0)}) > 0.1$ for $z > d$,
where binary expansions are determined by $\delta$ and $\pi$.
%Hence there exists no natural numbers $z$ and $z'$ such that
%$\varepsilon(z') < \varepsilon(z)$ with $z \leq d < z'$.
%Set
%$\pi = \varepsilon^{-1} \check \pi \varepsilon$
%with $\check \pi = \pi_0 s$.
Put $v = ws$.
Then $w' = uv\theta(u^{-1})$ corresponds to $S'$
with $u = x^{-1} \varepsilon^{-1} x \in W_J$.
The map $\pi$ is obtained by $\varepsilon^{-1} \pi_0 s \varepsilon$. 
In terminologies of ABS's,
multiplying $\pi_0$ by $s$ corresponds to constructing the small modification.
By the coordinate transformation by $\varepsilon$,
we obtain the map $\pi$ on $\tilde S'$.
%See the paragraph before Lemma~\ref{LemOfGeneSpew'w}
%for definitions of $x$ and $\theta$.
%For the construction of the specialization $S'$ from $S$,
%the permutation $u$ is associated to giving the order on $\tilde S'$.
The main purpose of Section~\ref{CombSpe} is
to give the construction of this $\varepsilon$ combinatorially.
\end{remark}
%% このパラグラフを追加

\begin{example}
Let us see an example of constructing a specialization of an ABS.
Let $S = (\tilde S, \delta, \pi_0)$ be the admissible ABS 
associated to $(1, 0, 0, 1, 0, 1, 0, 0)$.  
This $S$ is described as the following diagram.
\\
$$\xymatrix@=6pt{1_1\ar@/_20pt/[rrrrr] & 0_2\ar@/_20pt/[l]
& 0_3\ar@/_20pt/[l] & 1_4\ar@/_20pt/[rrr] 
& 0_5\ar@/_20pt/[ll]
& 1_6\ar@/_20pt/[rr] & 0_7\ar@/_20pt/[lll] & 0_8\ar@/_20pt/[lll]}.$$
\\
Let $\pi$ be the small modification map by $0_2$ and $1_4$.
Then the ABS $S' = (\tilde S', \delta, \pi)$ is described as\\ \\
$$\xymatrix@=6pt
{
1_{ 4 } \ar@/_20pt/[rrrrr] &
1_{ 1 } \ar@/_20pt/[rrrrr] &
0_{ 3 } \ar@/_20pt/[ll] &
0_{ 2 } \ar@/_20pt/[ll] &
0_{ 5 } \ar@/_20pt/[ll] &
0_{ 7 } \ar@/_20pt/[ll] &
1_{ 6 } \ar@/_20pt/[r] &
0_{ 8 } \ar@/_20pt/[lll]
}.$$
\vspace{1mm}
\\
One can check that this specialization is generic by
calculating lengths of $S$ and $S'$.
\end{example}

%% 次を追加．
Let $\xi = (m_1, n_1) + (m_2, n_2)$ 
be a Newton polygon satisfying that $\lambda_2 < 1/2 < \lambda_1$.
Let $S$ denote the ABS of $N_\xi$.
To classify
generic specializations of $S$,
in Section~\ref{CombSpe}
we introduce a method to construct specializations $S'$ combinatorially.
%and classify generic specializations in a certain case.
%This method explains how to construct the set $\tilde S'$
%using a sequence of ABS's.
%For $\xi$,
%suppose that $m_1 + 1 < n_1$ or  $n_2 + 1 < m_2$.
%for an exchange of $t_i$ and $t_j$ with $\delta(t_i) = 0$, $\delta(t_j) = 1$
%and $t_i < t_j$.
The method gives an explicit construction of 
permutations $\varepsilon$ as in Remark~\ref{RemOfww'}.
%Using this method,
%we construct the permutation $\varepsilon$ 
%of Remark~\ref{RemOfww'} combinatorially.
Moreover, we show some properties
of generic specializations; 
see Section~\ref{SomePropOfBoundaryComp}.
These properties are useful for 
to show Proposition~\ref{PropOfSpe}
which is a key step of the proof of Theorem~\ref{ThmOfzetaxi}.
Let $A$ and $B$ be ABS's 
corresponding to $N_{m_1,n_1}$ and $N_{m_2,n_2}$ respectively.
Then $S$ is obtained by $A \oplus B$.
We often denote by $\tau^A_i$ (resp. $\tau^B_i$),
with $\tau = 0$ or $1$,
the $i$-th symbol of $A$ (resp. $B$).
%Consider the case that the Newton polygon $\xi$ satisfies 
%$n_1 = m_1 + 1$ and $m_2 = n_2 + 1$
%with $m_1 > 0$ or $n_2 > 0$.
%In this case, for a certain exchange,
%we cannot construct the specialization by the combinatorial method.
%So for these exchanges, in the last paragraph of this section,
%we shall construct the specialization by 
%calculating binary expansions explicitly.
%Let $N_{\xi}$ be the minimal ${\rm DM_1}$ of $\xi$,
%where a Newton polygon $\xi = (m_1, n_1) + (m_2, n_2)$
%consists of two segments
%satisfying that $\lambda_2 < 1/2 < \lambda_1$, i.e.,
%one slope is less than $1/2$, and the other is greater than $1/2$.
%Let $S$ be the ABS associated to $N_\xi$.
%We will construct specializations determined by 
%the small modification by $0^A_i$ and $1^B_j$.
By Lemma~\ref{LemOf1/2SepABS},
there exist only three cases as the choice of exchange of $0^A_i$ and $1^B_j$: 
\begin{enumerate}
\item[(1)] $m_1 < i \leq n_1$ and $1 \leq j \leq n_2$;
\item[(2)] $m_1 < i \leq n_1$ and $n_2 < j \leq m_2$;
\item[(3)] $n_1 < i \leq h_1$ and $n_2 < j \leq m_2$. 
\end{enumerate}
We assume that $m_1+1 < n_1$ or $n_2+1 < m_2$ in the case (2).
For $\xi = (m_1, m_1 + 1) + (n_2 + 1, n_2)$
with $m_1 > 0$ or $n_2 > 0$,
the case (2) is treated separately;
see below the proof of Proposition~\ref{PropOfEmptyB}.
Since the case (3) can be regarded as the dual of (1),
it suffices to deal with the case (1) and (2).
From now on, we assume $m_1 < i \leq n_1$.
We often treat the cases (1) and (2) simultaneously,
but for example, the proof of Proposition~\ref{PropOfEmptyB}
is divided into the cases (1) and (2).
%% 例外処理の位置を移動

\subsection{Constructing a specialization}
\label{CombSpe}

Let $\xi = (m_1, n_1) + (m_2, n_2)$ be a Newton polygon of two segments
satisfying $\lambda_2 < 1/2 < \lambda_1$.
Let $S = (\tilde S, \delta, \pi_0)$ be the ABS 
associated to the minimal ${\rm DM_1}$ $N_\xi$.
In this section, 
we introduce a combinatorial method to construct a specialization $S'$ of $S$ 
in order to classify generic specializations.
Concretely, for the small modification $\pi$ by $t_i$ and $t_j$,
we shall construct the ordered set $\tilde S'$, which coincides with $\tilde S$
as sets, satisfying that 
$t < t'$ if and only if $\pi(t) < \pi(t')$
for elements $t$ and $t'$ of $\tilde S'$ with $\delta(t) = \delta(t')$.
For this ordered set, 
using Theorem~\ref{ThmOfDMAndSeq},
we get the specialization $N'_\xi$ of $N_\xi$.
%% Concretely...を追加
The method gives a useful characterization to ABS's $S'$
satisfying $\ell(S') = \ell(S) - 1$.
The main purpose of this section is 
to show that this combinatorial operation is well-defined.

%Let $S = (\tilde S, \delta, \pi_0)$ be the ABS associated to the ${\rm DM_1}$ $N_\xi$.
Let $A$ (resp. $B$) be the ABS corresponding to $N_{m_1, n_1}$ (resp. $N_{m_2, n_2}$).
Then $S = A \oplus B$.
%Let $\tilde A$ and $\tilde B$ denote the ordered set of $A$ and $B$ respectively. 
We denote by $\tau^A_i$ (resp. $\tau^B_j$) the $i$-th element of $A$ 
(resp. the $j$-th element of $B$) with $\tau = 0$ or $1$.
Fix natural numbers $i$ and $j$ satisfying that $m_1 < i \leq n_1$ and $1 \leq j \leq m_2$.
Let $\pi$ be the small modification by $0^A_i$ and $1^B_j$.
For the ordered set 
$$\tilde S = \{1^A_1 < \cdots < 0^A_{i-1} < 0^A_i < 0^A_{i+1} < \cdots
< 1^B_{j-1} < 1^B_j < 1^B_{j+1} < \cdots < 0^B_{h_2}\},$$
we define an ordered set $S^{(0)}$ by $S^{(0)} = \tilde S$ as sets with the ordering
$$S^{(0)} = \{1^A_1 < \cdots < 0^A_{i-1} < 1^B_j < 0^A_{i+1} < \cdots
< 1^B_{j-1} < 0^A_i < 1^B_{j+1} < \cdots < 0^B_{h_2}\}.$$
In Definition~\ref{DefOfTAn} and Definition~\ref{DefOfTBn},
we construct a sequence $(S^{(0)}, \delta, \pi), (S^{(1)}, \delta, \pi), \dots$ of ABS's.  
We will show that 
there exists a non-negative integer $n$ such that
the ABS $(S^{(n)}, \delta, \pi)$ coincides with the specialization of $S$
obtained by exchanging $0^A_i$ and $1^B_j$.
 
\begin{definition} \label{DefOfTAn}
For the ABS $(S^{(0)}, \delta, \pi)$,
we define a set $\mathcal A^{(0)}$ by
$$\mathcal A^{(0)} = 
\{t \in S_A^{(0)} \mid t < 0^A_i \text{ and } \pi(0^A_i) < \pi(t)
\text{ in } S^{(0)} \text{ with } \delta(t) = 0\},$$
where $S_A^{(0)}$ is a subset of $S^{(0)}$
consisting of all elements of $A$.
For non-negative integers $n$, 
put $\alpha_n = \pi^n(0^A_i)$.
For a natural number $n$,
we construct an ordered set $S^{(n)}$ and a set $\mathcal A^{(n)}$ 
from the ABS $(S^{(n-1)}, \delta, \pi)$ and the set $\mathcal A^{(n-1)}$
if $\mathcal A^{(n-1)}$ is not empty.
Let $S^{(n)} = S^{(n-1)}$ as sets.
We define the order on $S^{(n)}$ to be
for $t < t'$ in $S^{(n-1)}$,
we have $t > t'$ in $S^{(n)}$ if and only if 
$\alpha_n < t' \leq \pi(t_{\rm max})$ in $S^{(n-1)}$ and $t = \alpha_n$.
%\begin{itemize}
%\item[(i)] $\pi(t_{\rm max}) < \alpha_n$
%and there exists no element between these elements, 
%where $t_{\rm max}$ is the maximum element of $\mathcal A^{(n-1)}$,
%\item[(ii)] for $t, t' \in S^{(n)}$, if $t \neq \alpha_n$ and $t' \neq \alpha_n$,
%then $t < t'$ in $S^{(n)}$ if and only if $t < t'$ in $S^{(n-1)}$.
%\end{itemize}
We can regard $\pi$ as the map on $S^{(n)}$.
Thus we obtain an ABS $(S^{(n)}, \delta, \pi)$.
In other words, we define the ABS $(S^{(n)}, \delta, \pi)$ by
moving the element $\alpha_n$ to the right of $\pi(t_{\text{max}})$ in $S^{(n-1)}$.
We define 
$$\mathcal A^{(n)} = 
\{t \in S_A^{(n)} \mid t < \alpha_n \text{ and } \alpha_{n+1} < \pi(t)
\text{ in } S^{(n)} \text{ with } \delta(t) = \delta(\alpha_n)\},$$
where $S_A^{(n)}$ is the subset of $S^{(n)}$
consisting of all elements of $A$.
\end{definition}

We will see that there exists a non-negative integer $a$
satisfying $\mathcal A^{(a)} = \emptyset$ in Proposition~\ref{PropOfEmptyA}.
For this number $a$, we introduce a definition of sets $\mathcal B^{(n)}$ as follows.

\begin{definition} \label{DefOfTBn}
We define the set 
$$\mathcal B^{(0)} = \{t \in S^{(a)} \mid 1^B_j < t \text{ and } \pi(t) < \pi(1^B_j) 
\text{ in } S^{(a)} \text{ with } \delta(t) = 1\},$$
and let $I$ be the subset of $\mathcal B^{(0)}$ consisting of elements 
$t$ which are of the form $t  =1^A_x$ with a natural number $x$.
Note that for an element $t$ of $I$,
there exists a non-negative integer $n$ with $n < a$ such that $t = \alpha_n$.
We have $1^B_j < 1^B_z$ and $\pi(1^B_z) < \pi(1^B_j)$
for natural numbers $z$ with $z < j$ in $S^{(a)}$.
Thus the set $\mathcal B^{(0)}$ is determined as $I \cup \{1^B_1, \dots, 1^B_{j-1}\}$.
Set $\beta_n = \pi^n(1^B_j)$ for non-negative integers $n$.
We construct the ordered set $S^{(a+n)}$ for natural numbers $n$ inductively.
For the ABS $(S^{(a+n-1)}, \delta ,\pi)$ 
and the set $\mathcal B^{(n-1)}$,
if $\mathcal B^{(n-1)}$ is not empty, then
let $S^{(a+n)} = S^{(a+n-1)}$ as sets.
The ordering of $S^{(a + n)}$ is given so that
for $t < t'$ in $S^{(a+n-1)}$, 
we have $t > t'$ in $S^{(a+n)}$ if and only if 
$\pi(t_{\rm min}) \leq t < \beta_n$ in $S^{(a+n-1)}$
and $t' = \beta_n$.
%\begin{itemize}
%\item[(i)] $\beta_n < \pi(t_{\rm min})$
%and there exists no element between these elements,
%where $t_{\rm min}$ is the minimum element of $\mathcal B^{(n-1)}$,
%\item[(ii)] for $t, t' \in S^{(a + n)}$, if $t \neq \beta_n$ and $t' \neq \beta_n$,
%then $t < t'$ in $S^{(a + n)}$ if and only if $t < t'$ in $S^{(a + n-1)}$.
%\end{itemize}
We can regard $\pi$ as the map on $S^{(a+n)}$.
Thus we obtain an ABS $(S^{(a+n)}, \delta, \pi)$.
In other words, we obtain the ABS $(S^{(a+n)}, \delta, \pi)$ by
moving the element $\beta_n$ to the left of $\pi(t_{\rm min})$ in $S^{(a+n-1)}$.
We define 
$$\mathcal B^{(n)} = \{t \in S^{(a+n)} \mid \beta_n < t \text{ and } 
\pi(t) < \beta_{n+1} 
\text{ in } S^{(a+n)} \text{ with } \delta(t) = \delta(\beta_n)\}.$$
\end{definition}

If there exists a non-negative integer $b$ such that
$\mathcal B^{(b)} = \emptyset$, then
we call the ABS $(S^{(a + b)}, \delta, \pi)$ 
the {\it full modification by $0^A_i$ and $1^B_j$}.
%For the ABS $S = (\tilde S, \delta, \pi_0)$ associated to a ${\rm DM_1}$ $N_\xi$,
%we obtain the full modification by $0^A_i$ and $1^B_j$
%from the sequence of ABS's
%$\{(S^{(n)}, \delta, \pi)\}_n$,
%where $\pi$ is the small modification by $0^A_i$ and $1^B_j$.
If there exists the full modification by $0^A_i$ and $1^B_j$, 
then the specialization of $S$ obtained by exchanging $0^A_i$ and $1^B_j$
is given by this full modification.
We show that there exists a non-negative integer $b$
such that $\mathcal B^{(b)} = \emptyset$ 
in Proposition~\ref{PropOfEmptyB}.

\begin{example}\label{ExOfExchange}
Let us see an example of constructing $S'$ from $S$.
Let $N = N_{2,7} \oplus N_{5,3}$ be the ${\rm DM_1}$,
and let $S = (\tilde S, \delta, \pi)$ denote the ABS associated to $N$.
In Example~\ref{ExOfDirectSum}, we obtain 
the diagram of $S$.
Construct the small modification by $0^A_6$ and $1^B_3$.
Then we obtain the ABS $(S^{(0)}, \delta, \pi)$, and
the diagram of this ABS is described as
$$(S^{(0)}, \delta, \pi) :
\xymatrix@=1pt{1^A_1\ar@/_20pt/[rrrrrrrrrr] & 1^A_2\ar@/_20pt/[rrrrrrrrrr] &
0^A_3\ar@/_20pt/[ll] & 0^A_4\ar@/_20pt/[ll] & 0^A_5\ar@/_20pt/[ll] &
1^B_3\ar@/_30pt/[rrrrrrrrr]|\circ & 
0^A_7\ar@/_20pt/[ll]|\times &
1^B_1\ar@/_20pt/[rrrrr]|\times & 1^B_2\ar@/_20pt/[rrrrr]|\times &
0^A_6\ar@/_30pt/[llllll]|\circ & 0^A_8\ar@/_20pt/[lllll] &
0^A_9\ar@/_20pt/[lllll] &1^B_4\ar@/_20pt/[rrr] &
1^B_5\ar@/_20pt/[rrr] & 0^B_6\ar@/_20pt/[lllllll] &
0^B_7\ar@/_20pt/[lllllll] & 0^B_8\ar@/_20pt/[lllllll]}.$$
\\
First, let us construct sets $\mathcal A^{(n)}$.
We have the set $\mathcal A^{(0)} = \{0^A_7\}$.
To construct the ordered set $S^{(1)}$,
we move the element $0^A_4$ to the right of $0^A_5$,
and we obtain the following diagram of $(S^{(1)}, \delta, \pi)$:\\
$$(S^{(1)}, \delta, \pi) :
\xymatrix@=1pt{1^A_1\ar@/_20pt/[rrrrrrrrrr] & 1^A_2\ar@/_20pt/[rrrrrrrrrr] &
0^A_3\ar@/_20pt/[ll] & 0^A_5\ar@/_20pt/[l]|\times & 
0^A_4\ar@/_30pt/[lll]|\circ &
1^B_3\ar@/_30pt/[rrrrrrrrr]|\circ & 
0^A_7\ar@/_20pt/[lll] &
1^B_1\ar@/_20pt/[rrrrr]|\times & 
1^B_2\ar@/_20pt/[rrrrr]|\times &
0^A_6\ar@/_20pt/[lllll] & 0^A_8\ar@/_20pt/[lllll] &
0^A_9\ar@/_20pt/[lllll] &1^B_4\ar@/_20pt/[rrr] &
1^B_5\ar@/_20pt/[rrr] & 0^B_6\ar@/_20pt/[lllllll] &
0^B_7\ar@/_20pt/[lllllll] & 0^B_8\ar@/_20pt/[lllllll]}.$$
\\
We have the set $\mathcal A^{(1)} = \{0^A_5\}$.
In the same way, we get $S^{(2)}$
by moving the element $1^A_2$ to the right of $0^A_3$,
and we see that $\mathcal A^{(2)}$ is empty.
Hence we obtain $a = 2$.
Next, let us construct sets $\mathcal B^{(n)}$.
We have $\mathcal B^{(0)} = \{1^B_1, 1^B_2\}$.
We move the element $0^B_6$ to the left of $1^B_4$
which is the minimum element of $\pi(\mathcal B^{(0)})$
to construct $S^{(3)}$.
Then $\mathcal B^{(1)} = \emptyset$.
Hence we get an ABS $S'$ by $(S^{(3)}, \delta, \pi)$.
Thus the diagram of the specialization $S'$ is described as follows.\\
\\
$$S' : \xymatrix@=1pt{1^A_1\ar@/_20pt/[rrrrrrrrrr] & 0^A_3\ar@/_20pt/[l] &
1^A_2\ar@/_20pt/[rrrrrrrrr] & 0^A_5\ar@/_20pt/[ll] & 
0^A_4\ar@/_20pt/[ll] &
1^B_3\ar@/_20pt/[rrrrrrr] & 
0^A_7\ar@/_20pt/[lll] &
1^B_1\ar@/_20pt/[rrrrrr] & 
1^B_2\ar@/_20pt/[rrrrrr] &
0^A_6\ar@/_20pt/[lllll] & 0^A_8\ar@/_20pt/[lllll] &
0^A_9\ar@/_20pt/[lllll] &0^B_6\ar@/_20pt/[lllll] &
1^B_4\ar@/_20pt/[rr] & 1^B_5\ar@/_20pt/[rr] &
0^B_7\ar@/_20pt/[lllllll] & 0^B_8\ar@/_20pt/[lllllll]}.$$
\\
\\
One can check that these admissible ABS's $S'$ and $S$ 
satisfy that $\ell(S') = \ell(S) - 1$.
Moreover, one can see that for the same notation as Remark~\ref{RemOfww'},
the permutation $\varepsilon$ satisfies
$\varepsilon^{-1} = (13, 14, 15)(2, 3)(4, 5)$ in this case.
\end{example}

%\begin{remark} \label{RemOfSS'ww'}
%Suppose that there is the smallest integer $a$ (resp. $b$) such that 
%$\mathcal A^{(a)}$ (resp. $\mathcal B^{(b)}$) is empty.
%Then for the same notation as Remark~\ref{RemOfww'},
%there exist permutations $\mathfrak a_0, \dots, \mathfrak a_a, 
%\mathfrak b_0, \dots, \mathfrak b_b$ such that 
%$\varepsilon^{-1} 
%= \mathfrak b_b \cdots \mathfrak b_0 \mathfrak a_a \cdots \mathfrak a_0$,
%and $\mathfrak a_n$ (resp. $\mathfrak b_n$) corresponds to 
%constructing $S^{(n)}$ to $S^{(n+1)}$ (resp. $S^{(a+n)}$ to $S^{(a+n+1)}$).
%In other words,
%the ordered set $S^{(n+1)}$ (resp. $S^{(a+n+1)}$) is obtained by 
%the coordinate transformation of $S^{(n)}$ (resp. $S^{(a+n)}$)
%by $\mathfrak a_n$ (resp. $\mathfrak b_n$).
%For instance, in the case Example~\ref{ExOfExchange},
%$\varepsilon$ satisfies $\varepsilon^{-1} = (13, 14, 15)(2, 3)(4, 5)$.
%\end{remark}
%% 上のRemarkを追加

To construct full modifications of $S$,
%for the ABS $S$ corresponding to $N_\xi$,
we show some properties of ordered sets $S^{(n)}$
and sets $\mathcal A^{(n)}$
in Lemma~\ref{LemOfUniA}, Proposition~\ref{PropOfT},
Corollary~\ref{CoroOfT} and Proposition~\ref{PropOf1}.
For a Newton polygon $\xi = (m_1, n_1) + (m_2, n_2)$ consisting of 
two segments satisfying $\lambda_2 < 1/2 < \lambda_1$,
we denote by $S = (\tilde S, \delta, \pi_0)$ the ABS 
associated to the ${\rm DM_1}$ $N_\xi$.
Let $A$ (resp. $B$) be the ABS 
corresponding to the ${\rm DM_1}$ $N_{m_1, n_1}$ (resp. $N_{m_2, n_2}$).
We denote by $\tilde A$ (resp. $\tilde B$) the ordered set of $A$ (resp. $B$). 
Let $\pi$ denote the small modification by elements 
$0^A_i$ and $1^B_j$ of $\tilde S$ with $m_1 < i \leq n_1$.
Then we obtain sets $\mathcal A^{(n)}$ 
%ABS's $(S^{(n)}, \delta, \pi)$ 
for non-negative integers $n$.
%We often regard $\tilde A$ or $\tilde B$ as a subset of sets $S^{(n)}$.
We will show that there exists a non-negative integer $n$
such that $\mathcal A^{(n)} = \emptyset$ in Proposition~\ref{PropOfEmptyA}.

\begin{lemma}\label{LemOfUniA}
For every non-negative integer $n$, 
let
$\Sigma_{(n)}$ be the set consisting of elements $t$ of $\tilde A$
satisfying that
$t' < t$ and $\pi(t) < \pi(t')$ in $S^{(n)}$
for some elements $t'$ of $S^{(n)}$ with $\delta(t) = \delta(t')$.
If $\mathcal A^{(n)}$ is not empty, then $\Sigma_{(n)} = \{\alpha_n\}$,
where $\alpha_n = \pi^n(0^A_i)$.
Moreover, if $\mathcal A^{(n)} = \emptyset$, then $\Sigma_{(n)} = \emptyset$.
\end{lemma}

\begin{proof}
Let us show the assertion by induction on $n$.
The case $n = 0$ is obvious.
For a natural number $n$,
suppose that $\Sigma_{(n-1)} = \{\alpha_{n-1}\}$.
The statement follows from the construction of the set $S^{(n)}$.
In fact, by the construction of $S^{(n)}$,
we have $\pi(t_{\rm max}) < \alpha_n$ in $S^{(n)}$,
where $t_{\rm max}$ is the maximum element of $\mathcal A^{(n-1)}$.
It implies that $\pi(t) < \alpha_{n}$ holds 
for all elements $t$ of $\mathcal A^{(n-1)}$ in $S^{(n)}$.
Then $\alpha_{n-1}$ does not belong to $\Sigma_{(n)}$.
Thus we see that elements $t''$ of $\tilde A \setminus \{\alpha_n\}$
do not belong to $\Sigma_{(n)}$,
and we obtain $\Sigma_{(n)} = \{\alpha_n\}$.
\end{proof}

\begin{proposition}\label{PropOfT}
For all non-negative integers $m$,
set $\alpha_m = \pi^m(0^A_i)$.
Let $n$ be a natural number.
Then $\mathcal A^{(n)}$ is obtained by
$$\mathcal A^{(n)} = \{\pi(t) \in \tilde A \mid t \in \mathcal A^{(n-1)}\ {\rm and}\ 
\delta(\pi(t)) = \delta(\alpha_n)\}.$$
%where $S_A^{(n)}$ is a subset of $S^{(n)}$ consisting of all elements of $A$. 
\end{proposition}

\begin{proof}
We fix a natural number $n$.
Choose an element $t$ of $\mathcal A^{(n-1)}$
satisfying that $\delta(\pi(t)) = \delta(\alpha_n)$.
Let us see that if $\pi(t)$ belongs to $\tilde A$,
then the set $\mathcal A^{(n)}$ contains $\pi(t)$.
Since $\alpha_n < \pi(t)$ in $S^{(n-1)}$,
we have $\alpha_{n+1} < \pi(\pi(t))$ in $S^{(n-1)}$ and $S^{(n)}$.
By the construction of $S^{(n)}$, we have $\pi(t) < \alpha_n$ in $S^{(n)}$.
Hence the set $\mathcal A^{(n)}$ contains $\pi(t)$.
Conversely, let $t'$ be an element of $\mathcal A^{(n)}$.
%Let us see that there exists an element $t'$ of $\mathcal A^{(n-1)}$
%such that $\pi(t') = t$.
%We immediately obtain an element $t'$ of $S^{(n)}$
%such that $\pi(t') = t$ since the map $\pi$ is bijective.
We denote by $t$ the element of $S^{(n)}$ satisfying that $\pi(t) = t'$.
We have to see that $t$ belongs to $\mathcal A^{(n-1)}$.
In the sets $S^{(n)}$ and $S^{(n-1)}$,
we have $t < \alpha_{n-1}$.
Moreover, since $\alpha_{n+1} < \pi(\pi(t))$ holds
in $S^{(n)}$ and $S^{(n-1)}$,
we get $\alpha_n < \pi(t)$ in $S^{(n-1)}$.
Hence $t$ belongs to $\mathcal A^{(n-1)}$.
\end{proof}

Applying Proposition~\ref{PropOfT} repeatedly,
we obtain the following:

\begin{corollary}\label{CoroOfT}
Let $n$ be a natural number.
Then $\mathcal A^{(n)}$ is described as follows:
$$\mathcal A^{(n)} = \{\pi^n(t) \mid t \in \mathcal A^{(0)},\ 
\delta(\pi^m(t)) = \delta(\alpha_m) \text{ and } \pi^m(t) \in \tilde A
\text{ for all } m \text{ with } 0 \leq m \leq n\}.$$
\end{corollary}

\begin{proposition}\label{PropOf1}
Let $n$ be a non-negative integer.
Then the set $\mathcal A^{(n)}$ does not contain
an element $t$
which is of the form $t = \pi^m(0^A_i)$
with a non-negative integer $m$
for $m \leq n$.
\end{proposition}

\begin{proof}
For every non-negative integer $n$,
the set $\mathcal A^{(n)}$ consists of some elements of $\tilde A$.
Hence sets $\mathcal A^{(n)}$ do not contain the element $\pi^{-1}(0^A_i)$
which is an element of $\tilde B$.
We show the assertion by induction on $n$.
The case $n = 0$ is obvious.
For a natural number $n$,
assume that the set $\mathcal A^{(n-1)}$ contain no element
$\pi^m(0^A_i)$ with $m \leq n-1$.
We consider the set $\mathcal A^{(n)}$.
Suppose that $\mathcal A^{(n)}$ contains an element $\pi^m(0^A_i)$ for $m \leq n$.
Then by Proposition~\ref{PropOfT}, 
$\pi^{m-1}(0^A_i)$ belongs to $\mathcal A^{(n-1)}$.
This contradicts with the hypothesis of induction.
\end{proof}

\begin{proposition}\label{PropOfEmptyA}
For every small modification by $0^A_i$ and $1^B_j$
with $m_1 < i \leq n_1$,
there exists a non-negative integer $a = a(i, j)$ such that
$\mathcal A^{(a)} = \emptyset$.
\end{proposition}

\begin{proof}
In the case of $i = n_1$ and $1 \leq j \leq n_2$,
we immediately have $\mathcal A^{(0)} = \emptyset$.
Let us see the other cases.
By the hypothesis $m_1 < i \leq n_1$,
the map $\pi$ sends $0^A_{i+m_1}$ to $1^B_j$
in $S^{(n)}$ for all non-negative integers $n$.
Let $m$ be the minimum natural number satisfying 
$\pi^m(0^A_i) = 1^A_{m_1}$.
If there exists a natural number $n$
such that $\mathcal A^{(n-1)} = \{0^A_{i+m_1}\}$ for $n < m$,
we have then $\mathcal A^{(n)} = \emptyset$ by definition of $\mathcal A^{(n)}$.
Assume that sets $\mathcal A^{(0)}, \dots, \mathcal A^{(m-1)}$ are not empty.
Let us consider the set $\mathcal A^{(m)}$.
Note that $1^A_{m_1}$ is the maximum element of the set
$\{t \in A \mid \delta(t) = 1\}$.
Proposition~\ref{PropOf1} induces that 
every element $t$ of $\mathcal A^{(m-1)}$ satisfies that 
$\delta(\pi(t)) = 0$ or $\pi(t) = 1^B_j$,
whence we see $\mathcal A^{(m)} = \emptyset$.
\end{proof}

Thus by a small modification $\pi$ by $0^A_i$ and $1^B_j$,
we obtain the non-negative integer $a = a(i, j)$ satisfying that 
$\mathcal A^{(a)} = \emptyset$.
By the ABS $(S^{(a)}, \delta, \pi)$, we obtain  
the set $\mathcal B^{(0)} = \{1^B_1, \dots, 1^B_{j-1}\} \cup I$,
where $I$ is the subset of $S^{(a)}$ consisting of elements 
$\alpha_n =\pi^n(0^A_i)$
satisfying that $1^B_j < \alpha_n$ and $\alpha_{n+1} < \pi(1^B_j)$.
Here, we show some properties of sets $\mathcal B^{(0)}, \mathcal B^{(1)}, \dots$
in Lemma~\ref{LemOfUniB}, Proposition~\ref{PropOfInducB},
Proposition~\ref{PropOf0}, Lemma~\ref{LemOfI}
and Lemma~\ref{LemOfSubB}
to see that
there exists a non-negative integer $n$
such that $\mathcal B^{(n)} = \emptyset$
in Proposition~\ref{PropOfEmptyB}.
We suppose that if $n_1 = m_1 + 1$ and $m_2 = n_2 + 1$,
then $(i, j) \neq (n_1, m_2)$.

\begin{lemma} \label{LemOfUniB}
Let $n'$ be a non-negative integer with $n' \geq  a$.
Set $n = n' - a$.
Let $\Sigma'_{(n)}$ be the set
consisting of elements $t$ of $\tilde S$ satisfying that
$t < t'$ and $\pi(t') < \pi(t)$ in $S^{(n')}$
for some elements $t'$ of $S^{(n')}$ with $\delta(t) = \delta(t')$.
If $\mathcal B^{(n)}$ is not empty,
we have then $\Sigma'_{(n)} = \{\beta_n\}$,
where $\beta_n = \pi^n(1^B_j)$.
Moreover, if $\mathcal B^{(n)} = \emptyset$, then $\Sigma'_{(n)} = \emptyset$.
\end{lemma}
%% 記号の紛らわしさを避けるため，Sigma_n から Sigma_{(n)}に変更．

\begin{proof}
A proof is given by the same way as Lemma~\ref{LemOfUniA}.
\end{proof}

\begin{proposition} \label{PropOfInducB}
Set $\beta_m = \pi^m(1^B_j)$ for all non-negative integers $m$.
Let $n$ be a natural number.
The set $\mathcal B^{(n)}$ is obtained by
$$\mathcal B^{(n)} = \{\pi(t) \in S^{(a+n)} \mid t \in \mathcal B^{(n-1)}\ {\rm and}\ 
\delta(\pi(t)) = \delta(\beta_n)\}.$$
Moreover, this set is described as
$$\mathcal B^{(n)} = \{\pi^n(t) \mid t \in \mathcal B^{(0)}\ {\rm and}\ 
\delta(\pi^m(t)) = \delta(\beta_m) \text{ for all }m \text{ with }  
0 \leq m \leq n\}.$$
\end{proposition}

\begin{proof}
A proof is given by the same way as Proposition~\ref{PropOfT}.
\end{proof}

\begin{proposition}\label{PropOf0}
Let $n$ be a non-negative integer.
If the set $\mathcal B^{(n)}$ is a subset of $B$,
then $\mathcal B^{(n)}$ does not contain
an element $t$
which is of the form $t = \pi^m(1^B_j)$
with a non-negative integer $m$
for $m \leq n$.
\end{proposition}

\begin{proof}
A proof is given by the same way as Proposition~\ref{PropOf1}.
\end{proof}

\begin{comment}

\begin{proposition}
Let $n$ be a non-negative integer
with $\mathcal B^{(n)} \neq \emptyset$.
Let $t$ be an element of $S^{(a+n)}$.
If there exist elements $t'$ and $t''$ of $\mathcal B^{(n)} \cup \{\pi^n(1^B_j)\}$
such that $\pi(t') < \pi(t) < \pi(t'')$, 
then $t \in \mathcal B^{(n)}$.
\end{proposition}

\begin{proof}
A proof is given by the same way as Proposition~\ref{PropOf\mathcal ASaturated}.
\end{proof}

\end{comment}

\begin{lemma}\label{LemOfI}
If $1 \leq j \leq n_2$, 
then all elements $t$ of $I$ satisfy that $\delta(\pi^2(t)) = 0$.
\end{lemma}

\begin{proof}
If $m_1$ and $n_1$ satisfy $n_1-m_1 > m_1$,
then $\delta(\pi^2(t)) = 0$ holds
for all elements $t$ of $A$ with $\delta(t) = 1$.
It suffices to see the case of $n_1 -m _1 < m_1$.
Assume that there exists an element $t$ of $I$
satisfying $\delta(\pi^2(t)) = 1$.
We recall that there exists a natural number $n$ such that 
$t = \pi^n(0^A_i)$.
On the other hand, 
we can denote by $t = 1^A_r$ this element $t$ with a natural number $r$.
Put $r' = r+n_1-m_1$.
We have then $\pi^2(1^A_r) = 1^A_{r'}$.
By the definition of $I$, 
the set $\mathcal A^{(n-1)}$ contains the inverse image of $1^B_j$.
By construction, we have $|\mathcal A^{(m)}| < n_1 - m_1$ for all $m$.
In $S^{(n-1)}$, the number of elements between $1^A_r$ and $1^B_j$
is greater than $|\mathcal A^{(n-1)}|$,
whence we have $1^A_r < 1^B_j$ in $S^{(n)}$.
This contradict with $1^A_r \in I$.
%whence $\mathcal A^{(n)}$ does not contain $1^A_{r'}$.
%It implies that $1^B_j < 1^A_r < 1^A_{r'} < 1^A_{m_1}$ holds
%in the ABS $S^{(a+n)}$.
%Hence $0^A_i < 1^A_{r'}$ holds in the ABS $A$.
%It is a contradiction.
%
%S^{(n-1)}において1^A_r < 1^A_r' < 1^B_jとなっている．
%1^A_rと1^A_r'の間の項が動かなければ，1^A_rと1^B_jの間にある項数はn_1-m_1個
%であり，S^{(n)}を作る際に1^A_rが1^B_jの右側に行くことはない．
%また，1^A_rと1^A_r'の間の項が1^B_jの左側に動いたとしても同様に議論できる．
%1^A_rと1^A_r'の間の項が1^B_jの右側に動いて，間の項が減ったとすると，
%右側に動いていることから\mathcal A^{(n-1)}の項数も減っているのでOK.
\end{proof}

\begin{lemma}\label{LemOfSubB}
Let $n$ be a natural number with $n > 1$.
Suppose $1 \leq j \leq n_2$.
Then $\mathcal B^{(n)}$ is a subset of $B$.
\end{lemma}

\begin{proof}
First, let us see that sets $\mathcal B^{(n)}$ do not contain $0^A_i$ for all $n$.
Assume that a set $\mathcal B^{(n-1)}$ contains 
the inverse image of $0^A_i$.
Note that the inverse image of $0^A_i$ is obtained by 
$0^B_{j+m_2}$ in this hypothesis.  
We have then $\delta(\pi^{n-1}(1^B_j)) = 0$ by the definition of $\mathcal B^{(n-1)}$.
The condition $\lambda_2 < 1/2$ 
induces that $\delta(\pi^n(1^B_j)) = 1$.
It implies that the element $0^A_i$ does not belong to $\mathcal B^{(n)}$.

To show the statement of this lemma,
let us consider two cases depending on the value of $\delta(\pi(1^B_j))$.
On the one hand, suppose that $\delta(\pi(1^B_j)) = 1$,
and let us see that $\mathcal B^{(n)}$ consists of some elements of $B$ 
for $n \geq 1$.
We have 
$\mathcal B^{(1)} = \{\pi(t) \mid t \in \mathcal B^{(0)}\ {\rm with}\ \delta(\pi(t)) = 1\}$
in this assumption.
For elements $t$ of $I$, 
all elements $\pi(t)$ do not belong to $\mathcal B^{(1)}$
since $\delta(\pi(t)) = 0$ holds
for all elements $t$ of $\tilde A$ with $\delta(t) = 1$.
It induces that $\mathcal B^{(1)}$ is a subset of $B$,
whence we see that $\mathcal B^{(n)} \subset B$ for all natural numbers in this case.

On the other hand, suppose $\delta(\pi(1^B_j)) = 0$.
Let $\Lambda$ be the subset of $\mathcal B^{(1)}$ consisting of elements $t$ 
which belong to $\tilde B$.
Then $\mathcal B^{(1)}$ is obtained by the union of 
$\Lambda$ and $\pi(I)$.
By the condition of $m_2$ and $n_2$, for all elements $t$ of $\Lambda$, 
we have $\delta(\pi(t)) = 1$.
Moreover, by Lemma~\ref{LemOfI}, we have $\delta(\pi(t)) = 0$
for all elements $t$ of $\pi(I)$.
Thus we see that $\mathcal B^{(2)}$ is a subset of $B$ 
since $\delta(\pi^2(1^B_j)) = 1$ in this hypothesis.
Hence $\mathcal B^{(n)}$ is a subset of $\tilde B$
for every natural number $n$ with $n > 1$.
\end{proof}

By the above properties, we show Proposition~\ref{PropOfEmptyB}.
By this proposition, we see that specializations $S'$ of $S$
are obtained by the combinatorial method.

\begin{proposition}\label{PropOfEmptyB}
For the small modification by $0^A_i$ and $1^B_j$
with $m_1 < i \leq n_1$,
there exists a non-negative integer $b= b(i, j)$ such that 
$\mathcal B^{(b)} = \emptyset$.
\end{proposition}

\begin{proof}
First, let us consider the case $1 \leq j \leq n_2$.
If $j = 1$, then we have $\mathcal B^{(0)} = I$ and $\mathcal B^{(1)} = \emptyset$.
For $j > 1$, we have $\mathcal B^{(0)} = 
\{1^B_1, \dots, 1^B_{j-1}\} \cup I$.
Let $m$ be the minimum natural number satisfying 
$\pi^m(1^B_j) = 0^B_{m_2+1}$.
If $m = 1$, then $\mathcal B^{(1)} = \pi(I)$.
Since $\delta(\pi^2(1^B_j)) = 1$ in this case and Lemma~\ref{LemOfI},
we have $\mathcal B^{(2)} = \emptyset$.
Let us see the case $m > 1$.
Suppose that there exists a natural number $n$,
with $n < m$,
such that $\mathcal B^{(n)} = \{0^B_{j+m_2}\}$,
i.e., the set $\mathcal B^{(n)}$ consists of an element the inverse image of $0^A_i$.
By the proof of Lemma~\ref{LemOfSubB}, the element $0^A_i$
does not belong to $\mathcal B^{(n+1)}$, whence
we have $\mathcal B^{(n+1)} = \emptyset$.
Assume that $\mathcal B^{(0)}, \dots, \mathcal B^{(m-1)}$ are non-empty sets.
Note that $0^B_{m_2+1}$ is the minimum element of 
all elements $t$ of $\tilde B$ satisfying $\delta(t) = 0$.
%Proposition~\ref{PropOf0} implies that 
For an element $t$ of $\mathcal B^{(m-1)}$,
if $\pi(t)$ belongs to $\tilde B$, then $\delta(\pi(t)) = 1$ holds.
Moreover,  it follows from Lemma~\ref{LemOfSubB} 
that $\mathcal B^{(m)}$ does not contain an element of $\tilde A$.
Therefore we have $\mathcal B^{(m)} = \emptyset$.

Next, let us see the case $n_2 < j \leq m_2$.
We divide the proof into two cases depending on
whether $I$ is empty.
First, suppose $I = \emptyset$.
For the minimum non-negative integer $n$ satisfying that $\pi^n(1^B_j) = 1^B_1$,
we have $\mathcal B^{(n)} = \emptyset$.
In fact, if the set $\mathcal B^{(n-1)}$ is not empty,
then $\mathcal B^{(n-1)} = \{0^A_i\}$ holds
since $\pi^{n-1}(1^B_j) = 0^B_{m_2+1}$.
If $\delta(\pi(0^A_i)) = 1$,
then $I$ contains $\pi(0^A_i)$
since $\mathcal A^{(0)} = \{0^A_{i+1}, \dots, 0^A_{m_1+n_1}\}$
contains the inverse image of $1^B_j$.
This contradicts with the assumption. 
Hence we have $\delta(\pi(0^A_i)) = 0$, and then $\mathcal B^{(n)} = \emptyset$ holds.

Next, suppose $I \neq \emptyset$.
If $j > n_2+1$, then the non-negative integer $b$ is obtained by
the number $n$ satisfying that $\pi^n(1^B_j) = 1^B_{n_2+1}$.
In fact, $\mathcal B^{(n-1)}$ consists of some elements $t$ of $\tilde A$
since $\pi^{n-1}(1^B_j) = 1^B_1$ is the minimum element of $\tilde B$.
It is clear that $\delta(\pi(t)) = 0$ for all elements $t$ of $\mathcal B^{(n-1)}$,
and we have then $\mathcal B^{(n)} = \emptyset$.
If $j = n_2 + 1$,
then the non-negative integer $b$ is obtained by
the number $n$ satisfying $\pi^n(1^B_j) = 1^B_{n_2+2}$.
%In fact, for this $n$,
%the set $\mathcal B^{(n-1)}$ consists of $1^B_1$ and some elements of $I \subset A$.
%In this hypothesis, we have $\delta(\pi(1^B_1)) = 0$,
%whence we have $\mathcal B^{(n)} = \emptyset$. 
\end{proof}

%\begin{remark}
%Let us consider the case that $\xi = (m_1, n_1) + (m_2, n_2)$ satisfies
%$n_1 = m_1+1$ and $m_2 = n_2+1$ with $m_1 > 0$ or $n_2 > 0$.
%For $i = n_1$ and $j = m_2$, we have
%$\mathcal B^{(0)} = \{1^A_1, \dots, 1^A_{m_1}, 1^B_1, \dots, 1^B_{m_2}\}$.
%For two elements $t$ and $t'$ of $\mathcal B^{(0)}$
%and for every non-negative integer $n$,
%we have $\delta(\pi^n(t)) = \delta(\pi^n(t'))$.
%Hence there exists no non-negative integer $n$ 
%satisfying $\mathcal B^{(n)} = \emptyset$.
%In this case, the pair $(S^{(a)}, \delta)$ is associated to the specialization 
%$N' = (m_1 + m_2)N_{1, 1}$ of $N_\xi$.
%\end{remark}
%% For arbitrary elements を For two elementsに変更，
%% an arbitrary non-negative integer を for every non-negative integer に変更

For the Newton polygon $\xi$, we now suppose
$n_1 = m_1+1$ and $m_2 = n_2+1$
with $m_1 > 0$ or $n_2 > 0$.
If $m_1 < i \leq n_1$ and $n_2 < j \leq m_2$,
%with $m_1 > 0$ or $n_2 > 0$.
then $i$ and $j$ must be $i = m_1+1$ and $j = m_2$.
We construct the specialization by a well-known method
%% an well known method を a well-known method に変更
using binary expansions for such ABS's.
For the ABS $S = (\tilde S, \delta, \pi_0)$ corresponding to $N_\xi$,
%Let $S = (\tilde S, \delta, \pi_0)$ be the ABS
%corresponding to $N_\xi$. 
we have then the following diagram of $S$:\\
\\
$$\xymatrix@=5pt{1^A & \cdots &1^A
&\underline{0^A_{m_1+1}}\ar@/_20pt/[lll] & 1^B\ar@/_20pt/[rrrrrr] & \cdots &1^B
&0^A & \cdots &0^A\ar@/_20pt/[llllll]
&\underline{1^B_{m_2}}\ar@/_20pt/[rrr] & 0^B
& \cdots & 0^B}.$$
\\
%In the case (2), the elements $0^A_i$ and $1^B_j$ are uniquely determined by
%$0^A_{m_1+1}$ and $1^B_{m_2}$.
Note that the diagram satisfies that $\delta(t) \neq \delta(\pi_0^{-1}(t))$
for all elements $t$ except for these two elements.
Let $\pi$ be the small modification by $0^A_{m_1+1}$ and $1^B_{m_2}$,
and we obtain the admissible ABS $S' = (\tilde S', \delta, \pi)$.
By the construction of $S'$,
we have $\delta(t) \neq \delta(\pi(t))$ for all elements of $\tilde S'$.
Hence we obtain binary expansions of elements by
$b(t) = 0.1010\cdots$ if $\delta(t) = 0$
and $b(t) = 0.0101\cdots$ otherwise.
Therefore, the ABS $S'$ is associated to a ${\rm DM_1}$ $mN_{1,1}$
with $m = m_1+m_2$.
Note that this ABS satisfies that $\ell(S') < \ell(S) - 1$.
In this case, by the small modification, we have 
$\mathcal B^{(0)} = \{1^A_1, \dots, 1^A_{m_1}, 1^B_1, \dots, 1^B_{m_2}\}$.
For two elements $t$ and $t'$ of $\mathcal B^{(0)}$
and for every non-negative integer $n$,
we have $\delta(\pi^n(t)) = \delta(\pi^n(t'))$.
Hence there exists no non-negative integer $n$ 
satisfying $\mathcal B^{(n)} = \emptyset$.

Let $S$ be the ABS associated to $N_\xi$
with a Newton polygon $\xi = (m_1, n_1) + (m_2, n_2)$
satisfying that $\lambda_2 < 1/2 < \lambda_1$.
%In Section~\ref{NotaSpe},
For the case $\xi = (m_1, m_1+1) + (m_2, m_2-1)$,
we have seen that the specialization of $S$ obtained by 
exchanging $0^A_{m_1+1}$ and $1^B_{m_2}$
is associated to the ${\rm DM_1}$ $(m_1 + m_2)N_{1, 1}$.
Moreover, for the other cases, 
we obtain specializations by full modifications.
From now on, for a small modification $\pi$ by $0^A_i$ and $1^B_j$,
we denote by $a$ and $b$
the minimum non-negative integers satisfying that $\mathcal A^{(a)} = \emptyset$
and $\mathcal B^{(b)} = \emptyset$.
The specialization $S'$ obtained by exchanging $0^A_i$ and $1^B_j$
is equal to the ABS $(S^{(a + b)}, \delta, \pi)$.
%where $S$ is the ABS associated to the minimal ${\rm DM_1}$ $N_\xi$

\subsection{Some examples}
\label{SomeExamples}

In Example~\ref{ExOfExchange}, we introduced an example of 
full modifications by $0^A_i$ and $1^B_j$ with
$m_1 < i \leq n_1$ and $1 \leq j \leq n_2$.
Here, let us see examples of the case $n_2 < j \leq m_2$,
which are examples of the latter part of the proof of Proposition~\ref{PropOfEmptyB}.

\begin{example}
Let $N = N_{2,7} \oplus N_{5,3}$.
For the ABS $S = (\tilde S, \delta, \pi_0)$ associated to $N$,
we construct the small modification $\pi$ by $0^A_6$ and $1^B_5$.
Then we get $\mathcal A^{(0)} = \{0^A_7, 0^A_8, 0^A_9\}$ and $a = 2$.
The ABS $(S^{(2)}, \delta, \pi)$ is described as\\
\\
$$(S^{(2)}, \delta, \pi) : \xymatrix@=1pt{1^A_1\ar@/_20pt/[rrrrrrrrrr] & 0^A_3\ar@/_20pt/[l]
& 0^A_5\ar@/_20pt/[l] & 1^A_2\ar@/_20pt/[rrrrrrrr]
& 1^B_5\ar@/_35pt/[rrrrrrrrrrrr]|\circ & 0^A_7\ar@/_20pt/[lll]
& 0^A_4\ar@/_20pt/[lll] & 1^B_1\ar@/_20pt/[rrrrr]|\times
& 1^B_2\ar@/_20pt/[rrrrr]|\times & 1^B_3\ar@/_20pt/[rrrrr]|\times
& 0^A_8\ar@/_20pt/[llllll] & 0^A_9\ar@/_20pt/[llllll] 
& 1^B_4\ar@/_20pt/[rrr]|\times & 0^A_6\ar@/_20pt/[lllllll]
& 0^B_6\ar@/_20pt/[lllllll] & 0^B_7\ar@/_20pt/[lllllll]
& 0^B_8\ar@/_20pt/[lllllll]}.$$
Then we have $I = \emptyset$, and sets $\mathcal B^{(n)}$ are given by
$$\mathcal B^{(0)} = \{1^B_1, 1^B_2, 1^B_3, 1^B_4\}, 
\ \ \mathcal B^{(1)} = \{0^A_6, 0^B_6, 0^B_7\},
\ \ \mathcal B^{(2)} = \{1^B_1, 1^B_2\},
\ \ \mathcal B^{(3)} = \{0^A_6\},
\ \ \mathcal B^{(4)} = \emptyset.$$
The ABS $S'$ is obtained by the following:
$$S' = 1^A_1\ 0^A_3\ 0^A_5\ 1^A_2\ 1^B_5\ 0^A_7\ 1^B_3\ 
1^B_1\ 0^A_4\ 1^B_2\ 0^A_8\ 0^A_9\ 0^B_8\ 0^B_6\ 1^B_4\ 0^A_6\ 0^B_7.$$
\end{example}

\begin{example}
Let $N = N_{2,7} \oplus N_{5,3}$.
For the ABS $S = (\tilde S, \delta, \pi_0)$ associated to $N$,
we construct the small modification $\pi$ by $0^A_4$ and $1^B_4$.
Then we get $\mathcal A^{(0)} = \{0^A_5, 0^A_6, 0^A_7, 0^A_8, 0^A_9\}$
and $a = 1$.
The ABS $(S^{(1)}, \delta, \pi)$ is described as follows:\\
\\
$$(S^{(1)}, \delta, \pi) : \xymatrix@=1pt{1^A_1\ar@/_20pt/[rrrrrrrrrr]
& 0^A_3\ar@/_20pt/[l] & 1^B_4\ar@/_35pt/[rrrrrrrrrrrrr]|\circ
& 0^A_5\ar@/_20pt/[ll] & 0^A_6\ar@/_20pt/[ll]
& 0^A_7\ar@/_20pt/[ll] & 1^A_2\ar@/_20pt/[rrrrr]|\times
& 1^B_1\ar@/_20pt/[rrrrr]|\times & 1^B_2\ar@/_20pt/[rrrrr]|\times
& 1^B_3\ar@/_20pt/[rrrrr]|\times & 0^A_8\ar@/_20pt/[llllll] 
& 0^A_9\ar@/_20pt/[llllll] & 0^A_4\ar@/_20pt/[llllll]
& 1^B_5\ar@/_20pt/[rrr] & 0^B_6\ar@/_20pt/[lllllll]
& 0^B_7\ar@/_20pt/[lllllll] & 0^B_8\ar@/_20pt/[lllllll]}.$$
Then $I = \{1^A_2\}$, and we have sets
$$\mathcal B^{(0)} = \{1^A_2, 1^B_1, 1^B_2, 1^B_3\},\ \ 
\mathcal B^{(1)} = \{0^A_9, 0^A_4, 0^B_6\},\ \ 
\mathcal B^{(2)} = \{1^A_2, 1^B_1\},\ \ \mathcal B^{(3)} = \emptyset.$$
The ABS $S'$ is obtained by the following:
$$S' = 1^A_1\ 0^A_3\ 1^B_4\ 0^A_5\ 0^A_6\ 
1^B_2\ 0^A_7\ 1^A_2\ 1^B_1\ 1^B_3\ 0^A_8\ 0^B_7\ 
1^B_5\ 0^A_9\ 0^A_4\ 0^B_6\ 0^B_8.$$
\end{example}

\section{Classification of boundary components}
\label{ClassBoundaryComp}

Let $\xi$ be a Newton polygon consisting of two segments satisfying 
$\lambda_2 < 1/2 < \lambda_1$.
In this section, for the arrowed binary sequence $S$ 
associated to the minimal ${\rm DM_1}$ $N_\xi$,
we characterize specializations $S'$ of $S$
satisfying $\ell(S') = \ell(S) - 1$,
i.e., we classify boundary components of central streams. 
Moreover, in Section~\ref{SomePropOfBoundaryComp},
we show some properties of generic specializations.
We will give a proof of Theorem~\ref{ThmOfzetaxi} in Section~\ref{ConstGoodSpe}
using these properties.

\subsection{Criterion of boundary components}
\label{CriterionOfBoundaryComp}

Now, we state the first result (Theorem~\ref{ThmOfCentralClassification}).
%We fix notation.
%Let $N_{\xi}$ be a minimal ${\rm DM_1}$,
%where $\xi = (m_1, n_1) + (m_2, n_2)$ 
%with $\lambda_2< 1/2 < \lambda_1$,
%%上を以下に変更
Let $\xi = (m_1, n_1) + (m_2, n_2)$ be a Newton polygon
with $\lambda_2< 1/2 < \lambda_1$,
and let $N_\xi$ be the minimal ${\rm DM_1}$ associated to $\xi$.
Let $S = A \oplus B$ be the ABS
corresponding to $N_\xi$,
where $A$ (resp. $B$) is the ABS
corresponding to $N_{m_1,n_1}$ (resp. $N_{m_2,n_2}$).
Put $(\tilde S, \delta, \pi_0) = S$.
Set $h_1 = m_1+ n_1$ and $h_2 = m_2 + n_2$.
%Let $0^A_i$ (resp. $1^B_j$) be the $i$-th element of $A$ 
%(resp. $j$-th element of $B$).
We construct the small modification $\pi$ 
by $0^A_i$ with $m_1 < i \leq n_1$ and $1^B_j$ with $1 \leq j \leq m_2$.
Then for the notation of Section~\ref{ConSpeAbs}
we obtain sets $\mathcal A^{(n)}$ for $0 \leq n \leq a$, 
sets $\mathcal B^{(n)}$ for $0 \leq n \leq b$
and ABS's $(S^{(n)}, \delta, \pi)$ for $0 \leq n \leq a + b$,
where $a$ and $b$ are the smallest non-negative integers satisfying that
$\mathcal A^{(a)} = \emptyset$ and $\mathcal B^{(b)} = \emptyset$.
The main result of this section is

\begin{theorem}\label{ThmOfCentralClassification}
%Let $S$ be the ABS associated with $N_{\xi}$,
%where $\xi = (m_1, n_1) + (m_2, n_2)$ with 
%$\lambda_2 < 1/2 < \lambda_1$.
%Let $S'$ be a ABS obtained by exchanging 
%$0^A_i$ and $1^B_j$.
Let $\xi = (m_1, n_1) + (m_2, n_2)$ be a Newton polygon 
with $\lambda_2 < 1/2 < \lambda_1$.
Let $S = (\tilde S, \delta, \pi_0)$ be the ABS associated to the ${\rm DM_1}$ $N_\xi$.
Let $S' = (\tilde S', \delta, \pi)$ be the ABS obtained by exchanging $0^A_i$ and $1^B_j$.
Then $\ell(S') = \ell(S) - 1$ holds if and only if
there exists no non-negative integer $n$ such that
$\mathcal A^{(n)}$ contains the inverse image of $1^B_j$ or
$\mathcal B^{(n)}$ contains the inverse image of $0^A_i$ for the small modification $\pi$.
\end{theorem}

This theorem gives a classification of generic specializations of $H(\xi)$.
To show the above, we divide the problem into three cases
depending on conditions of $i$ and $j$ as follows.

\begin{definition} \label{DefOfH123}
For the ABS $S = A \oplus B$ associated to the minimal ${\rm DM_1}$ $N_\xi$
with the Newton polygon $\xi = (m_1, n_1) + (m_2, n_2)$,
we denote by $S' = S'(i, j)$ the specialization
obtained by exchanging $0^A_i$ and $1^B_j$.
We define sets
\begin{eqnarray*}
\mathcal H_1(S) &=& 
\{S'(i, j) \mid m_1 < i \leq n_1 \text{ and }1 \leq j \leq n_2\}, \\
\mathcal H_2(S) &=& 
\{S'(i, j) \mid m_1 < i \leq n_1 \text{ and }n_2 < j \leq m_2\}, \\
\mathcal H_3(S) &=&
\{S'(i, j) \mid n_1 < i \leq h_1 \text{ and }n_2 < j \leq m_2\}.
\end{eqnarray*}
\end{definition}

Let $T$ be the ABS associated to a ${\rm DM_1}$ $N$,
and let $T_D$ be the ABS associated to the dual $N^D$ of $N$.
We call this ABS {\it dual ABS of $T$}.
For the above notation, we will give 
a concrete condition of $i$ and $j$ 
satisfying that the exchange of $0^A_i$ and $1^B_j$ is a good.
Theorem~\ref{ThmOfCentralClassification} follows from this proposition:

\begin{proposition}\label{PropOfCentralClassification}
The following holds:
\begin{itemize}
\item[(i)] For $S' \in \mathcal H_1(S)$, 
the formula $\ell(S') = \ell(S) - 1$ holds if and only if
there exists no non-negative integer $n$ satisfying that 
$\mathcal A^{(n)}$ contains $0^A_{i+m_1}$ or $\mathcal B^{(n)}$ contains $0^B_{j+m_2}$.
\item[(ii)] For $S' \in \mathcal H_2(S)$,
we have $\ell(S') < \ell(S) - 1$.
\item[(iii)]For $S' \in \mathcal H_3(S)$,
the formula $\ell(S') = \ell(S) - 1$ holds if and only if
the dual ABS's $S_D$ and $S_D'$ satisfy $\ell(S_D') = \ell(S_D) - 1$. 
\end{itemize}
\end{proposition}

As an example, for an element $S'$ of $\mathcal H_1(S)$,
we immediately see that 
$0^A_{i+m_1}$ and $0^B_{j+m_2}$ are
inverse images of $1^B_j$ and $0^A_i$ respectively.

If the Newton polygon $\xi = (m_1, n_1) + (m_2, n_2)$ 
satisfies that $n_1 = m_1 + 1$ and $m_2 = n_2 + 1$ with $m_1 > 0$ or $n_2 > 0$,
then the specialization $S'$ obtained by exchanging $0^A_{n_1}$ and $1^B_{m_2}$
corresponds to the ${\rm DM_1}$ $m N_{1,1}$ with $m = m_1+m_2$.
Clearly this $S'$ satisfies $\ell(S') < \ell(S) - 1$.
In this case, the set $\mathcal A^{(0)}$ contains the inverse image of $1^B_{m_2}$
for the small modification $\pi$ by $0^A_{n_1}$ and $1^B_{m_2}$.
We may assume that every specialization $S'$ is obtained by 
the full modification $(S^{(a+b)}, \delta, \pi)$.

First, we will show (i) and (ii) of Proposition~\ref{PropOfCentralClassification}.
For the ABS $(\tilde S, \delta, \pi_0)$ associated to $N_\xi$,
fix elements $0^A_i$ and $1^B_j$ with $m_1 < i \leq n_1$.
Let $\pi$ be the small modification by $0^A_i$ and $1^B_j$.
Let $a$ (resp. $b$) denote the smallest non-negative integer such that
$\mathcal A^{(a)} = \emptyset$ (resp. $\mathcal B^{(b)} = \emptyset$).
We introduce some definitions to calculate the length of 
the specialization $S' = (S^{(a + b)}, \delta, \pi)$.
For simplicity, we often write $\ell(S^{(n)})$ for the length of the ABS $(S^{(n)}, \delta, \pi)$.

\begin{notation}
For non-negative integers $n$ with $n < a$,
we define $d_A(n)$ by
$$d_A(n) = |\mathcal A^{(n)}| - |\mathcal A^{(n+1)}|.$$
Moreover, we define 
$\Delta \ell_A(n)$ by
$$\Delta \ell_A(n) = \ell(S^{(n+1)}) - \ell(S^{(n)}).$$
Put $\Delta \ell_A = \sum_n \Delta \ell_A(n)$.
\end{notation}

\begin{notation}
Let $n'$ be a non-negative integer with $n' \geq a$.
Put $n = n' - a$.
We define $d_B(n)$ by
$$d_B(n) = |\mathcal B^{(n)}| - |\mathcal B^{(n+1)}|.$$
Moreover, we define 
$\Delta \ell_B(n)$ by
$$\Delta \ell_B(n) = \ell(S^{(n'+1)}) - \ell(S^{(n')}).$$
Put $\Delta \ell_B = \sum_n \Delta \ell_B(n)$.
\end{notation}
%% DefinitionからNotationに変更

\begin{example}
Let $N = N_{2,7} \oplus N_{5,3}$.
In Example~\ref{ExOfExchange},
we constructed the small modification by elements $0^A_6$ and $1^B_3$.
By this small modification, the length of the ABS's decrease by four
from $S$ to $S^{(0)}$.
We have $\Delta \ell_A = 1$ and $\Delta \ell_B = 2$.
Thus the length is increased by three
from $S^{(0)}$ to $S'$,
and eventually we see $\ell(S') = \ell(S) - 1$.
\end{example}

For the above definition of $d_A(n)$ and $d_B(n)$,
we obtain 
Lemma~\ref{LemOfSumd} and Lemma~\ref{LemOfSumd2}
which are used for evaluating values $\Delta \ell_A$ and $\Delta \ell_B$.
Recall that $I$ is the subset of $\mathcal B^{(0)}$ consisting of elements 
$t$ of $\tilde A$.

\begin{lemma}\label{LemOfSumd}
Let $S' \in \mathcal H_1(S)$.
The following are true:
\begin{enumerate}
\item[(1)]$\sum_n d_A(n) = n_1 - i$,
\item[(2)]$\sum_n d_B(n) = |I| +j - 1$.
\end{enumerate}
\end{lemma}

\begin{proof}
Clearly $\sum_n d_A(n)$ is given by $|\mathcal A^{(0)}| - |\mathcal A^{(a)}|$.
As $|\mathcal A^{(0)}| = n_1 - i$ and $|\mathcal A^{(a)}| = 0$,
we obtain the desired value.
Similarly, we obtain (2) 
since $|\mathcal B^{(0)}| = |I| + j - 1$ and $|\mathcal B^{(b)}| = 0$.
\end{proof}

\begin{lemma}\label{LemOfSumd2}
Let $S' \in \mathcal H_2(S)$.
The following are true:
\begin{enumerate}
\item[(1)]$\sum_n d_A(n) = h_1-i$,
\item[(2)]$\sum_n d_B(n) = |I|+ j - 1$.
\end{enumerate}
\end{lemma}

\begin{proof}
Note that we have $\mathcal A^{(0)} = \{0^A_{i+1}, \dots, 0^A_{h_1}\}$
and $\mathcal B^{(0)} = I \cup \{1^B_1, \dots, 1^B_{j-1}\}$ in this case.
A proof is given by the same way as Lemma~\ref{LemOfSumd}.
\end{proof}

\begin{notation}
For an element $t$ of $S^{(n)}$ with $\delta(t) = 1$,
we define $\ell(t,n)$ by the number of elements $t'$
satisfying that $t' < t$ and $\delta(t') = 0$ in $S^{(n)}$.
For instance, the sum $\sum_t \ell(t, n)$
is equal to the length of the ABS $(S^{(n)}, \delta, \pi)$.
\end{notation}

We will give a criterion of generic specializations
in Proposition~\ref{PropOfCentralClassification} and 
Theorem~\ref{ThmOfCentralClassification} by 
comparing values of $d_A(n)$ and $\Delta \ell_A(n)$,
or $d_B(n)$ and $\Delta \ell_B(n)$ using Proposition~\ref{PropOfldA} 
and Proposition~\ref{PropOfldB} below.
 
\begin{lemma}\label{LemOfLengthCal}
Let $n$ be a non-negative integer.
Put $\alpha = \pi^{n+1}(0^A_i)$ and $\beta = \pi^{n+1}(1^B_j)$.
The following holds:
\begin{eqnarray}
|\Delta \ell_A(n)| &=&
\#\{t \in \mathcal A^{(n)} \mid \delta(\pi(t)) \neq \delta(\alpha)\}, \label{eqellTA}\\
|\Delta \ell_B(n)| &=&
\# \{t \in \mathcal B^{(n)} \mid \delta(\pi(t)) \neq \delta(\beta)\}. \label{eqellTB}
\end{eqnarray}
Concretely, if $\delta(\alpha) = 0$ (resp. $\delta(\alpha) = 1$),
we have then $\Delta \ell_A(n) \leq 0$ (resp. $\Delta \ell_A(n) \geq 0$),
and if $\delta(\beta) = 0$ (resp. $\delta(\beta) = 1$),
we have then $\Delta \ell_B(n) \geq 0$ (resp. $\Delta \ell_B(n) \leq 0$).
\end{lemma}

\begin{proof}
We fix a non-negative integer $n$,
and let us show that the equation \eqref{eqellTA} holds.
In the same way, we can obtain the equation \eqref{eqellTB}.
We divide the proof into two cases depending on the value of $\delta(\alpha)$.
First, suppose $\delta(\alpha) = 0$.
If $\pi(\mathcal A^{(n)})$ does not contain $1^B_j$,
then it follows from Proposition~\ref{PropOf1} that
all elements $t$ of $\pi(\mathcal A^{(n)})$ satisfy $\delta(t) = 0$, and hence
we have $d_A(n) = 0$.
Moreover, in this case $\Delta \ell_A(n) = 0$ holds.
If $\pi(\mathcal A^{(n)})$ contains $1^B_j$,
then all elements $t$ of $\pi(\mathcal A^{(n)})$ satisfy $\delta(t) = 0$
except for $1^B_j$, whence we have $d_A(n) = 1$.
Then $\ell(1^B_j, n+1) = \ell(1^B_j, n) - 1$ holds.
Moreover, $\ell(t, n+1) = \ell(t, n)$ holds for the other elements,
whence we have $\Delta \ell_A(n) = -1$.
Next, let us see the case of $\delta(\alpha) = 1$.
By the construction of $S^{(n+1)}$, it is clear that
$\ell(\alpha, n+1) = \ell(\alpha, n) + r$, where
$r = \#\{t \in \mathcal A^{(n)} \mid \delta(\pi(t)) \neq \delta(\alpha)\}$.
Moreover, $\ell(t, n+1) = \ell(t, n)$ holds for the other elements $t$.
Clearly we have $d_B(n) = r$.
Hence we get desired equality for the case of $\delta(\alpha) = 1$.
This completes the proof.
\end{proof}

\begin{proposition}\label{PropOfldA}
For all non-negative integers $n$,
an inequality $\Delta \ell_A(n) \leq d_A(n)$ holds.
The equality holds for all $n$ if and only if $\mathcal A^{(n)}$ 
do not contain the inverse image of $1^B_j$ for all $n$.
\end{proposition}

\begin{proof}
By the condition $m_1 < i \leq n_1$, in ABS's $(S^{(n)}, \delta, \pi)$,
the inverse image of $1^B_j$ is obtained by $0^A_{i+m_1}$.
We fix a non-negative integer $n$ with $n < a$.
Put $\alpha = \pi^{n+1}(0^A_i)$.
In the ordered set $S^{(n+1)}$,
the element $\alpha$ is located in the right of elements of $\pi(\mathcal A^{(n)})$.
Let us suppose that $\mathcal A^{(n)}$ does not contain $0^A_{i+m_1}$.
This assumption implies that $\pi(\mathcal A^{(n)})$ contains
no element of $B$.
%By Proposition~\ref{PropOf1},
If $\delta(\alpha) = 0$, 
then all elements $t$ of $\mathcal A^{(n)}$ satisfy $\delta(\pi(t)) = 0$.
Hence we have $\Delta \ell_A(n) = 0$,
and then $d_A(n) = 0$ holds.
If $\delta(\alpha) = 1$,
by Proposition~\ref{PropOfT}, the set $\mathcal A^{(n+1)}$ is obtained by 
$\pi(\mathcal A^{(n)}) \setminus \Xi$, where 
$$\Xi = \{\pi(t) \mid t \in \mathcal A^{(n)} \text{ and } 
\delta(\pi(t)) \neq \delta(\alpha)\}.$$
We immediately obtain $|\Xi| = d_A(n)$.
Lemma~\ref{LemOfLengthCal}
concludes that $\Delta \ell_A(n) = d_A(n)$.

Assume that there exists a non-negative integer $n$
such that $\mathcal A^{(n)}$ contains $0^A_{i+m_1}$.
We divide the proof into two cases depending on the value of $\delta(\alpha)$.
First, if $\delta(\alpha) = 0$,
it follows from $\ell(1^B_j, n+1) = \ell(1^B_j, n) - 1$
that $\Delta \ell_A(n) = -1$.
As $\mathcal A^{(n+1)}$ is obtained by
$\mathcal A^{(n+1)} = \pi(\mathcal A^{(n)}) \setminus \{1^B_j\}$,
we have $d_A(n) = 1$.
Hence we get $\Delta \ell_A(n) < d_A(n)$.
Next, if $\delta(\alpha) = 1$,
we obtain the set $\mathcal A^{(n+1)}$ by $\pi(\mathcal A^{(n)}) \setminus \Xi'$,
where $$\Xi' = \{\pi(t) \mid t \in \mathcal A^{(n)}, \text{ with } 
\delta(\pi(t)) = 0\ {\rm or}\ \pi(t) \in \tilde B\}.$$
It is clear that $\Xi'$ contains $1^B_j$ in this hypothesis.
We have $d_A(n) = |\Xi'|$.
On the other hand, since $\delta(1^B_j) = \delta(\alpha)$,
we have $\ell(\alpha, n+1) = \ell(\alpha, n) + (|\Xi'| - 1)$, and 
it implies that $\Delta \ell_A(n) = |\Xi'| - 1$.
Hence we get $\Delta \ell_A(n) < d_A(n)$.
\end{proof}

\begin{proposition}\label{PropOfldB}
For all non-negative integers $n$,
an inequality $\Delta \ell_B(n) \leq d_B(n)$ holds.
Moreover, for $1 \leq j \leq n_2$,
the equality holds for all $n$ if and only if 
\begin{enumerate}
\item[(1)]$\mathcal B^{(n)}$ do not contain the inverse image of $0^A_i$
for all $n$, and
\item[(2)]$I = \emptyset$.
\end{enumerate}
\end{proposition}

\begin{proof}
For all $j$, the inequality follows from 
Lemma~\ref{LemOfLengthCal}.
To see the latter part, we treat the case of $1 \leq j \leq n_2$.
In this hypothesis, in ABS's $(S^{(n)}, \delta, \pi)$,
the inverse image of $0^A_i$ is obtained by $0^B_{j+m_2}$. 
If sets $\mathcal B^{(n)}$ do not contain $0^B_{j+m_2}$ for all $n$
and $I = \emptyset$,
then we can show that the equality $\Delta \ell_B(n) = d_B(n)$ holds
in the same way as Proposition~\ref{PropOfldA}.

Let us see the converse.
Put $\beta_n = \pi^n(1^B_j)$ for non-negative integers $n$.
We assume that $\mathcal B^{(n)}$ contains $0^B_{j+m_2}$
for a non-negative integer $n$.
By the condition $n_2/h_2 < 1/2$ with $h_2 = m_2 + n_2$, 
we have $\delta(t) = 0$ and 
$\delta(\pi(t)) = 1$
for all elements $t$ of $\mathcal B^{(n)}$
except for $0^B_{j+m_2}$.
Moreover, we have $\delta(\beta_n) = 0$ and
$\delta(\beta_{n+1}) = 1$.
In the ABS $(S^{(a+n+1)}, \delta, \pi)$
we have $\beta_{n+1} < \pi(t_{\rm min})$,
where $t_{\rm min}$ is the minimum element of $\mathcal B^{(n)}$.
We have then $\Delta \ell_B(n) = -1$
since $\ell(\beta_{n+1}, n+1) = \ell(\beta_{n+1}, n) - 1$
and $\ell(t, n+1) = \ell(t, n)$ for the other elements $t$ of $S^{(a+n+1)}$.
On the other hand, we have $d_B(n) = 1$.
In fact, $\mathcal B^{(n+1)}$ is given by 
$\mathcal B^{(n+1)} = \pi(\mathcal B^{(n)}) \setminus \{0^A_i\}$.

Next, assume $I \neq \emptyset$.
We divide the proof into two cases depending on values of $\delta(\beta_1)$.
If $\delta(\beta_1) = 0$, 
then the set $\mathcal B^{(1)}$ is the union of 
$\Lambda$ and $\pi(I)$,
where $\Lambda$ is the subset of $\mathcal B^{(1)}$ consisting of elements $t$ of $\mathcal B^{(1)}$
satisfying $t \in B$.
Note that $\delta(\beta_2) = 1$ in this hypothesis.
We have $\delta(\pi(t)) = 1$
for every element $t$ of $\Lambda$. 
Moreover, Lemma~\ref{LemOfI} implies that
$\delta(\pi(t)) = 0$ for every element $t$ of $\pi(I)$.
Hence we have $\ell(\beta_2, 2) = \ell(\beta_2, 1) - |I|$,
and it implies that $\Delta \ell_B(1) = -|I|$. 
On the other hand, we have $d_B(1) = |I|$,
whence $\Delta \ell_B(1) < d_B(1)$ holds.
Let us suppose $\delta(\beta_1) = 1$.
In this case,
we obtain $\Delta \ell_B(0) = -|I|$.
Since $d_B(0) = |I|$,
we have $\Delta \ell_B(0) < d_B(0)$.
\end{proof}

Thanks to the above propositions, 
we can prove Proposition~\ref{PropOfCentralClassification}.

\begin{proof}[Proof of Proposition~\ref{PropOfCentralClassification} (i)]
We have $\ell((S^{(0)}, \delta, \pi)) - \ell(S) = -(n_1-i+j)$.
Note that if there exists no non-negative integer $n$ such that
$\mathcal A^{(n)}$ contains the inverse image of $1^B_j$,
then $I = \emptyset$ holds.
Furthermore, if there exists no non-negative integer $n$ such that
$\mathcal B^{(n)}$ contains the inverse image of $0^A_i$, then
Lemma~\ref{LemOfSumd}, 
Proposition~\ref{PropOfldA} and Proposition~\ref{PropOfldB} imply
that $\Delta \ell_A = n_1 - i$ and $\Delta \ell_B = j - 1$.
Hence we have $\ell(S') - \ell((S^{(0)}, \delta, \pi)) = n_1 - i + j - 1$,
and $\ell(S') = \ell(S) - 1$ holds.

Suppose that 
there exists a non-negative integer $n$
such that $\mathcal A^{(n)}$ contains the inverse image of $1^B_j$ 
or $\mathcal B^{(n)}$ contains the inverse image of $0^A_i$.
If $I = \emptyset$, then we have 
$\Delta \ell_A < n_1 - i$ or $\Delta \ell_B < j - 1$.
%Since $\ell((S^{(0)}, \delta, \pi)) - \ell(S) = -(n_1 - i + j)$,
Then we have $\ell(S') < \ell(S) - 1$.
On the other hand, if $I \neq \emptyset$,
then we have $\Delta \ell_A \leq n_1 - i - |I|$.
Moreover, by the proof of Proposition~\ref{PropOfldB},
as there exists a non-negative integer $m$ such that
$\Delta \ell_B(m) = -|I|$,
we have $\Delta \ell_B < j - 1$.
Hence we have $\ell(S') < \ell(S) - 1$.
\end{proof}

\begin{proof}[Proof of Proposition~\ref{PropOfCentralClassification} (ii)]
In this case, the set $\mathcal A^{(0)}$ is given by
$\mathcal A^{(0)} = \{0^A_{i+1}, \dots, 0^A_{h_1}\}$.
We have $\ell((S^{(0)}, \delta, \pi)) - \ell(S) = -(h_1-i+j)$.
By the condition of $i$,
the set $\mathcal A^{(0)}$ contains $0^A_{i+m_1}$
which is the inverse image of $1^B_j$.
Hence we have $\Delta \ell_A < h_1-i$
by Lemma~\ref{LemOfSumd2} and Proposition~\ref{PropOfldA}.
Moreover, we have $\Delta \ell_B \leq j - 1$
since if $\mathcal B^{(n)}$ contains the element $\pi^n(t)$ for $t \in I$,
then $\Delta \ell_B(n) < 0$.
and hence $\ell(S') < \ell(S) - 1$ holds.
In the case $n_1 = m_1 + 1$ and $m_2 = n_2 + 1$,
for the small modification $\pi$ by $0^A_{n_1}$ and $1^B_{m_2}$,
the set $\mathcal A^{(0)}$ contains the inverse image of $1^B_{m_2}$.
\end{proof}

Let us classify specializations satisfying 
$\ell(S')=\ell(S)-1$ for $S' \in \mathcal H_3(S)$.
We use the duality to consider this case.
Let $N = N_{\xi}$ be the ${\rm DM_1}$ 
with a Newton polygon $\xi = (m_1,n_1) + (m_2,n_2)$
satisfying $\lambda_2 < 1/2 < \lambda_1$.
Let $N^D$ be the dual of $N$.
Then we have $N^D=N_{n_2,m_2}\oplus N_{n_1,m_1}$.
Let $S_D$ be the ABS corresponding to $N^D$.
%This ABS $S_D$ is obtained as follows:
%For the ABS $S = (\{t_1, \dots, t_h\}, \delta, \pi)$,
%the dual ABS $S_D = (\{t'_1, \dots, t'_h\}, \delta_d, \pi_d)$
%satisfies that 
%$\delta_d(t'_x) \neq \delta(t_{h-x+1})$ and
%$\pi_d(t'_x) = t'_y$, where
%$\pi(t_{h-x+1}) = t_{h-y+1}$
%for all $x$.
Note that $\ell(S) = \ell(S_D)$ follows from
the definition of the lengths of ABS's. 

\begin{proof}[Proof of Proposition~\ref{PropOfCentralClassification} (iii)]
Fix an ABS $S' \in \mathcal H_3(S)$.
For the same notation as above,
$S_D = B_D\oplus A_D$, where
$B_D$ and $A_D$ correspond to 
$N_{n_2,m_2}$ and $N_{n_1,m_1}$ respectively.
The elements $0^A_i$ and $1^B_j$ correspond to $1_{j'} = 1^{A_D}_{j'}$
and $0_{i'} = 0^{B_D}_{i'}$ by the duality,
with $i'=h_2-j+1$ and $j'=h_1-i+1$.
Then $S'_D = (S_D)'$ is the ABS obtained by the small modification by
$0_{i'}$ and $1_{j'}$ in $S_D$.
Since $S'_D$ belongs to $\mathcal H_1(S_D)$,
the exchange of $0_{i'}$ and $1_{j'}$ in $S_D$
is good if and only if the small modification by $0_{i'}$ and $1_{j'}$
satisfies the necessary and sufficient condition
of Proposition~\ref{PropOfCentralClassification} (i).
As the small modification by $0^A_i$ and $1^B_j$ in $S$ corresponds to 
the small modification by $0_{i'}$ and $1_{j'}$ in $S_D$,
the equality $\ell(S')=\ell(S)-1$ holds if and only if
$\ell(S'_D)=\ell(S_D)-1$holds.
\end{proof}

Finally, Proposition~\ref{PropOfCentralClassification} induces the main theorem 
of this section (Theorem~\ref{ThmOfCentralClassification}).

\begin{proof}[Proof of Theorem~\ref{ThmOfCentralClassification}]
Proposition~\ref{PropOfCentralClassification} (i) and (ii) imply
that the statement of Theorem~\ref{ThmOfCentralClassification} holds 
for ABS's $S'$ of $\mathcal H_1(S)$ and $\mathcal H_2(S)$. 
Moreover, by the duality,
Proposition~\ref{PropOfCentralClassification} (iii) concludes that
$\ell(S') = \ell(S) - 1$ holds if and only if 
there exists no non-negative integer $n$ such that
$\mathcal A^{(n)}$ (resp. $\mathcal B^{(n)}$) does not contain 
the inverse image of $1^B_j$ (resp. $0^A_i$) for $S' \in \mathcal H_3(S)$.
\end{proof}

We want to determine boundary components of $H(\xi)$ for
arbitrary Newton polygons $\xi$.
If Conjecture~\ref{ConjOfArbitBC} is true,
then the two segments case is essential for 
classification of boundary components of all central streams.

\begin{conjecture} \label{ConjOfArbitBC}
Let $\xi$ be a Newton polygon of $z$ segments.
Let $N_\xi$ be the minimal ${\rm DM_1}$ of $\xi$.
A specialization of $N_\xi$ is generic if and only if
it is the direct sum of 
a generic specialization of $N_r \oplus N_{r+1}$ and 
the minimal $p$-divisible group 
$N_1 \oplus \cdots \oplus N_{r-1} \oplus N_{r+2} \oplus \cdots \oplus N_z$
for a natural number $r$.
Here, $N_i$ is the simple ${\rm DM_1}$ associated with the $i$-th segment of $\xi$.
\end{conjecture}
%% $N_1 \oplus \cdots \oplus N_{r-1} \oplus N_{r+2} \oplus \cdots \oplus N_z$の直前にthe minimal p-divisible groupを追加
%%% boundary componentがどうとかは削除

\subsection{Some properties of generic specializations}
\label{SomePropOfBoundaryComp}
%% タイトルを変更

Here, we introduce some notation 
and properties of generic specializations,
which are 
%% which is を which are に変更
useful for showing Theorem~\ref{ThmOfzetaxi}.

Theorem~\ref{ThmOfCentralClassification} and 
Proposition~\ref{PropOfCentralClassification} imply that
it suffices to deal with ABS's of $\mathcal H_1(S)$
to study boundary components of central streams.
From now on, in this paper, for elements $0^A_i$ and $1^B_j$
used in the construction of the small modification $\pi$, 
%%using for constructing を used in the construction of に変更
we make the assumption
%% an assumption を the assumption に変更
$m_1 < i \leq n_1$ and $1 \leq j \leq n_2$.
Note that for the good exchange of $0^A_i$ and $1^B_j$,
%% if an exchange を if the exchange of に変更
if $a$ and $b$ are positive,
then non-negative integers $a$ and $b$ satisfy that
$\pi^{a}(0^A_i) = 1^A_{m_1}$ and $\pi^{b}(1^B_j) = 0^B_{m_2+1}$.

By Theorem~\ref{PropOfCentralClassification},
if the exchange of $0^A_i$ and $1^B_j$ is good, 
then sets $\mathcal A^{(n)}$ (resp. $\mathcal B^{(n)}$)
are subsets of $\tilde A$ (resp. $\tilde B$) for all non-negative integers $n$, whence
for the small modification by $0^A_i$ and $1^B_j$, 
sets $\mathcal A^{(n)}$ (resp. $\mathcal B^{(n)}$) and the value $\Delta \ell_A$
(resp. $\Delta \ell_B$) do not depend on $j$ (resp. $i$). 
Thus we define the following sets.

\begin{definition} \label{DefOfC'D'}
Let $A = (\tilde A, \delta_A, \pi_A)$ and $B = (\tilde B, \delta_B, \pi_B)$
be the ABS of $N_{m_1, n_1}$ and $N_{m_2, n_2}$ respectively.
For the ABS $S = A \oplus B$,
which is associated to $N_\xi$,
let $G$ be the subset of $\tilde A \times \tilde B$ consisting of pairs $(0^A_i, 1^B_j)$
such that exchanges of $0^A_i$ and $1^B_j$ are good.
By the above, this set is described as $G = C' \times D'$
with $C' \subset \tilde A$ and $D' \subset \tilde B$.
Equivalently, we obtain a generic specialization of $S$
if and only if we construct a small modification by 
an element of $C'$ and an element of $D'$.
Moreover, let $C$ (resp. $D$) be the subset 
of $C'$ (resp. $D'$) consisting of elements 
satisfying that $\mathcal A^{(0)}$ (resp. $\mathcal B^{(0)}$) is not empty.
\end{definition}

From now on, we fix the notation of $C',\ C,\ D'$ and $D$.
We have 
$C' \setminus C = \{0^A_{n_1}\}$ and $D' \setminus D = \{1^B_1\}$.
We call the construction of a full modification using an 
%% constructing a full modification by an を the construction of a full modification using anに変更
element of $C'$ and an element of $D'$
good exchange.

\begin{example}
Let $N_\xi = N_{2,7} \oplus N_{5,3}$.
Then the set $G = C' \times D'$ is given by
$$C' \times D' = \{(0^A_6, 1^B_1),\ (0^A_6, 1^B_3),\ (0^A_7, 1^B_1),\ (0^A_7, 1^B_3)\}.$$
We have $C' = \{0^A_6, 0^A_7\}$ and $D' = \{1^B_1, 1^B_3\}$.
\end{example}

In Lemma~\ref{LemOfDecrease} and Proposition~\ref{PropOfNumbers},
we give some properties of $\{\mathcal A^{(n)}\}_{n = 0, \dots, a}$ and 
$\{\mathcal B^{(n)}\}_{n = 0, \dots, b}$
for generic specializations.
For the ABS $S$ associated to $N_\xi$,
let $(S^{(n)}, \delta, \pi)$ be ABS's 
obtained by constructing the full modification by $0^A_i \in C'$ and $1^B_j \in D'$
for $n = 0, \dots, a+b$.

\begin{lemma}\label{LemOfDecrease}
Let $n$ be a non-negative integer.
For a good exchange,
if $d_A(n) > 0$ (resp. $d_B(n) > 0$)
and $\mathcal A^{(n+1)}$ (resp. $\mathcal B^{(n+1)}$) is not empty,
then $\mathcal A^{(n+1)}$ (resp. $\mathcal B^{(n+1)}$)
has the maximum element $1^A_{m_1}$ 
(resp. the minimum element $0^B_{m_2+1}$).
%おそらく逆も成り立つ．
%%%is not empty の仮定を追加
\end{lemma}

\begin{proof}
Let us see the case $d_A(n) > 0$.
We can show the case $d_B(n) > 0$ in the same way.
Fix a non-negative integer $n$.
Put $\alpha = \pi^{n}(0^A_i)$.
Clearly, if $\delta(\alpha) = \delta(\pi(\alpha)) = 0$,
then $\delta(\pi(t)) = \delta(\pi(\alpha))$ holds for all elements $t$ of $\mathcal A^{(n)}$.
Moreover, if $\delta(\alpha) = 1$, we have then
$\delta(\pi(\alpha)) = \delta(\pi(t)) = 0$ for all elements $t$ of $\mathcal A^{(n)}$.
It implies that $d_A(n) = 0$ holds in these cases.
Hence it suffices to see the case that 
$\delta(\alpha) = 0$ and $\delta(\pi(\alpha)) = 1$.
Then we have
$\pi(\mathcal A^{(n)}) = \{1^A_x, \dots , 1^A_{m_1}, 0^A_{m_1+1}, \dots, 0^A_y\}$.
It induces that the maximum element of $\mathcal A^{(n+1)}$ is $1^A_{m_1}$.
\end{proof} 

\begin{proposition}\label{PropOfNumbers}
For a good exchange,
let $n$ be a non-negative integer satisfying $n < a$, and
let $c$ be the natural number satisfying that 
$\pi^n(0^A_i) = \tau^A_c$ with $\tau = 0$ or $1$.
The set $\mathcal A^{(n)}$ is obtained by
$$\mathcal A^{(n)} = \{\tau^A_{c+1},\ \tau^A_{c+2}, \dots, \tau^A_{c+u}\},$$
where $u = |\mathcal A^{(n)}|$.
Similarly, for a non-negative integer $n$ with $n < b$
and the natural number $c$ satisfying $\pi^n(1^B_j) = \tau^B_c$,
the set $\mathcal B^{(n)}$ is obtained by
$$\mathcal B^{(n)} = \{\tau^B_{c-v},\ \tau^B_{c-v+1}, \dots, \tau^B_{c-1}\},$$
where $v = |\mathcal B^{(n)}|$.
\end{proposition}

\begin{proof}
Let us see sets $\mathcal A^{(n)}$.
Note that we have $I = \emptyset$,
and we can see the latter part in the same way.
We use induction on $n$.
Clearly the statement holds for $n = 0$.
For a natural number $n$,
suppose that $\mathcal A^{(n-1)} = \{\tau^A_{c+1}, \dots, \tau^A_{c+u}\}$
with $u = |\mathcal A^{(n-1)}|$.
We have $\mathcal A^{(n)} = \pi(\mathcal A^{(n-1)}) \setminus T$, with
$$T = \{\pi(t) \mid \delta(\pi(t)) \neq \delta(\pi^n(0^A_i)) 
\text{ for } t\in \mathcal A^{(n-1)}\}.$$
If $T = \emptyset$, we clearly have the property for $\mathcal A^{(n)}$. 
Suppose that $T \neq \emptyset$.
By Lemma~\ref{LemOfDecrease}, we see
$\mathcal A^{(n)} = \{1^A_{c+1},\ 1^A_{c+2}, \dots, 1^A_{m_1-1}, 1^A_{m_1}\}$.
This completes the proof.
\end{proof}

\section{Determining Newton polygons of generic specializations}
\label{ConstGoodSpe}
%% タイトル変更

The main purpose of this section is to show Theorem~\ref{ThmOfzetaxi}.
For a minimal ${\rm DM_1}$ $N_\xi$ and the ABS $S$ associated to $N_\xi$, 
we say a specialization $N'$ of $N_\xi$ is generic
if the specialization $S'$ associated to $N'$ satisfies $\ell(S') = \ell(S) - 1$,
and we denote by $N_\xi^-$ this generic specialization of $N_\xi$.
We will show Proposition~\ref{PropOfSpe}
which is a key step of constructing specializations
from a generic specialization of $N_{\xi}$ 
to a minimal ${\rm DM_1}$ $N_{\zeta}$ with $\zeta \prec \xi$ is saturated. 
The main result of this section is

\begin{proposition}\label{PropOfSpe}
Let $\xi = (m_1,n_1) + (m_2,n_2)$ be a Newton polygon
satisfying $\lambda_2< 1/2 < \lambda_1$, where
$\lambda_1 = n_1/(m_1+n_1)$ and $\lambda_2 = n_2/(m_2+n_2)$.
Assume that $\xi$ is not $(0, 1) + (1, 0)$.
Let $N_\xi$ be the ${\rm DM_1}$ of $\xi$.
For every generic specialization $N^-_{\xi}$, there exist $N^{--}_{\xi}$ and $N^-_{\xi'}$
satisfying
\begin{equation}\label{EqOfInduction}
N^{--}_{\xi}=N^-_{\xi'}\oplus N_{\rho},
\end{equation} 
where $\rho = (f, g)$ is a Newton polygon and 
$\rho$ is uniquely determined by $\xi$ so that 
the area of the region surrounded by $\xi$, $\xi'$ and $\rho$ is one.
The diagram of these Newton polygons is described as either of the 
following:
\\
\\
\includegraphics[width=120mm]{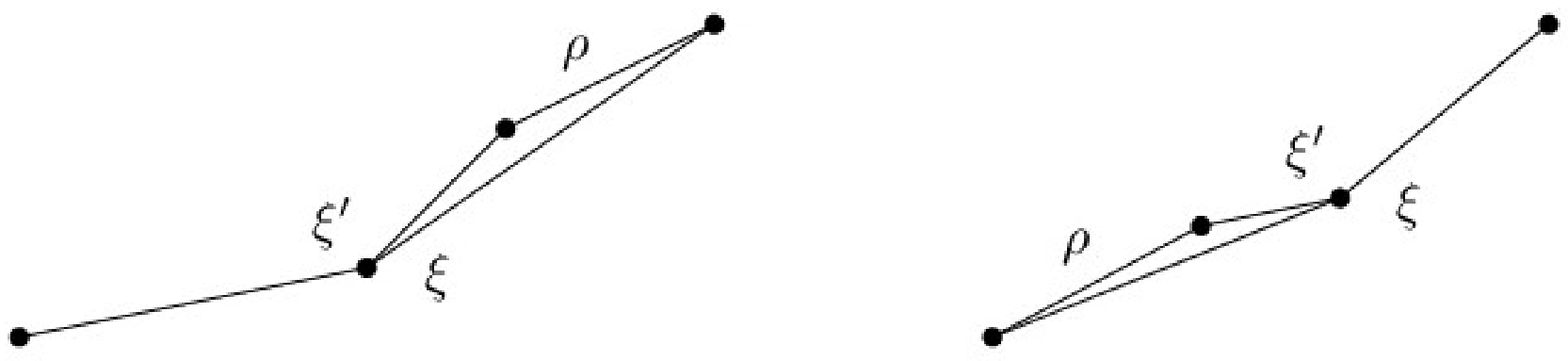}
%\begin{tikzpicture}
%\node [anchor=north] at (6/2, 1.5/2){$\xi$};
%\node [anchor=north] at (8/2, 4.6/2){$\rho$};
%\node [anchor=north] at (4.5/2, 2.3/2){$\xi'$};

%\node [anchor=north] at (20/2, 2.5/2){$\xi$};
%\node [anchor=north] at (1.2/2+7, 1.8/2){$\rho$};
%\node [anchor=north] at (18.5/2, 3.3/2){$\xi'$};

%\draw (0,0) -- ++(5/2, 1/2) -- ++(5/2, 3.5/2);
%\draw (5/2, 1/2) -- ++ (2/2, 2/2) -- ++ (3/2, 1.5/2);

%\draw (7,0) -- ++(5/2, 2/2) -- ++(3/2, 2.5/2);
%\draw (7,0) -- ++(3/2, 1.6/2) -- ++(2/2, 0.4/2);

%\fill (0,0) circle (2pt);
%\fill (5/2, 1/2) circle (2pt);
%\fill (7/2, 3/2) circle (2pt);
%\fill (10/2, 4.5/2) circle (2pt);

%\fill (7,0) circle (2pt);
%\fill (5/2+7, 2/2) circle (2pt);
%\fill (3/2+7, 1.6/2) circle (2pt);
%\fill (8/2+7, 4.5/2) circle (2pt);
%\end{tikzpicture}

\end{proposition}

For the above statement, we will determine the diagram of Newton polygons
in Proposition~\ref{ThmOfCases} dividing the cases into the conditions of $\xi$.
We fix notations as follows:
Let $\xi = (m_1, n_1) + (m_2, n_2)$ be a Newton polygon 
satisfying that $\lambda_2 < 1/2 < \lambda_1$.
Set $h_1 = m_1 + n_1$ and $h_2 = m_2 + n_2$.
For the ABS $S = A \oplus B$ associated to the ${\rm DM_1}$ 
$N_\xi = N_{m_1, n_1} \oplus N_{m_2, n_2}$,
let $S^-$ denote the ABS obtained by a good exchange by $0^A_i$ and $1^B_j$,
and let $N_\xi^-$ be the ${\rm DM_1}$ corresponding to $S^-$.
Recall that by the small modification by $0^A_i$ and $1^B_j$,
we obtain sets 
$\{\mathcal A^{(n)}\}_{n = 0, \dots, a}$ and $\{\mathcal B^{(n)}\}_{n = 0, \dots, b}$,
where $a$ (resp. $b$) is the smallest non-negative integer satisfying that
$\mathcal A^{(a)} = \emptyset$ (resp. $\mathcal B^{(b)} = \emptyset$).
Note that
if $a$ and $b$ are positive, then
$\pi^a(0^A_i) = 1^A_{m_1}$ and $\pi^b(1^B_j) = 0^B_{m_2+1}$.
For the ABS $S = (\tilde S, \delta, \pi_0)$,
the generic specialization $S^-$ obtained by exchanging  $0^A_i$ and $1^B_j$
is constructed by ABS's $\{(S^{(n)}, \delta, \pi)\}_{n = 0, \dots, a + b}$.
%To recall the details, see Definition~\ref{DefOfTAn} and Definition~\ref{DefOfTBn}.
Since we can use the duality of ${\rm DM_1}$'s
for $S^- \in \mathcal H_3(S)$,
it suffices to see the case that $S^-$ belongs to $\mathcal H_1(S)$.
%see Definition~\ref{DefOfH123}.
Hence we assume that $m_1 < i \leq n_1 \text{ and }1 \leq j \leq n_2$.

Theorem~\ref{ThmOfzetaxi} is obtained by 
applying Proposition~\ref{PropOfSpe} inductively.
In Proposition~\ref{ThmOfCases}, we concretely give 
an operation to obtain $N_\xi^{--}$ 
satisfying the equality \eqref{EqOfInduction}.
%and Proposition~\ref{PropOfSpe} immediately follows from this proposition.

\begin{proposition}\label{ThmOfCases}
Let
$\xi = (m_1, n_1) + (m_2, n_2)$ be a Newton polygon satisfying 
$0 < \lambda_2 < 1/2 < \lambda_1 < 1$.
Let $N^-_{\xi}$ be any generic specialization of $N_\xi$.  
Let $S^-$ be the ABS associated to $N_\xi^-$.
\begin{enumerate}
\item[(I)] If $n_1 > m_1 + 1$,
then we obtain the equality \eqref{EqOfInduction}
for $N_\xi^{--}$ corresponding to $S^{--}$,
where $S^{--}$ is either of the following:
\begin{itemize}
\item[(a)] $S^{--}$ is the specialization obtained by 
exchanging $0^A_{m_1+1}$ and $1^A_{m_1}$ for $S^-$, or
\item[(b)] $S^{--}$ is the specialization obtained by 
exchanging $0^A_{i-1}$ and $1^B_j$ for $S^-$.
\end{itemize}
For these cases,
the Newton polygon $\xi'$ of \eqref{EqOfInduction} %of Proposition~\ref{PropOfSpe}
is of the form $\xi' = (m_1-f, n_1-g) + (m_2, n_2)$.
\item[(II)] If $n_1 = m_1 + 1$,
then we obtain the equality \eqref{EqOfInduction}
for $N_\xi^{--}$ corresponding to $S^{--}$,
where $S^{--}$ is any of the following:
\begin{itemize}
\item[(c)] $S^{--}$ is the specialization obtained by 
exchanging $0^B_{m_2+1}$ and $1^B_{m_2}$ for $S^-$,
\item[(d)] $S^{--}$ is the specialization obtained by 
exchanging $0^A_i$ and $1^B_{j+1}$ for $S^-$,
\item[(e)] $S^{--}$ is the specialization obtained by 
exchanging $0^A_{h_1}$ and $1^B_{n_2+1}$ for $S^-$.
\end{itemize} 
For (c) and (d), the Newton polygon $\xi'$ 
of \eqref{EqOfInduction} %of Proposition~\ref{PropOfSpe} 
is of the form 
$\xi' = (m_1, n_1) + (m_2-f, n_2-g)$,
and for the case (e), we have $\xi' = (m_1-1, n_1-1) + (m_2, n_2)$.
\end{enumerate}
\end{proposition}
%% m_i, n_iが自然数を仮定．

\subsection{Proof of Proposition~\ref{ThmOfCases} (I)}
\label{PfOfFormerThmOfCases}

Let us show Proposition~\ref{ThmOfCases} (I).
Assume $n_1 > m_1+1$.
By the construction of $S^-$
which is obtained by a good exchange of $0^A_i \in C'$ and $1^B_j \in D'$, 
we have $0^A_{i-1} < 1^B_j$ in this ABS;
see Definition~\ref{DefOfC'D'} for the definition of sets $C'$ and $D'$.
In fact, it is clear that $0^A_{i-1} < 1^B_j$ in $S^{(0)}$.
Moreover, if $1^B_j < 0^A_{i-1}$ is true in $S^-$,
then there exists a natural number $n$ with $n < a$ such that
$\pi^n(0^A_i) = 0^A_{i-1}$.
Then $\mathcal A^{(n-1)}$ contains the inverse image of $1^B_j$.
This contradicts with $\ell(S^-) = \ell(S) - 1$.
Here, to treat the case (a) of Proposition~\ref{ThmOfCases}, 
in Proposition~\ref{PropOfm1+1m1}, we will see that $0^A_{m_1+1} < 1^A_{m_1}$ holds in $S^-$.
The following notation is useful for showing some properties of ABS's.

\begin{notation}
Let $(\tilde T, \delta, \pi)$ be an ABS.
For an element $t$ of $\tilde T$,
we often express a subset 
$\{\pi^m(t) \mid 0 \leq m \leq n \}$ of $\tilde T$ as
$$t \rightarrow \pi(t) \rightarrow \cdots 
\rightarrow \pi^n(t),$$
and we call this diagram a {\it path}.
\end{notation}

\begin{proposition}\label{PropOfm1+1m1}
We have $0^A_{m_1+1} < 1^A_{m_1}$
in the ABS $S^-$ obtained by a good exchange of 
$0^A_i \in C$ and $1^B_j \in D'$.
Moreover, there exists no non-negative integer $n$
such that $\pi^n(0^A_i) = 0^A_{m_1+1}$ with $n \leq a$.
\end{proposition}

\begin{proof}
By the condition $n_1 - m_1 > 1$,
we have $h_1 > 2$. 
First, to show the former statement,
let us see that
the non-negative integer $a$ is not greater than $h_1 - 2$.
For a good exchange,
the non-negative integer $a$ satisfies that
$\pi^{a}(0^A_i) = 1^A_{m_1}$.
We clearly have $a < h_1$.
If $a = h_1 - 1$,
we have then $i = h_1$.
This contradicts with the condition of the natural number $i$.
Hence we have $a \leq h_1 - 2$.

In the ABS $A$ which is associated to the simple ${\rm DM_1}$ $N_{m_1,n_1}$,
binary expansions of $1^A_{m_1}$ and $0^A_{m_1+1}$ 
are given by 
\begin{eqnarray}
b(1^A_{m_1}) &=& 0.b_1b_2\cdots b_{h_1-2}01\cdots,
\label{BinOf1}
\\
b(0^A_{m_1+1}) &=& 0.b_1b_2\cdots b_{h_1-2}10\cdots.
\label{BinOf0}
\end{eqnarray}
For elements $0^A_i$ and $0^A_{i+1}$ of $S^-$,
we have two paths as follows:
\begin{eqnarray}
0^A_i &\xrightarrow{\pi} \cdots \xrightarrow{\pi}& 1^A_{m_1}, \label{DiagOf1}\\
0^A_{i+1} &\xrightarrow{\pi} \cdots \xrightarrow{\pi}& 0^A_{m_1+1}. \label{DiagOf0}
\end{eqnarray}
It is clear that $0^A_{i+1}$ belongs to $\mathcal A^{(0)}$.
Moreover, the above binary expansions and Corollary~\ref{CoroOfT}
induce that
$\pi^n(0^A_{i+1})$ belongs to the set $\pi(\mathcal A^{(n-1)})$
for every natural number $n$ with $n \leq a$.
In particular, we apply this property for $n = a$,
and we have $0^A_{m_1+1} \in \pi(\mathcal A^{(a-1)})$.
By the construction of $S^{(a)}$,
%the element $1^A_{m_1}$ is situated in the right of 
%the maximum element of $\pi(\mathcal A^{(a-1)})$, and
we obtain $0^A_{m_1+1} < 1^A_{m_1}$ in $(S^{(a)}, \delta, \pi)$ and $S^-$.

Let us see the latter statement.
Suppose that there exists a non-negative integer $n$
with $n \leq a$ satisfying that $\pi^n(0^A_i) = 0^A_{m_1+1}$.
If this hypothesis leads $1^A_{m_1} < 0^A_{m_1+1}$,
then we have a contradiction with the former statement.
Suppose that $0^A_{m_1+1} < 1^A_{m_1}$ in $S^-$.
Then the set $\pi(\mathcal A^{(a-1)})$ contains the element $\pi^n(0^A_i)$.
It implies that the set $\mathcal A^{(a-1)}$ contains the element $\pi^{n-1}(0^A_i)$.
Proposition~\ref{PropOf1} concludes that this is a contradiction,
and hence $1^A_{m_1} < 0^A_{m_1+1}$ holds.
\end{proof}

In the same way, we obtain the following assertion
which is used in the proof of Proposition~\ref{ThmOfCases} (II).
%We state it as follows.

\begin{proposition}\label{PropOfm_2}
For a good exchange of $0^A_i \in C'$ and $1^B_j \in D$, 
we have $0^B_{m_2+1} < 1^B_{m_2}$ in the ABS $S^-$.
Moreover, there exists no non-negative integer $n$
such that $\pi^n(1^B_j) = 1^B_{m_2}$
with $n \leq b$.
\end{proposition}

\begin{notation}
For the ABS $A$ associated to $N_{m_1,n_1}$,
we define sub-paths $P$ and $Q$ of $A$ by the following:
%\begin{eqnarray*}
%P &=& \{\pi^n(1^A_{m_1}) \mid 0 \leq n \leq \lambda_1,\
%\text{with }\pi^{\lambda_1}(1^A_{m_1}) = 0^A_{2m_1+1}\}, \\
%Q &=& \{\pi^n(0^A_{m_1+1}) \mid 0 \leq n \leq \lambda_2,\ 
%\text{with }\pi^{\lambda_2}(0^A_{m_1+1}) = 0^A_{2m_1}\}.
%\end{eqnarray*}
%We use paths, and the above sets denoted as
\begin{eqnarray*}
P &:& 1^A_{m_1} \rightarrow 0^A_{h_1} \rightarrow 
\cdots \rightarrow 0^A_{2m_1+1},\\
Q &:& 0^A_{m_1+1} \rightarrow 1^A_1 \rightarrow 
\cdots \rightarrow 0^A_{2m_1}.
\end{eqnarray*}
Clearly $A$ is a disjoint union of $P$ and $Q$ as sets.
\end{notation}
%% DefinitionからNotationに変更

The above paths $P$ and $Q$ are useful.
For instance, in the case (a) of Proposition~\ref{ThmOfCases},
for the ABS $S_\rho = (\tilde S_\rho, \delta_\rho, \pi_\rho)$
which is associated to $N_\rho$ of Proposition~\ref{PropOfSpe},
the set $\tilde S_\rho$ consists of all elements of $P$.
Moreover, the path $Q$ has the following property:

\begin{lemma}\label{LemOfCQ}
The set $C$ is contained in $Q$.
\end{lemma}

\begin{proof}
Suppose that there exists an element $0^A_i$ of $C$ 
which belongs to $P$.
Then there exists a natural number $n$ such that 
$\pi^n(0^A_i) = 0^A_{m_1+1}$ with $n < a$.
This contradicts with the latter statement of Proposition~\ref{PropOfm1+1m1}.
\end{proof}

Here, let us consider the construction of the ABS
obtained by (a) or (b) of Proposition~\ref{ThmOfCases}.
The ABS $S^-$ obtained by
the small modification by $0^A_i$ and $1^B_j$
is described as the following diagram:
$$\xymatrix@=15pt{\bullet \ar[d] & \cdots\ar[l] & 
1^A_{m_1}\ar[l] & 0^A_{2m_1}\ar[l] & \cdots\ar[l] & 0^A_i\ar[l] & 
0^B_{j+m_2}\ar[l] & \cdots\ar[l] \\
\bullet \ar[r] & \cdots\ar[r] & 0^A_{2m_1+1}\ar[r] & 
0^A_{m_1+1}\ar[r] & 
\cdots\ar[r] & 0^A_{i+m_1}\ar[r] & 
1^B_j\ar[r] & \cdots\ar[u]_{\text{elements of }B}}$$
First, let us consider the case (a).
By constructing the small modification by $0^A_{m_1+1}$ and $1^A_{m_1}$, 
images of $0^A_{2m_1+1}$ and $0^A_{2m_1}$
are switched, and we obtain the ABS
which consists of two components as follows:
$$\xymatrix@=15pt{\bullet \ar[d]& \cdots\ar[l] & 
1^A_{m_1}\ar[l] & 0^A_{2m_1}\ar[d] \ar@{.>}[l]|{\times} & 
\cdots \ar[l] & 0^A_i\ar[l] & 
0^B_{j+m_2}\ar[l] & \cdots\ar[l] \\
\bullet \ar[r] & \cdots\ar[r] & 
0^A_{2m_1+1}\ar[u] \ar@{.>}[r]|{\times}& 
0^A_{m_1+1}\ar[r] & \cdots\ar[r] & 0^A_{i+m_1}\ar[r] & 
1^B_j\ar[r] & \cdots\ar[u]_{\text{elements of }B}}$$
The former component consists of all elements of $P$.
The ${\rm DM_1}$ corresponding to this component 
is described as $N_\rho$ with a Newton polygon $\rho = (f, g)$.
Since this component coincides with the component obtained from $A$ by
applying \cite[Lemma~5.6]{HarashitaCon} to 
the adjacent $1_{m_1}^A\ 0_{m_1+1}^A$,
we have $f n_1 - g m_1 = 1$.
Next, let us see the case (b).
By constructing the small modification by $0^A_{i-1}$ and $1^B_j$, 
images of $0^A_{i-1+m_1}$ and $0^A_{i+m_1}$
are switched, and we obtain the ABS
which consists of two components:
$$\xymatrix{\bullet \ar[d] & \cdots \ar[l] & 
0^A_{i-1}\ar[l] &
0^A_{i-1+m_1}\ar[d] \ar@{.>}[l]|{\times}&
\cdots\ar[l] & \bullet \ar[l] 
\\ \bullet \ar[r] & \cdots \ar[r] & 
0^A_{i+m_1}\ar[u] \ar@{.>}[r]|{\times} & 
1^B_j\ar[r]_{\ \ \ \ \ \ \ \ \ \ \ \ \ \ \ \ \, \text{elements of }B} & 
\cdots\ar[r] & 0^A_i\ar[u]}$$
The latter component contains elements $\pi^n(0^A_i)$
for all non-negative integers $n$ with $n \leq a$.
The former component corresponds to the ${\rm DM_1}$ $N_\rho$
with a Newton polygon $\rho = (f, g)$.
%since the elements of this cycle can be regarded as a subset of $\tilde A$.
%with the set of symbols $\tilde A$ of $A$.
Since this component coincides with the component obtained from $A$ by
applying \cite[Lemma~5.6]{HarashitaCon} to 
the adjacent $0_{i-1}^A\ 0_{i}^A$,
we have $f n_1 - g m_1 = 1$.

Hence for the cases (a) and (b), the ABS $S^{--}$ has
two components, where one is associated to the ${\rm DM_1}$ $N_\rho$.
We will show that for the other component,
there exists a Newton polygon $\xi'$ such that
this component corresponds to the ${\rm DM_1}$ $N_{\xi'}^-$
satisfying \eqref{EqOfInduction}
of Proposition~\ref{ThmOfCases}.

\begin{definition}
We define $C_1$ (resp. $C_2$) to be the subset of $C'$ consisting of
elements $0^A_i$ satisfying that
for $N_{\xi}^-$ obtained by a small modification by $0^A_i$ and $1^B_j \in D'$,
we have the equality \eqref{EqOfInduction}
by the case (a) (resp. (b)) of Proposition~\ref{ThmOfCases}.
\end{definition}

%Let $S^{--}$ be the ABS associated with $N_\xi^{--}$
%obtained by (a) or (b) of Proposition~\ref{ThmOfCases}.
We denote by $S''$ the ABS obtained by 
the small modification by $0^A_{m_1+1}$ and $1^A_{m_1}$
or the small modification by $0^A_{i-1}$ and $1^B_j$ for $S^-$.
By the above, we see that the ABS $S''$ consists of two components,
and a component of $S''$ is associated to $N_\rho$.
Let $\Psi$ be the other component of $S''$.
%Put $N_{\xi}^{--} = N \oplus N_{\rho}$.
%Let $\Psi$ be the ABS associated with $N$. 
To show $C_1 \cup C_2 = C'$,
we give a condition that an element $t$ of $C'$
belongs to $C_1$ or $C_2$ in Proposition~\ref{PropOfSaturated}.
In the proof of this proposition, we give a method to
determine the structure of the ABS $\Psi$.

\begin{proposition}\label{PropOfSaturated}
Let $S^-$ be the generic specialization 
obtained by a small modification by $0^A_i$ and $1^B_j$.
If $\Psi$ contains no element $t$ satisfying that
$0^A_{m_1+1} < t < 1^A_{m_1}$ (resp. $0^A_{i-1} < t < 1^B_j$) in $S^-$,
then $0^A_i$ belongs to $C_1$ (resp. $C_2$).
\end{proposition}

\begin{proof}
Let us see the case that we construct $S''$
by the small modification by $0^A_{m_1+1}$ and $1^A_{m_1}$.
The other case
($S''$ is constructed by the small modification by $0^A_{i-1}$ and $1^B_j$)
is shown by the same way.
For the ABS $S^- = (\tilde S', \delta, \pi)$,
we construct the small modification $\pi''$ by $0^A_{m_1+1}$ and $1^A_{m_1}$,
and we obtain the admissible ABS $S'' = (\tilde S'', \delta, \pi'')$
consisting of two components.
A component
corresponds to $N_{\rho}$ with $\rho = (f, g)$ satisfying $fn_1 - gm_1 = 1$.
Let $N$ be the ${\rm DM_1}$ associated to the other component $\Psi$.
To see that there exists a Newton polygon $\xi'$
such that $N = N_{\xi'}^-$,
we consider the small modification by $1^B_j$ and $0^A_i$ in $S''$.
%We regard the maps $\delta : \tilde S'' \to \{0, 1\}$ and $\pi'' : \tilde S'' \to \tilde S''$ as 
%maps $\tilde \Psi \to \{0, 1\}$ and $\tilde \Psi \to \tilde \Psi$.
Let $\chi$ be the small modification by $1^B_j$ and $0^A_i$,
%Constructing the small modification by $1^B_j$ and $0^A_i$, 
and we obtain the map $\chi$ on $\tilde \Psi$.
%swapping images of $\pi''^{-1}(1^B_j)$ and $\pi''^{-1}(0^A_i)$.
For the ordered set $\{t_1 < \dots < t_{h'}\}$
of $\Psi$,
put $\Psi^{(0)} = \{t'_1 <' \dots <' t'_{h'}\}$,
where if $0^A_i$ and $1^B_j$ are $i'$-th element and $j'$-th element of $\tilde \Psi$
respectively,
then for $s = (i', j')$ transposition,
we set $t'_z = t_{s(z)}$.
We have the ABS $(\Psi^{(0)}, \delta, \chi)$.
Here, we define sets
\begin{eqnarray*}
\mathcal D^{(0)} &=& \{t \in \Psi^{(0)} \mid 0^A_i < t \text{ and } 
\chi(t) < \chi(0^A_i) \text{ in } \Psi^{(0)}
\text{ with } \delta(t) = 0\},\\
\mathcal E^{(0)} &=& \{t \in \Psi^{(0)} \mid t < 1^B_j \text{ and } 
\chi(1^B_j) < \chi(t) \text{ in } \Psi^{(0)}
\text{ with } \delta(t) = 1\},
\end{eqnarray*}
Put $\alpha_n = \chi^n(0^A_i)$ and $\beta_n = \chi^n(1^B_j)$
for all non-negative integers $n$.
For sets $\Psi^{(0)}, \dots, \Psi^{(n-1)}$ 
and sets $\mathcal D^{(0)}, \dots, \mathcal D^{(n-1)}$ with a natural number $n$, 
%we define a set $\Psi^{(n)}$ as follows.
let $\Psi^{(n)} = \Psi^{(n-1)}$ as sets.  
We define the order on $\Psi^{(n)}$ to be
for $t < t'$ in $\Psi^{(n-1)}$,
we have $t > t'$ if and only if 
$\pi(t_{\rm min}) \leq t < \alpha_n$ in $\Psi^{(n-1)}$ and $t' = \alpha_n$,
where $t_{\text{min}}$ is the minimum element of $\mathcal D^{(n-1)}$.
In other words, in the ABS $\Psi^{(n-1)}$, 
we move the element $\alpha_n$
to the left of $\chi(t_{\text{min}})$ to construct $\Psi^{(n)}$.
We regard $\chi$ as a map on $\Psi^{(n)}$.
Then we obtain the ABS $(\Psi^{(n)}, \delta, \chi)$.
We define a set
$$\mathcal D^{(n)} = \{t \in \Psi^{(n)} \mid \alpha_n < t \text{ and } 
\chi(t) < \alpha_{n+1} \text{ in } \Psi^{(n)} \text{, with } \delta(t) = \delta(\alpha_n)\}.$$
By hypothesis, we have $\mathcal D^{(n)} = \mathcal A^{(n)} \setminus \tilde S_\rho$
for all $n$.
Here, $S_\rho = (\tilde S_\rho, \delta_\rho ,\pi_\rho'')$ is the ABS associated to $N_\rho$,
where $\delta_\rho$ (resp. $\pi''_\rho$) denotes 
the restriction of $\delta$ (resp. $\pi''$)
to $\tilde S_\rho$.
Hence there exists the smallest integer $a'$ such that
$\mathcal D^{(a')} = \emptyset$.
For sets $\Psi^{(a')}, \dots, \Psi^{(a'+n-1)}$ 
and sets $\mathcal E^{(0)}, \dots, \mathcal E^{(n-1)}$,
%we define $\Psi^{(a'+n)}$ 
%by $\Psi^{(a'+n-1)}$ as follows.
let $\Psi^{(a'+n)} = \Psi^{(a'+n-1)}$ as sets.
The ordering of $\Psi^{(a'+n)}$ is given so that
for $t < t'$ in $\Psi^{(a'+n-1)}$,
we have $t > t'$ if and only if 
$\beta_n < t' \leq \pi(t_{\rm max})$ in $\Psi^{(a'+n-1)}$ and $t = \beta_n$,
where $t_{\text{max}}$ is the maximum element of $\mathcal E^{(n-1)}$.
In other words, to obtain $\Psi^{(a'+n)}$,
we move the element $\beta_n$ to the right of $\chi(t_{\text{max}})$.
We regard $\chi$ as a map on $\Psi^{(a'+n)}$.
Thus we obtain the ABS $(\Psi^{(a' + n)}, \delta, \chi)$.
We define a set
$$\mathcal E^{(n)} = \{t \in \Psi^{(n)} \mid t < \beta_n \text{ and }
\beta_{n+1} < \chi(t) \text{ in } \Psi^{(a'+n)} \text{, with } \delta(t) = \delta(\beta_n)\}.$$
By hypothesis, we have $\mathcal E^{(n)} = \mathcal B^{(n)}$ for all $n$.
Hence there exists the smallest integer $b'$ such that
$\mathcal E^{(b')} = \emptyset$.
We obtain the admissible ABS $\Psi' = (\Psi^{(a'+b')}, \delta, \chi)$ 
which is associated to $N_{\xi'}$ with $\xi' = (m_1-f, n_1-g) + (m_2, n_2)$.
%$N_{m_1-a, n_1-b}\oplus N_{m_2,n_2}$.
We immediately obtain 
$\ell(\Psi') = \ell(\Psi) + 1$.
Hence the ABS $\Psi$ corresponds to $N^-_{\xi'}$.
\end{proof}

\begin{example}\label{ExOfSpe}
Let us see an example of constructing 
$N_{\xi'}$ and $N_{\rho}$ from $N_{\xi}^-$.
Let $N_{\xi} = N_{2,7} \oplus N_{5,3}$,
and we construct the full modification by $0^A_6$ and $1^B_3$.
By Example~\ref{ExOfExchange},
we have $N_{\xi}^-$ as follows:\\
\\
$$N^-_{\xi} : \xymatrix@=3pt{1^A_1\ar@/_20pt/[rrrrrrrrrr] & 0^A_3\ar@/_20pt/[l] &
1^A_2\ar@/_20pt/[rrrrrrrrr] & 0^A_5\ar@/_20pt/[ll] & 
0^A_4\ar@/_20pt/[ll] &
1^B_3\ar@/_20pt/[rrrrrrr] & 
0^A_7\ar@/_20pt/[lll] &
1^B_1\ar@/_20pt/[rrrrrr] & 
1^B_2\ar@/_20pt/[rrrrrr] &
0^A_6\ar@/_20pt/[lllll] & 0^A_8\ar@/_20pt/[lllll] &
0^A_9\ar@/_20pt/[lllll] &0^B_6\ar@/_20pt/[lllll] &
1^B_4\ar@/_20pt/[rr] & 1^B_5\ar@/_20pt/[rr] &
0^B_7\ar@/_20pt/[lllllll] & 0^B_8\ar@/_20pt/[lllllll]}.$$
\\ \\
We construct the small modification by elements $0^A_3$ and $1^A_2$
for the ABS associated to $N_\xi^-$, 
and we obtain $N_{\xi}^{--}$ by the following diagram
which is decomposed into two cycles:
\\ \\
$$\xymatrix@=3pt{N_{\xi}^{--}: &1^A_1\ar@/_20pt/[rrrrrrr] &0^A_3\ar@/_20pt/[l]
&0^A_4\ar@/_20pt/[l] &1^B_3\ar@/_20pt/[rrrrr] &1^B_1\ar@/_20pt/[rrrrr]
&1^B_2\ar@/_20pt/[rrrrr]
&0^A_6\ar@/_20pt/[llll] &0^A_8\ar@/_20pt/[llll] &0^B_6\ar@/_20pt/[llll]
&1^B_4\ar@/_20pt/[rr] &1^B_5\ar@/_20pt/[rr] &0^B_7\ar@/_20pt/[llllll]
& 0^B_8\ar@/_20pt/[llllll] & \oplus
&1^A_2\ar@/_20pt/[rrr] &0^A_5\ar@/_20pt/[l] &0^A_7\ar@/_20pt/[l]
&0^A_9\ar@/_20pt/[l]}.$$
\\ \\
It is clear that the latter component is associated to 
the simple ${\rm DM_1}$ $N_{1,3}$.
Let us see that the former component $N$ is a specialization of a minimal ${\rm DM_1}$.
We construct the small modification by elements $1^B_3$ and $0^A_6$,
and we obtain the following diagram.\\
\\
$$\xymatrix@=5pt{1^A_1\ar@/_20pt/[rrrr] &0^A_3\ar@/_20pt/[l]
&0^A_4\ar@/_20pt/[l] &0^A_6\ar@/_20pt/[l] &0^A_8\ar@/_20pt/[l]
&\oplus &1^B_1\ar@/_30pt/[rrrr]|\times &1^B_2\ar@/_30pt/[rrrr]|\times
&1^B_3\ar@/_20pt/[r]|\circ &0^B_6\ar@/_20pt/[lll]
&1^B_4\ar@/_20pt/[rr] &1^B_5\ar@/_20pt/[rr] &0^B_7\ar@/_20pt/[lllll]
& 0^B_8\ar@/_20pt/[lllll]}.$$
Clearly the former summand is associated to $N_{1, 4}$.
For the notation of Proposition~\ref{PropOfSaturated},
we have $\mathcal D^{(0)} = \emptyset$ and $\mathcal E^{(0)} = \{1^B_1, 1^B_2\}$.
We move the element $0^B_6$ to 
between $1^B_5$ and $0^B_7$.
Thus we see that the latter summand is associated to the simple ${\rm DM_1}$ $N_{5, 3}$.
Therefore we see that
this diagram corresponds to $N_{1, 4} \oplus N_{5, 3}$,
and we have $N_\xi^{--} = N_{(1,4)+(5,3)}^- \oplus N_{1,3}$.
\end{example}

Here, for an element $t$ of $C$,
we give a condition of 
$t$ belongs to $C_2$ as follows:

\begin{proposition}\label{PropOfIffi-1j}
For a good exchange of  $0^A_i \in C$ and $1^B_j \in D'$,  
we have an element $t$ of $\tilde S'$ satisfying $0^A_{i-1} < t < 1^B_j$ in $S^-$
if and only if there exists a non-negative integer $n$ such that 
the set $\pi(\mathcal A^{(n)})$ has the maximum element $0^A_{i-1}$.
%Moreover, $\pi(\mathcal A^{(n)})$ contains $0^A_{i-1}$
%if and only if $0^A_{i-1}$ is the maximum element of $\pi(\mathcal A^{(n)})$.
\end{proposition}

\begin{proof}
%First, let us see the former statement.
In the ABS $(S^{(0)}, \delta, \pi)$,
clearly there exists no element $t$ satisfying $0^A_{i-1} < t < 1^B_j$.
Hence every element $t$ between $0^A_{i-1}$ and $1^B_j$
in $S^-$ is of the form $\pi^m(0^A_i)$ for a natural number $m$ with $m \leq a$.
Fix a natural number $n$.
By definitions of $\mathcal A^{(n)}$ and $S^{(n)}$,
there exists an element $t$ between $0^A_{i-1}$ and $1^B_j$ in $S^{(n+1)}$
if and only if $\pi(\mathcal A^{(n)})$ has the maximum element $0^A_{i-1}$.
Indeed, the element $t$ is obtained by $t = \pi^{n+1}(0^A_i)$.
%
%Let us see the latter statement.
%Assume that there exists an element $t$ of $\pi(\mathcal A^{(n)})$
%such that $0^A_{i-1} < t$. 
%Since we have then $0^A_{i-1} < 1^B_j \leq t$,
%the set $\mathcal A^{(n)}$ contains the inverse image of $1^B_j$.
%It is a contradiction.
%The converse is obvious.
\end{proof}

%\begin{lemma}\label{LemOfm1+1}
%there exists no non-negative integer $v$ such that 
%$\mathcal A^{(v)}$ contains the element $0^A_{m_1+1}$.
%\end{lemma}
%\begin{lemma}\label{LemOfm2}
%there exists no non-negative integer $v$ such that
%$\mathcal B^{(v)}$ contains $1^B_{m_2}$.
%\end{lemma}
%\begin{proof}
%A proof is given by the same way as Lemma~\ref{LemOfm1+1}.
%\end{proof}

We consider the case that $C$ is not empty. 
We set an order on the set $C$ which plays an important role.

\begin{notation} \label{NotationOfCQ}
%We use notation of  Lemma~\ref{LemOfFirst}. 
%By this lemma,
%$t_{i_1}$ is the first element
%of $C$ which appears in path $Q$.
Put $\nu = |C|$.
For $x = 1, \dots, \nu$,
let $i_x$ be the natural number with $m_1 < i_x \leq h_1$ such that
$0^A_{i_x}$ is the element of $C$ appearing in the $x$-th in the path $Q$.
In the other words,
we put $C = \{0^A_{i_1}, 0^A_{i_2}, \dots, 0^A_{i_\nu}\}$,
where for elements 
$0^A_{i_x} = \pi^{l_x}_0(0^A_{m_1+1})$ and $0^A_{i_y} = \pi^{l_y}_0(0^A_{m_1+1})$ of $C$,
we have $x < y$ if and only if $l_x < l_y$.
The elements $0^A_{i_x}$, for $x = 1, \dots, \nu$, of $C$ appear in the path $Q$ as follows.
$$Q : 0^A_{m_1+1} \rightarrow \cdots \rightarrow 0^A_{i_1}
\rightarrow \cdots \rightarrow 0^A_{i_2} \rightarrow \cdots \rightarrow 0^A_{i_\nu}
\rightarrow \cdots \rightarrow 0^A_{2m_1}.$$
\end{notation}
%% このNotationの位置を移動

Here, we give a characterization of the ``first'' element $0^A_{i_1}$
of $C$ in Lemma~\ref{LemOfFirst}.

\begin{lemma}\label{LemOfFirst}
If there exists the minimum number $y$ such that
the element $t = \pi^y_0(0^A_{m_1+1})$ of $Q$
satisfies $0^A_{m_1+1} < t < 0^A_{n_1}$ in $A$,
then $t = 0^A_{i_1}$.
\end{lemma}

\begin{proof}
First, let us see that for an small modification by $0^A_i$ and $1^B_j$,
there exists no non-negative integer $n$ such that 
$\mathcal A^{(n)}$ contains the element $0^A_{m_1+1}$.
If $\mathcal A^{(n)}$ contains $0^A_{m_1+1}$ for a non-negative integer $n$,
then $\pi^{n+1}(0^A_i) < 1^A_{1}$ holds.
This is a contradiction.

To see that $t$ belongs to $C$,
for the ABS $S$ of $N_\xi$,
we consider an small modification $\pi$ by $t$ and $1^B_j$ with $1^B_j \in D'$.
Let $S^{(n)}$ and $\mathcal A^{(n)}$ denote sets 
obtained by the small modification by $t$ and $1^B_j$.
Assume that for a non-negative integer $n$, 
the set $\mathcal A^{(n)}$ contains 
the inverse image $t'$ of $1^B_j$ for $\pi$.
In the ABS $A$, this $t'$ is the inverse image of $t$.
It is clear that $t'$ belongs to $Q$.
By Corollary~\ref{CoroOfT}, 
there exists an element $t''$ of $\mathcal A^{(0)}$ such that $\pi^n(t'') = t'$. 
By definition of $t$,
and since elements $1^A_{m_1}$ and $0^A_{n_1}$ belong to $P$,
elements $\tau^A_x$, with $\tau = 0$ or $1$, 
of the path $Q$ between $0^A_{m_1+1}$ and $t$
satisfies $x < m_1$ or $n_1 < x$.
It implies that these elements do not belong to 
$\mathcal A^{(0)} \subset \{0^A_{m_1+1}, \dots, 0^A_{n_1}\}$.
Hence $t''$ belongs to $P$.
Then the path from $t''$ to $t'$ through $0^A_{m_1+1}$,
i.e., there exists a natural number $m$ such that $\pi^m(t'') = 0^A_{m_1+1}$
with $m < a$.
This contradicts with the above property.
\end{proof}

%If there exists no such a number $y$ of Lemma~\ref{LemOfFirst},
%then the set $C$ is empty.
We consider the case $C \neq \emptyset$.
Then we have the element $0^A_{i_1}$ of $C$.
To apply Proposition~\ref{PropOfSaturated},
it is important to study elements between $0^A_{m_1+1}$
and $1^A_{m_1}$ in $S^-$.
The following set given in Notation~\ref{NotOfTi} and the element of $C$ 
given in Proposition~\ref{PropOfC0} are useful
for studying the elements between $0^A_{m_1+1}$ and $1^A_{m_1}$ in $S^-$.

\begin{notation} \label{NotOfTi}
For an element $0^A_i$ of $C'$,
we often write 
$\mathcal A_i^{(n)}$ for sets $\mathcal A^{(n)}$ obtained by 
the small modification by $0^A_i$ and $1^B_j$ to avoid confusion.
Moreover, we often write $a_i$
for the smallest non-negative integer $a$ satisfying $\mathcal A_i^{(a)} = \emptyset$.
For sets $\{ \mathcal A_i^{(n)}\}_{n = 0, \dots, a_i}$ with $0^A_i \in C$,
we set $$T_i = \pi(\mathcal A_i^{(a_i-1)}).$$
This set consists of all elements $t$ satisfying
$0^A_{m_1+1} \leq t < 1^A_{m_1}$ in $S^-$.
\end{notation}

\begin{comment}
\begin{proposition}\label{PropOfChain}
For non-negative integers $x$ and $y$,
if $x < y$, then $T_{i_x} \subset T_{i_y}$ holds.
\end{proposition}

\begin{proof}
By assumption, there exists a natural number $n$ 
such that $\pi^n(0^A_{i_x}) = 0^A_{i_y}$.
We have then $T_{A, i_x}^{(n)} = \{0^A_{i_y+1}, \dots, 0^A_{a}\}$,
with $a \leq n_1$.
Hence we have $T_{A, i_x}^{(n+m)} \subset T_{A, i_y}^{(m)}$
for all non-negative integers $m$,
and we obtain desired relation.
\end{proof}
\end{comment}

\begin{proposition}\label{PropOfC0}
Put $i = n_1 - \gamma$,
with $\gamma = |T_{i_1}|$.
Then $0^A_i$ belongs to $C$.
Moreover, $T_{i_1} = T_i$ holds. 
\end{proposition}

\begin{proof}
Since $T_{i_x}$ consists of some elements of $\{0^A_{m_1+1}, \dots, 0^A_{n_1}\}$,
we have $|T_{i_x}| < n_1 - m_1$ for all $0^A_{i_x} \in C$.
Hence we have $m_1 < i < n_1$.
To show that $0^A_i$ belongs to $C$,
it suffices to see that 
there exists a natural number $m$
such that $\mathcal A_i^{(0)} = \mathcal A_{i_1}^{(m)}$.
In fact, if this statement is true, 
then $\mathcal A_i^{(n)} = \mathcal A_{i_1}^{(m+n)}$ holds
for all $n$ with $0 \leq n \leq a_i$.
Put $\alpha = \pi^{m-1}(0^A_{i_1})$.
Note that $\alpha$ is the inverse image of $0^A_i$
in the ABS's $(S^{(n)}, \delta, \pi)$,
where $\pi$ is the small modification by $0^A_i$ and $1^B_j$.
By Proposition~\ref{PropOf1}, 
sets $\mathcal A_{i_1}^{(m+n)}$ does not contain $\alpha$.
Hence sets $\mathcal A^{(n)}_i$ do not contain the inverse image of $0^A_i$
for all $n$, and this completes the proof.
Let us see that there exists such a number $m$.
By the definition of $\gamma$ and Lemma~\ref{LemOfDecrease}, 
there exists a natural number $m'$ such that
$\mathcal A_{i_1}^{(m')} = \{1^A_{m_1-\gamma+1}, \dots, 1^A_{m_1}\}$.
We have then $\mathcal A_{i_1}^{(m'+2)} = \{0^A_{n_1 - \gamma+1}, \dots, 0^A_{n_1}\}$,
and this set is equal to $\mathcal A_{i}^{(0)}$.
Therefore the number $m$ is obtained by $m = m'+2$.
The latter part $|T_{i_1}| = |T_i|$ follows from the former part.
\end{proof}

Let $d$ be a natural number satisfying $i_d = n_1 - \gamma$.
From now on, we fix notations of 
the non-negative integers $d$ and $\gamma$.
As can be seen in Proposition~\ref{PropOfDivide},
this natural number $d$ plays an important role 
to show the former part of Proposition~\ref{ThmOfCases}.

\begin{proposition}\label{PropOfDivide}
Let $0^A_{i_x}$ be an element of $C$. 
For the above natural number $d$, we have
\begin{itemize}
\item[(1)] If $x \leq d$, then 
$0^A_{i_x}$ belongs to $C_1$,
\item[(2)]If $x > d$, then 
$0^A_{i_x}$ belongs to $C_2$. 
\end{itemize}
\end{proposition}

Thanks to this proposition, we can show the case of $n_1 > m_1 + 1$
for Proposition~\ref{ThmOfCases}.

\begin{proof}[Proof of Proposition~\ref{ThmOfCases} (I)]
If Proposition~\ref{PropOfDivide} is true,
then the equality \eqref{EqOfInduction} of Proposition~\ref{ThmOfCases} holds 
for ${\rm DM_1}$'s $N^-_{\xi}$ which are obtained by small modifications by
elements $0^A_i$ and $1^B_j$ of 
$C = C' \setminus \{0^A_{n_1}\}$ and $D'$ respectively.
Let us see the case of $i = n_1$.
In this case, we have $\mathcal A^{(0)} = \emptyset$,
and there exists no element $t$ of $S^-$ satisfying 
$0^A_{i-1} < t < 1^B_j$ in $S^-$.
Hence we obtain desired $N_\xi^{--} = N_{\xi'}^- \oplus N_\rho$ by the case (b). 
\end{proof}

Let us show some properties of sets $T_i$,
the natural numbers $d$ and $\gamma$.
These properties are used for showing Proposition~\ref{PropOfDivide}.

\begin{proposition}\label{PropOfPiandd}
The following are true:
\begin{itemize}
\item[(i)]If $x < y$, then $T_{i_x} \subset T_{i_y}$ holds;
\item[(ii)]For all $n$ with $n < a_{i_d}$, we have $|\mathcal A_{i_d}^{(n)}| = \gamma$;
\item[(iii)]For all $x$ and $n$ with $n < a_{i_x}$, 
we have $|\mathcal A_{i_x}^{(n)}| \geq \gamma$;
\item[(iv)]$T_{i_d} \subsetneq T_{i_x}$ holds 
for all $x$ with $x >d$;
\item[(v)]$T_{i_x} = T_{i_d}$ is true if and only if 
$x \leq d$.
\end{itemize}
\end{proposition}

\begin{proof}
Let us see (i).
By assumption, there exists a natural number $n$ 
such that $\pi^n(0^A_{i_x}) = 0^A_{i_y}$ and $n < a_{i_x}$.
We have then $\mathcal A_{i_x}^{(n)} = \{0^A_{i_y+1}, \dots, 0^A_z\}$,
with $z \leq n_1$.
Hence $\mathcal A_{i_x}^{(n)} \subset \mathcal A_{i_y}^{(m+n)}$ holds
for all non-negative integers $m$,
and it induces the desired relation.

It is clear that $|\mathcal A^{(n)}| \geq |\mathcal A^{(n+1)}|$
for all $n$ and all elements $0^A_i$ of $C$.
By the definition of $i_d$ and Proposition~\ref{PropOfC0}, 
we have $|\mathcal A_{i_d}^{(0)}| = |\mathcal A_{i_d}^{(a-1)}| = \gamma$
for $a = a_{i_d}$.
Thus (ii) holds.

By (i) and the definition of $\gamma$,
we have $|T_{i_x}| \geq \gamma$ for all $x$.
Moreover, it is clear that $|\mathcal A_{i_x}^{(n)}| \geq |T_{i_x}|$
for all $n$ and $x$.
Hence we see (iii).

To see (iv), we fix a natural number $x$ satisfying $x > d$.
It suffices to see that $|T_{i_x}|$ is greater than $\gamma$.
Suppose that $|T_{i_x}| = \gamma$.
Then there exists the minimum number $u$ such that
$|\mathcal A_{i_x}^{(u)}| = \gamma$.
We have $\mathcal A_{i_x}^{(u)} = 
\{1^A_{m_1-\gamma +1}, \dots, 1^A_{m_1}\}$
and $\pi^u(0^A_{i_x}) = 1^A_{m_1 - \gamma}$.
Then we have a sub-path of $Q$:
$$0^A_{i_d} \rightarrow \cdots \rightarrow 0^A_{i_x} \rightarrow
\cdots \rightarrow 1^A_{m_1-\gamma} \rightarrow 0^A_{h_1-\gamma}
\rightarrow 0^A_{n_1-\gamma}$$
which implies that there exists a natural number $m$ such that 
$\pi^m(0^A_{i_d}) = 0^A_{i_d}$
for $m < m_1 + n_1$.
This is a contradiction.

The statement (v) follows from
Proposition~\ref{PropOfC0}, (i) and (iv).
\end{proof}

%d' = 0の場合に注意して読め．

%\begin{lemma}
%For the above $d'$,
%if $x \leq d'$, then there exists a non-negative integer $u$
%satisfying that $0^A_{i_x-1}$ belongs to $T_{A, i_x}^{(u)}$.
%\end{lemma}

The natural number $d'$ introduced in Notation~\ref{DefOfd'}
has a relation with the natural number $d$.
We give the proof of Proposition~\ref{PropOfDivide}
using this relation.

\begin{notation} \label{DefOfd'}
For the admissible ABS $S$ of $N_\xi$
and the set $C$,
we define a set $\mathcal L$ by
$$\mathcal L = \{x \in \bar C \mid 
\text{For a non-negative integer } n = n(x),
\text{ the maximum element of } \mathcal A_{i_d}^{(n)} \text{ is } 0^A_{i_x-1} \},$$
where $\bar C = \{x \in \mathbb N \mid 1 \leq x \leq |C|\}$.
Let $d'$ denote the maximum element of $\mathcal L$.
If $\mathcal L$ is empty, then we set $d' = 0$.
\end{notation}
%% DefinitionからNotationに変更

\begin{proposition} \label{PropOfd'L}
For the set $\mathcal L$ and the natural number $d'$ of Notation~\ref{DefOfd'},
if $\mathcal L$ is not empty, then $\mathcal L = \{1, 2, \dots, d'\}$.
\end{proposition}

\begin{proof}
Fix a natural number $x$ satisfying $x \leq d'$.
We show that $0^A_{i_x-1}$ belongs to $\mathcal A_{i_d}^{(u)}$
for a non-negative integer $u$. 
Let $n$ be the non-negative integer satisfying that 
the maximum element of $\mathcal A_{i_d}^{(n)}$ is 
$0^A_{i_{d'}-1}$.
Let us consider the path consisting of maximum elements of $\mathcal A_{i_d}^{(m)}$
for $0 \leq m \leq n$
\begin{eqnarray} \label{PathOfd'}
0^A_{n_1} \to \cdots \to 0^A_{i_{d'}-1}.
\end{eqnarray}
For the sub-path of $Q$
$$0^A_{m_1+1} \to 1^A_1 \to 0^A_{n_1+1} \to \cdots \to 0^A_{i_x}
\to \cdots \to 0^A_{i_{d'}},$$
since this path contains $0^A_{i_x}$,
the path \eqref{PathOfd'} contains $0^A_{i_x-1}$.
This completes the proof.
\end{proof}

\begin{proposition}\label{PropOfd'd}
For the natural number $d'$ of Notation~\ref{DefOfd'}
and the natural number $d$ obtained by Proposition~\ref{PropOfC0}, 
we have $d' \leq d$.
\end{proposition}

\begin{proof}
First, let us see that
there exists no non-negative integer $n$ such that
the maximum element of $\mathcal A_{i_d}^{(n)}$ is $1^A_{m_1-1}$.
Assume that the set $\mathcal A_{i_d}^{(n)}$ has 
the maximum element $1^A_{m_1-1}$ for a non-negative integer $n$.
Then this set has the minimum element $1^A_{m_1-\gamma}$,
and we have $\pi(1^A_{m_1-\gamma}) = 0^A_{i_d + m_1}$.
It implies that the set $\mathcal A_{i_d}^{(n+1)}$ contains
the inverse image of $0^A_{i_d}$,
and this is a contradiction.

Suppose $d < d'$,
and we fix a natural number $x$ satisfying $d < x \leq d'$.
Let us consider the ABS $(\tilde S', \delta, \pi)$ obtained by 
the small modification by $0^A_{i_d}$ and $1^B_j \in D'$.
By Proposition~\ref{PropOfd'L},
we obtain the path which consists of maximum elements of $\mathcal A_{i_d}^{(n)}$ 
for all $n$ and $T_{i_d}$ as follows:
\begin{eqnarray}
0^A_{n_1} \rightarrow \cdots \rightarrow 0^A_{i_x-1} 
\rightarrow \pi(0^A_{i_x-1}) \rightarrow
\cdots \rightarrow 0^A_{m_1+\gamma}. \label{PathOfMaxElements}
\end{eqnarray}
Let us consider the sub-path of $A = (\tilde A, \delta_A, \pi_A)$:
\begin{eqnarray}
0^A_{i_x} \rightarrow \pi_A(0^A_{i_x}) \rightarrow \cdots \rightarrow 
0^A_{m_1+\gamma+1}. \label{PathOfPin}
\end{eqnarray}
Since the path \eqref{PathOfMaxElements}
does not contain $1^A_{m_1-1}$,
the path \eqref{PathOfPin} does not contain $1^A_{m_1}$.
Hence for the ABS $(\tilde S', \delta, \pi_x)$ 
obtained by the small modification by $0^A_{i_x}$ and $1^B_j$ for $S$,
we have a natural number $n$ satisfying that 
$\pi_x^n(0^A_{i_x}) = 0^A_{m_1+\gamma+1}$ with $n < a_{i_x}$.
By Proposition~\ref{PropOfPiandd} (iv), the set
$T_{i_x}$ contains $0^A_{m_1+\gamma+1}$.
This contradicts with Proposition~\ref{PropOf1}.
\end{proof}

\begin{comment}

\begin{lemma}\label{LemOfMaxElement}
Let $t$ be the maximum element of $\mathcal A^{(n)}$.
Then there exists an integer $m$ such that
$t = \pi^m(0^A_{n_1})$ with $-2 \leq m \leq n$.
\end{lemma}

\begin{proof}
First, let us see the case that $|\mathcal A^{(0)}| = |\mathcal A^{(n)}|$.
Since $0^A_{n_1}$ is the maximum element of $\mathcal A^{(0)}$,
we obtain the integer $m$ by $m = n$.
Next, suppose that $|\mathcal A^{(0)}| > |\mathcal A^{(n)}|$.
Then there exists the smallest non-negative integer $m'$
such that $|\mathcal A^{(m')}| = |\mathcal A^{(n)}|$,
and we have the maximum element $1^A_{m_1}$ of $\mathcal A^{(m')}$.
Note that $\pi^2(1^A_{m_1}) = 0^A_{n_1}$,
and we obtain the integer $m = n - m' - 2$.
\end{proof}

\end{comment}

We fix notation of the non-negative integer $d'$,
and we give a proof of Proposition~\ref{PropOfDivide} as follows.

\begin{proof}[Proof of Proposition~\ref{PropOfDivide}]
For the ABS $S^-$, 
let $\Psi$ and $S_\rho$ be components of $S^-$
given by (a) or (b) of Proposition~\ref{ThmOfCases},
where $S_\rho$ is the ABS associated to the ${\rm DM_1}$ $N_\rho$.
Let us see the case (1) of Proposition~\ref{PropOfDivide}.
Note that 
$S_\rho$ obtained by (a) of Proposition~\ref{ThmOfCases}
consists of all elements of $P$.
Recall that $T_{i_x}$ is the set which consists of elements $t$ 
satisfying that $0^A_{m_1+1} \leq t < 1^A_{m_1}$ in $S^-$.
To apply Proposition~\ref{PropOfSaturated},
we will show that if $x \leq d$,
then all elements of $T_{i_x}$ except for $0^A_{m_1+1}$ belong to $\tilde S_\rho$
for the small modification by $0^A_{m_1+1}$ and $1^A_{m_1}$.
By Proposition~\ref{PropOfPiandd} (v),
it suffices to show that each element of $Q$ except for $0^A_{m_1+1}$
does not belong to $T_{i_1}$.
In fact, this claim induces that 
all elements of $T_{i_1}$ belong to $P$ and $\tilde S_\rho$.
We put two sub-paths $\omega_1$ and $\omega_2$ of $Q$ as follows:
$$Q : \underbrace{0^A_{m_1+1} \rightarrow \cdots
\rightarrow 0^A_{i_1+m_1}}_{\omega_1} \rightarrow 
\underbrace{0^A_{i_1}
\rightarrow \cdots \rightarrow 0^A_{2m_1}}_{\omega_2}.$$
It follows from Proposition~\ref{PropOf1} that 
every element, which is of the form $\pi^m(0^A_{i_1})$ with $m < a_{i_1}$, 
of $\omega_2$ does not belong to $T_{i_1}$.
The property, which is given in Lemma~\ref{LemOfFirst}, 
of $i_1$ implies that 
all elements of $\omega_1$ except for $0^A_{m_1+1}$ do not belong to 
$T_{i_1}$ which is a subset of $\{0^A_{m_1+1}, \dots, 0^A_{n_1}\}$.

Next, let us see the case (2).
We consider the small modification by $0^A_{i_x-1}$ and $1^B_j$ in $S^-$.
By Proposition~\ref{PropOfSaturated}, it suffices to show that
if $x > d$, then there exists no element $t$ 
between $0^A_{i_x-1}$ and $1^B_j$ in $S^-$.
For the natural number $x$ with $x > d$,
by Proposition~\ref{PropOfd'd},
clearly $x > d'$ holds.
To lead a contradiction,
let us suppose that there exists an element between $0^A_{i_x-1}$ and $1^B_j$
in the ABS $S^-$.
By Proposition~\ref{PropOfIffi-1j}, 
there exists a non-negative integer $v$ such that
$0^A_{i_x-1}$ is the maximum element of 
$\pi(\mathcal A_{i_x}^{(v)}) =: T$.
For sets $\{\mathcal A_{i_d}^{(n)}\}_{n = 0, \dots, a-1}$ and $T_{i_d}$,
let $\Phi$ be the path consisting of maximum elements of these sets:
$$\Phi :0^A_{n_1} \rightarrow 
\cdots \rightarrow 0^A_{m_1+\gamma}.$$
We define a non-negative integer $m'$ to be  
$\mathcal A_{i_x}^{(m')} = \{1^A_{m_1-u+1}, \dots, 1^A_{m_1}\}$,
with $u = |T|$.
Since $m'$ is the minimum number satisfying that
$|\mathcal A_{i_x}^{(m')}| = |T|$,
we have $m' < v$.
Put $m = m'+2$.
Then $\mathcal A_{i_x}^{(m)}$ has the maximum element $0^A_{n_1}$.
We consider the path consisting of the maximum elements 
of sets $\mathcal A_{i_x}^{(m)}, \mathcal A_{i_x}^{(m+1)}, \dots, T$,
and we obtain the path $O$ from $0^A_{n_1}$ to $0^A_{i_x-1}$.
%$$O:0^A_{n_1} \rightarrow
%\cdots \rightarrow 0^A_{i_{d'}-1} \rightarrow 
%\cdots \rightarrow 0^A_{i_x-1}.$$
Here, let us show that $O$ can be regarded as a subset of $\Phi$.
It suffices to see that $O$
does not contain $0^A_{m_1+\gamma}$.
If $0^A_{m_1+\gamma}$ is contained in the path $O$,
then $|T_{i_x}| \leq \gamma$ follows from Proposition~\ref{PropOfNumbers}.
This contradicts with the hypothesis $x > d$ 
and Proposition~\ref{PropOfPiandd} (iv).
Hence $O$ is a subset of $\Phi$,
and it implies that for a non-negative integer $n$,
the set $\mathcal A_{i_d}^{(n)}$ contains the maximum element $0^A_{i_x-1}$
with $x > d'$.
This contradicts with the definition of $d'$. 
\end{proof}

\subsection{Proof of Proposition~\ref{ThmOfCases} (II)}
\label{PfOfLatterThmOfCases}

Here, let us see the remaining case $n_1 = m_1 + 1$
for $\xi = (m_1, n_1) + (m_2, n_2)$,
and we will give a proof of Proposition~\ref{ThmOfCases} (II).
The discussion of this proof 
is given by the same way as the proof of Proposition~\ref{ThmOfCases} (I).
As we have seen in Proposition~\ref{PropOfm_2},
for every good exchange of $0^A_i$ and $1^B_j$ with $1^B_j \in D$, 
we have $0^B_{m_2+1} < 1^B_{m_2}$ in $S^-$.
Assume $n_1 = m_1+1$
for the Newton polygon $\xi = (m_1, n_1) + (m_2, n_2)$.
In this case, we have $C' = \{0^A_{n_1}\}$.
Let us see the case $D \neq \emptyset$.

\begin{notation}
For the simple ${\rm DM_1}$ $N_{m_2, n_2}$ and its ABS $B$, 
we define sub-paths $U$ and $V$ of $B$ by the following:
%\begin{eqnarray*}
%U &=& \{\pi^n(1^B_{m_2}) \mid 0 \leq n \leq \eta_1, \text{ with }
%\pi^{\eta_1}(1^B_{m_2}) = 1^B_{m_2-n_2+1}\},\\
%V &=& \{\pi^n(0^B_{m_2+1}) \mid 0 \leq n \leq \eta_2, \text{ with }
%\pi^{\eta_2}(0^B_{m_2+1}) = 1^B_{m_2-n_2}\}. 
%\end{eqnarray*}
%We often identify the above sets $U$ and $V$ with the following paths:
\begin{eqnarray*}
U &:& 1^B_{m_2} \rightarrow 0^B_{h_2} \rightarrow 1^B_{n_2} \rightarrow
\cdots \rightarrow 1^B_{m_2-n_2+1},\\
V &:& 0^B_{m_2+1} \rightarrow 1^B_1 \rightarrow 1^B_{n_2+1} \rightarrow 
\cdots \rightarrow 1^B_{m_2-n_2}.
\end{eqnarray*}
We clearly have $U \sqcup V = B$ as sets.
\end{notation}
%% DefinitionからNotationに変更

The above components $U$ and $V$ of $B$ are useful.
Concretely, as can be seen in Lemma~\ref{LemOfDU},
all elements of $D$ belongs to the component $U$.
Moreover, in the case (c) of Proposition~\ref{ThmOfCases},
the ABS corresponding to $N_\rho$ consists of all elements of $V$.

\begin{lemma}\label{LemOfDU}
For the above notation, $D \subset U$ holds.
Set $\iota = |D|$.
Let $j_1, \dots, j_\iota$ be natural numbers 
such that $1^B_{j_x}$ is the element of $D$
appearing in the $x$-th in the path $U$:
$$1^B_{m_2} \rightarrow \cdots \rightarrow 1^B_{j_1}
\rightarrow \cdots \rightarrow 1^B_{j_2} \rightarrow 
\cdots \rightarrow 1^B_{j_\iota} 
\rightarrow \cdots \rightarrow 1^B_{m_2-n_2+1}.$$ 
We have then $j_1 = n_2$.
\end{lemma}

\begin{proof}
For the former part, we see that there exists no non-negative integer $n$ 
with $n \leq b$ such that
$\pi^n(1^B_j) = 1^B_{m_2}$ for every element $1^B_j$ of $D$
by the same reasoning as Proposition~\ref{PropOfm1+1m1}.
A proof is given by the same way as Lemma~\ref{LemOfCQ}.

Let us see the latter part.
By the former part and the hypothesis $1 \leq j \leq n_2$,
if $1^B_{n_2}$ belongs to $D$, then immediately we see $j_1 = n_2$.
Suppose that the set $\mathcal B^{(n)}$ contains $0^B_{h_2}$
which is the inverse image of $0^A_i$ in $S^{(n)}$
for a non-negative integer $n$.
We have then $0^B_{h_2} < \pi^n(1^B_j)$ in $S^{(a+n-1)}$.
Since $0^B_{h_2}$ is the maximum element of the ABS $S$,
this is a contradiction.
\end{proof}

\begin{definition}
Let $D_1$ (resp. $D_2$) be the subset of $D'$ consisting of $1^B_j$
satisfying that for a generic specialization $S^-$ obtained by 
$0^A_i \in C'$ and $1^B_j$,
we have the equality \eqref{EqOfInduction} by the case (c)
(resp. (d)) of Proposition~\ref{ThmOfCases}.
\end{definition}

We give a key element of $D$ in Proposition~\ref{PropOfje}
to show Proposition~\ref{ThmOfCases} (II).
For the above notation, this element characterize the sets $D_1$ and $D_2$
as seen in Proposition~\ref{PropOfD1D2}.
Sets given in Notation~\ref{NotationOfZ} are used for introducing 
the key element and describing the proof of Proposition~\ref{PropOfD1D2}.

\begin{notation} \label{NotationOfZ}
For an element $1^B_j$ of $D'$, 
we often write $\mathcal B_{j}^{(n)}$ for 
sets $\mathcal B^{(n)}$ to avoid confusion.
Moreover, we often write $b_j$ for
the smallest non-negative integer $b$ satisfying 
$\mathcal B_{j}^{(b)} = \emptyset$.
For sets $\{\mathcal B_{j}^{(n)}\}_{n = 0, \dots, b_j}$,
we define $$Z_j = \pi(\mathcal B_{j}^{(b-1)}).$$
This set consists of all elements $t$ satisfying
$0^B_{m_2+1} < t \leq 1^B_{m_2}$ in $S^-$.
\end{notation}

\begin{proposition} \label{PropOfje}
Put $j = 1+\mu$, with $\mu = |Z_{j_1}|$.
Then $1^B_j$ belongs to $D$.
Let $e$ be the natural number satisfying 
$j = j_e$.
Then $Z_{j_1} = Z_{j_e}$ holds.
\end{proposition}

\begin{proof}
By Lemma~\ref{LemOfDecrease},
there exists a natural number $n$ such that
$\mathcal B_{j_1}^{(n-1)} = \{0^B_{m_2+1}, \dots, 0^B_{m_2+\mu}\}$.
We have then $\mathcal B_{j_1}^{(n)} = \{1^B_1, \dots, 1^B_\mu\}$.
%and $\pi^n(1^B_{j_1}) = 1^B_j$.
This set is equal to $\mathcal B_{j}^{(0)}$.
A proof is given by the same way as Proposition~\ref{PropOfC0}.
\end{proof}

Proposition~\ref{PropOfD1D2} is shown 
in the same way as the proof of
Proposition~\ref{PropOfDivide}.
The proposition
is used for the proof of Proposition~\ref{ThmOfCases} (II).

\begin{proposition}\label{PropOfD1D2}
Let $j_x$ be an element of $D$.
We have then
\begin{itemize}
\item[(1)] If $x \leq e$, then $1^B_{j_x}$ belongs to $D_1$,
\item[(2)] If $x > e$, then $1^B_{j_x}$ belongs to $D_2$.
\end{itemize}
\end{proposition}

\begin{proof}
Let $S^-$ be the generic specialization obtained by the exchange of
$0^A_i$ and $1^B_j$ in $S$.
For the ABS $S^-$, 
by (c) or (d) of Proposition~\ref{ThmOfCases},
we obtain two components $\Psi$ and $S_\rho$,
where $S_\rho$ is associated with $N_{\rho}$ for $\rho = (f, g)$.
Since $S_\rho$ coincides with 
the component obtained from $B$ by
applying \cite[Lemma 5.6]{HarashitaCon} to 
the adjacent $1_{m_2}^B\ 0_{m_2+1}^B$ or $1^B_j\ 1^B_{j+1}$,
we have $g m_2 - f n_2 = 1$.
In the same way as Proposition~\ref{PropOfSaturated},
we have the property:
If there exists no element $t$ of $\Psi$ satisfying that
$0^B_{m_2+1} < t < 1^B_{m_2}$ (resp. $0^A_i < t < 1^B_{j+1}$) in $S^-$,
then $1^B_j$ belongs to $D_1$ (resp. $D_2$).

First, let us show the statement (1).
We have properties
\begin{itemize}
\item[(i)]If $x < y$, then $Z_{j_x} \subset Z_{j_y}$ holds;
\item[(ii)]For all $n$ with $n < b_{j_e}$, 
we have $|\mathcal B_{j_e}^{(n)}| = \mu$;
\item[(iii)]For all $j_x$ and all $n$ with $n < b_{j_x}$, 
we have $|\mathcal B_{j_x}^{(n)}| \geq \mu$;
\item[(iv)]$Z_{j_e} \subsetneq Z_{j_x}$ holds 
for all $x$ with $x >e$;
\item[(v)]$Z_{j_x} = Z_{j_e}$ is true if and only if 
$x \leq e$.
\end{itemize}
These properties are shown by the same way 
as the proof of Proposition~\ref{PropOfPiandd}.
By (v), 
it suffices to consider the case $x = 1$
to show the statement (1).
Note that $S_\rho$ obtained by
the small modification by $0^B_{m_2+1}$ and $1^B_{m_2}$
consists of all elements of $V$.
There exists no element $t$ of $U$ satisfying that 
$0^B_{m_2+1} < t < 1^B_{m_2}$ in $S^-$.
In fact, for the path $U$:
$$\underbrace{1^B_{m_2} \rightarrow 0^B_{h_2}} \rightarrow \underbrace {1^B_{j_1} 
\rightarrow \cdots \rightarrow 1^B_{m_2-n_2+1}},$$
clearly if $t = 1^B_{m_2}$ or $t = 0^B_{h_2}$, then $t$ does not satisfy 
$0^B_{m_2+1} < t < 1^B_{m_2}$ in $S^-$.
By Proposition~\ref{PropOf0},
every element of the latter sub-path,
which is of the form $\pi^m(1^B_j)$ with $m < b$, 
does not belong to $Z_{j_1}$.
It induces that 
all elements in between $0^B_{m_2+1}$ and $1^B_{m_2}$ in $S^-$
do not belong to $\Psi$.

Next, let us show the statement (2).
We define a non-negative integer $e'$ to be
the maximum number of the set
$\{x \in \bar D \mid 1^B_{j_x+1} 
\text{ is the maximum element of } \mathcal B_{j_e}^{(n)} \text{ for } n = n(x)\}$,
where $\bar D = \{x \in \mathbb N \mid 1 \leq x \leq \iota\}$.
If this set is empty,
then we define $e' = 0$.
By the same way as Proposition~\ref{PropOfd'd},
we have $e' \leq e$. 
For these notation, a proof is obtained by the same way as the proof of 
Proposition~\ref{PropOfDivide} (2).
\end{proof}

\begin{proof}[Proof of Proposition~\ref{ThmOfCases} (II)]
By Proposition~\ref{PropOfD1D2},
for every element of $D = D' \setminus \{1^B_1\}$,
we construct the equality \eqref{EqOfInduction} by (c) or (d).
Let us see the remaining case $j = 1$.
If $n_2 > 1$, 
then we have $0^A_i < 1^B_{j+1}$
and there exists no element $t$ in between $0^A_i$ and $1^B_{j+1}$.
Hence we obtain the equality \eqref{EqOfInduction} by (d). 
Suppose $n_2 = 1$.
In this case, we construct small modification by 
elements $0^A_{h_1}$ and $1^B_{n_2+1}$ in $S^-$,
and we obtain two components $\Psi$ and $S_{\rho}$ of $S^{--}$,
where $\rho = (1,1)$.
Concretely, we have $S_\rho = 1^B_1\ 0^A_{h_1}$.
Let $N$ be the ${\rm DM_1}$ associated with $\Psi$.
We have then $N = N_{\xi'}^-$ with $\xi' = (m_1-1, n_1-1) + (m_2, n_2)$.
Hence we obtain the equality \eqref{EqOfInduction} by (e).
\end{proof}

\subsection{Proofs of Proposition~\ref{PropOfSpe} and Theorem~\ref{ThmOfzetaxi}}
\label{SecOf5.11.3}
%% subsection追加

In this section, we show Proposition~\ref{PropOfSpe}.
Theorem~\ref{ThmOfzetaxi} follows from this proposition.

\begin{proof}[Proof of Proposition~\ref{PropOfSpe}]
By Proposition~\ref{ThmOfCases},
it remains to show the case
$\lambda_1 = 1$ or $\lambda_2 = 0$.
If $\lambda_1 = 1$, then
for $S^-$ obtained by a good exchange of $0^A_1$ and $1^B_j$,
we get ABS's corresponding to $N_{\xi'}^-$ and $N_\rho$ by
(c) or (d) of Proposition~\ref{ThmOfCases}.
If $\lambda_2 = 0$, then
for $S^-$ obtained by a good exchange $0^A_i$ and $1^B_1$,
we get ABS's corresponding to $N_{\xi'}^-$ and $N_\rho$ by
(a) or (b) of Proposition~\ref{ThmOfCases}.
\end{proof}

%% m_1 = 0 or n_2 = 0の場合を追加．

Finally, we prove Theorem~\ref{ThmOfzetaxi}.
\begin{proof}[Proof of Theorem~\ref{ThmOfzetaxi}]
The assertion is paraphrased as follows:
For any generic specialization $N_\xi^-$ of
the ${\rm DM_1}$ $N_\xi$ with $\xi = (m_1, n_1) + (m_2, n_2)$,
there exists a Newton polygon $\zeta$ such that
$N_\zeta$ appears as a specialization of $N_\xi^-$,
and $\zeta \prec \xi$ is saturated.
We show this by induction on height of $\xi$.
%We give a proof of the assertion by induction on $\xi$

If the height of $\xi$ is two
(the case that the height is minimal), 
then $\xi = (0, 1) + (1, 0)$. 
In this case $N_\xi^- = N_{1, 1}$ holds,
whence there is nothing to prove.
%Let us see the case of $m_1 = 0$ and $n_2 = 1$.
%Then the ${\rm DM_1}$ $N = N_{(0, 1)+(m_2, 1)}$ has one generic specialization $N^-$
%which is obtained by the small modification by $0^A_1$ and $1^B_1$
%in the ABS associated with $N$.
%In fact, for $j > 1$, if we construct the small modification by $0^A_1$ and $1^B_j$,
%then the set $\mathcal B^{(0)}$ contains the inverse image $1^B_{j-1}$ of $0^A_1$.
%By Theorem~\ref{ThmOfCentralClassification},
%this small modification does not construct a boundary component.
%Let us consider the ABS obtained by the small modification by $0^A_1$ and $1^B_2$ 
%in the ABS associated to $N^-$.
%This ABS consists of two components.
%The component containing $0^A_1$ is described as $1^B_1\ 0^A_1$.
%This ABS corresponds to $N_{1, 1}$.
%Moreover, the other component is associated to $N_{m_2-1, 1}$.
%Hence for the case $m_1 = 0$ and $n_2 = 1$, 
%we have $N_{\xi}^{--} = N_{1, 1} \oplus N_{m_2-1, 1}$.
%By the duality, for $\xi = (1, n_1) + (1, 0)$,
%we have $N_\xi^{--} = N_{1, n_1-1} \oplus N_{1, 1}$.

Assume that the height of $\xi$ is greater than two.
By Proposition~\ref{PropOfSpe}, we obtain Newton polygons 
$\xi'$ and $\rho$ such that $N_\xi^{--} = N_{\xi'}^- \oplus N_\rho$,
where the area of the region surrounded by $\xi$, $\xi'$ and $\rho$ is one.
By the hypothesis of induction, there exists a Newton polygon $\zeta'$ such that
$\zeta' \prec \xi'$ is saturated
and $N_{\zeta'}$ is a specialization of $N^-_{\xi'}$.
If we put $\zeta = \zeta' + \rho$, then $\zeta \prec \xi$ is saturated, and
$N_\zeta$ is a specialization of $N_{\xi'}^- \oplus N_\rho (= N_\xi^{--})$,
and therefore is a specialization of $N_\xi$
(cf. \cite[Proposition~3.5]{HiguchiHarashita}).
%Concretely, by the case (a) or (b) of Proposition~\ref{ThmOfCases},
%the induction arrives at the case $n_1 = m_1 + 1$.
%By the case (c) or (d), 
%it suffices to consider the case $n_2 = 1$ or $0$.
%If $n_2 = 1$ (resp. $n_2 = 0$), then by the case (e) (resp. by the case (b)),
%the induction arrives at the case $\xi = (0, 1) + (m_2, 1)$ (resp. $\xi = (0, 1) + (1, 0)$).
\end{proof}
%一般の\xiのついて，Prop 5.2の(a), (b)からn_1 = m_1 + 1の場合に帰着．
%その後(c), (d)からm_2 = n_2 + 1またはn_2 = 1の場合に帰着．
%前者の場合は\rho = (1, 1)を得ながらn_2 = 1の場合に帰着する．
%Prop 5.2の(e)によって(m_1, n_1) = (0, 1)の場合に帰着する．
%% 無駄を省略

\subsection*{Acknowledgments}
I thank the referee for careful reading and helpful comments.
This paper is written when the author is a Ph.D student.
I thank the supervisor professor Harashita 
for the constant support from the early stage of this paper.

\bigskip
\address{ % First Author
Graduate School of Environment and \\
Information Sciences, \\
Yokohama National University, \\
79-1 Tokiwadai, Hodogaya-ku, \\
Yokohama 240-8501 \\
Japan
}
{higuchi-nobuhiro-sy@ynu.jp}

\end{document}